\pdfoutput=1
\documentclass[11pt,draftcls,onecolumn]{IEEEtran}
\usepackage[margin=1in]{geometry}             
\usepackage{graphicx}
\usepackage{amssymb}
\usepackage{amsmath}
\usepackage{amsfonts}
\usepackage{epstopdf}
\usepackage{dblfloatfix}
\usepackage{color}
\usepackage{array}
\usepackage{multirow}
\usepackage{booktabs}
\usepackage{balance}
\usepackage{setspace}
\newtheorem{theorem}{\textbf{Theorem}}

\newtheorem{example}{\textbf{Example}}
\newtheorem{remark}{\textbf{Remark}}
\newtheorem{lemma}{\textbf{Lemma}}
\newtheorem{corollary}{\textbf{Corollary}}

\newcommand{\norm}[1]{\lVert#1\rVert}
\newcommand{\Rbb}{\mathbb{R}}
\renewcommand{\l}{\ell}
\newcommand{\Identity}{1\!\!1}
\newcommand{\G}{{\mathcal{G}}}

\newcommand\ip[2]{\langle #1, #2\rangle}
\renewcommand{\L}{{\mathcal{L}}}

\begin{document}

\title{
Spectrum-Adapted Tight Graph Wavelet and Vertex-Frequency Frames  
}
\author{\IEEEauthorblockN{David I Shuman, Christoph Wiesmeyr, Nicki Holighaus, and Pierre Vandergheynst}
\thanks{This work was supported by FET-Open grant number 255931 UNLocX.}
\thanks{David I Shuman and Pierre Vandergheynst are with the Signal Processing Laboratory (LTS2), Ecole Polytechnique F{\'e}d{\'e}rale de Lausanne (EPFL), Institue of Electrical Engineering, CH-1015 Lausanne, Switzerland (email: david.shuman@epfl.ch, pierre.vandergheynst@epfl.ch). Christoph Wiesmeyr is with the Numerical Harmonic Analysis Group, Faculty of Mathematics, University of Vienna, Oskar-Morgenstern-Platz 1, 1090 Wien, Austria (email: christoph.wiesmeyr@univie.ac.at). Nicki Holighaus is with the Acoustics Research Institute, Austrian Academy of Sciences, Wohllebengasse 12Ð14, 1040 Vienna, Austria (email: nicki.holighaus@kfs.oeaw.ac.at).}
\thanks{The first author would like to thank Daniel Kressner for helpful discussions about methods to approximate the cumulative spectral density function.}  
}
\maketitle

\begin{abstract}
We consider the problem of designing spectral graph filters for the construction of dictionaries of atoms that can be used to efficiently represent signals residing on weighted graphs. While the filters used in previous spectral graph wavelet constructions are only adapted to the length of the spectrum, the filters proposed in this paper are adapted to the distribution of graph Laplacian eigenvalues, and therefore lead to atoms with better discriminatory power. Our approach is to first characterize a family of systems of uniformly translated kernels in the graph spectral domain that give rise to tight frames of atoms generated via generalized translation on the graph. We then warp the uniform translates with a function that approximates the cumulative spectral density function of the graph Laplacian eigenvalues. We use this approach to construct computationally efficient, spectrum-adapted, tight vertex-frequency and graph wavelet frames. We give numerous examples of the resulting spectrum-adapted graph filters, and also present an illustrative example of vertex-frequency analysis using the proposed construction.   
\end{abstract}

\section{Introduction}
One of the main focuses of the emerging field of graph signal processing (see \cite{shuman_SPM} for a recent overview) is to develop transforms that enable us to efficiently extract information from high-dimensional data residing on the vertices of weighted graphs. In particular, researchers are designing dictionaries of atoms adapted to the underlying graph data domain, and representing graph signals as linear combinations of those atoms in various signal processing tasks. 
To date, the major thrust has been to design dictionaries whose atoms are jointly localized in the vertex domain (the analogue of the time domain for signals on the real line) and the graph spectral domain (the analogue of the frequency domain). For example, classical wavelet and time-frequency dictionary designs have been generalized to the graph setting in different manners (see, e.g., \cite{Crovella2003, diffusion_wavelets, sgwt} for graph wavelet constructions and \cite{shuman_SSP_2012,shuman_ACHA_2013} for a generalization of windowed Fourier frames). 

The dictionaries we consider in this paper are comprised of atoms constructed by translating smooth graph spectral filters to be centered at different vertices of the graph (a more precise mathematical definition of this class of dictionaries is included in Section \ref{Se:notation}). The translation is accomplished by multiplying each filter by a graph Laplacian eigenvector in the graph Fourier domain, and the smoothness of the graph spectral filters ensures that the atoms are localized around their center vertices. The first example of such a dictionary is 
the spectral graph wavelet frame of \cite{sgwt}, where the system of graph spectral filters consists of a single lowpass kernel and a sequence of dilated bandpass kernels. In \cite{leonardi_fmri,leonardi_multislice}, Leonardi and Van De Ville introduce Meyer-like wavelet and scaling kernels that lead to tight spectral graph wavelet frames. 
As outlined in \cite[Section 6.7.2]{shuman_ACHA_2013}, 
benefits of a tight frame include increased numerical stability when reconstructing a signal from noisy coefficients \cite{christensen,kovacevic_frames1,kovacevic_frames2}, faster computations (e.g., when computing proximity operators in convex regularization problems \cite{combettes_chapter}), and the ability to interpret the spectrogram of a generalized time-frequency frame as an energy density function.

While the structure of the graph under consideration is incorporated into the spectral graph wavelets of \cite{sgwt} via the graph Laplacian, the spectral graph wavelet and scaling kernels are only adapted to the maximum graph Laplacian eigenvalue, and not to the specific graph Laplacian spectrum.\footnote{The tight graph wavelet kernels of \cite{leonardi_multislice} are also adapted to the maximum degree of the graph. We discuss the relation between those kernels and the proposed kernels in more detail in Section \ref{Se:adapted}, Remark \ref{Re:tight_comparison}.} As a result, for graphs with irregularly spaced graph Laplacian eigenvalues, many spectral graph wavelets may be highly correlated with the wavelets centered at nearby vertices and scales, and therefore, their coefficients may not provide as much discriminatory power when analyzing graph signals. We provide more detailed examples of this phenomenon in Sections \ref{Se:adapted} and \ref{Se:wavelets}. 

On the other hand, the windowed graph Fourier frames of \cite{shuman_SSP_2012,shuman_ACHA_2013} feature atoms that are adapted to the specific discrete graph Laplacian spectrum via a generalized modulation; however, the extra adaptation comes at the computational expense of having to compute a full eigendecomposition of the graph Laplacian, which does not scale well with the number of vertices and edges in the graph. A Chebyshev polynomial approximation method \cite[Section 6]{sgwt}, \cite{shuman_DCOSS_2011} enables the spectral graph wavelets (and other dictionaries belonging to the family considered in this paper) to be implemented without performing this full eigendecomposition.

References \cite{Zhang2012} and \cite{thanou_GlobalSIP_2013} use sets of graph signals as training data to learn 
dictionaries of atoms that are also constructed by translating smooth graph kernels to different vertices. 
Since the graph structure is incorporated into the learning process, the learned kernels are indirectly adapted to the underlying graph data domain (and its Laplacian spectrum) as well as the training data. 

In this paper, we propose a new method to design the sequence of spectral graph kernels so that (i) they are adapted to the entire graph Laplacian spectrum of the graph under consideration (but not to training data as in \cite{Zhang2012,thanou_GlobalSIP_2013}), and (ii) the dictionary resulting from translations of these kernels to all vertices on the graph is both a tight frame and has a fast implementation via the Chebyshev polynomial approximation method of \cite[Section 6]{sgwt}. Our main idea is to construct the spectral graph kernels as warped versions of uniformly translated kernels in the graph spectral domain, with the warping function approximating the cumulative spectral density function of the graph Laplacian in order to adapt the kernels to the entire spectrum. We use this approach to generate new vertex-frequency frames (generalizations of time-frequency analysis), as well as new graph wavelet frames. 

In addition to the primary contribution of adapting the graph spectral filters to the spectrum of the specific graph under consideration when generating dictionaries to represent graph signals, the paper contains two additional contributions that may find application outside the context of graph signal processing: (i) a new method to generate uniform translates of smooth windows whose squares sum to a constant function over either the entire real line or a finite interval (Section \ref{Se:uniform_translates}); and (ii) an implementation of a method to approximate the empirical spectral cumulative distribution of a large, sparse matrix based on classical spectrum slicing theory (Section \ref{Se:approx_spectrum}).

\section{Notation and Background} \label{Se:notation}

We generally follow the notation from \cite{shuman_SPM}, with the combinatorial and normalized graph Laplacians denoted by $\L$ and $\tilde{\L}$, respectively, and their eigenvalue and eigenvector pairs denoted by $\{(\lambda_{\l},u_{\l})\}_{\l=0,1,\ldots,N-1}$ and $\{(\tilde{\lambda}_{\l},\tilde{u}_{\l})\}_{\l=0,1,\ldots,N-1}$, where $N$ is the number of vertices in the graph. We denote the entire discrete graph Laplacian spectrum $\{\lambda_0=0, \lambda_1, \ldots, \lambda_{N-1}=\lambda_{\max}\}$ by $\sigma({\L})$. Given a graph signal $f_{in} \in \Rbb^N$ and a graph spectral filter (which we also refer to as a kernel) $\hat{g}: \sigma(\L) \rightarrow \Rbb$, graph spectral filtering is defined as multiplication in the graph Fourier domain:
\begin{align}\label{Eq:filtering}
\widehat{f_{out}}(\lambda_{\l}):=\widehat{f_{in}}(\lambda_{\l})\hat{g}(\lambda_{\l}),
\end{align}
where the graph Fourier transform is $\widehat{f_{in}}(\lambda_{\l}):=\ip{f_{in}}{u_{\l}}$.\footnote{Alternatively, the normalized graph Laplacian eigenvectors can be used as the Fourier basis, with $\tilde{\lambda}_{\l}$ replacing $\lambda_{\l}$ in \eqref{Eq:filtering}. We use the normalized graph Laplacian eigenvectors as the graph Fourier basis in Section \ref{Se:norm_ER}.} Note that the graph spectral filter $\hat{g}(\cdot)$ is often defined as a continuous function on the nonnegative real line and then restricted to $\sigma(\L)$. In the vertex domain (i.e., taking an inverse graph Fourier transform of \eqref{Eq:filtering}), graph spectral filtering reads as a generalized convolution \cite{shuman_SSP_2012}:
\begin{align*}
f_{out}=f_{in} \ast g = \sum_{\l=0}^{N-1} \widehat{f_{in}}(\lambda_{\l})\hat{g}(\lambda_{\l})u_{\l}.
\end{align*}

The dictionaries we consider in this paper are completely characterized by a sequence of graph spectral filters $\left\{\widehat{g_m}(\cdot)\right\}_{m=1,2,\ldots,M}$, and consist of all $M\cdot N$ atoms of the form
\begin{align}\label{Eq:atom_form}
g_{i,m}:=T_i g_m
=\sqrt{N}\delta_i \ast g_m = \sqrt{N} \widehat{g_m}(\L)\delta_i= \sqrt{N}\sum_{\l=0}^{N-1} \widehat{g_m}(\lambda_{\l})u_{\l}^*(i)u_{\l}.
\end{align}
In \eqref{Eq:atom_form}, $\delta_i$ is the Kronecker delta, 
and $T_i$ is a generalized translation operator that localizes each atom $g_{i,m}$ around its center vertex $i$. The spread of the atom $g_{i,m}$ around its center vertex $i$ is controlled by the smoothness of the filter $\widehat{g_m}(\cdot)$ \cite{sgwt,shuman_ACHA_2013}. The spectral graph wavelets \cite{sgwt} satisfy \eqref{Eq:atom_form}, with each bandpass kernel given by $\widehat{g_m}(\lambda_{\l})=\widehat{g}(t_m \lambda_{\l})$ for a fixed mother kernel $\widehat{g}(\cdot)$ and different dilation factors $t_m$ that are only adapted to the length of the spectrum. 

In the remainder of the paper, we suggest different methods to choose the sequence of graph spectral filters $\left\{\widehat{g_m}(\cdot)\right\}_{m=1,2,\ldots,M}$. We want the resulting dictionaries to form tight frames, so before proceeding, we present a sufficient condition on the spectral graph filters to ensure this property.
A proof of the following lemma is included in the Appendix.
\begin{lemma}[Slight generalization of Theorem 5.6 of \cite{sgwt}] \label{Le:frame}
Let ${\cal D}:=\left\{g_{i,m}\right\}_{i=1,2,\ldots,N;~m=1,2,\ldots,M}$ be a dictionary of atoms with $g_{i,m}:=T_i g_m$, and define
\begin{align*}
G(\lambda):=\sum_{m=1}^{M} \bigl[\hat{g}_m(\lambda)\bigr]^2.
\end{align*}
If $G(\lambda)>0$ for all $\lambda \in \sigma(\L)$, 
then  
for all $f \in \Rbb^N$,
\begin{align*}
A \norm{f}_2^2 \leq \sum_{i=1}^N \sum_{m=1}^{M} \left|\ip{f}{g_{i,m}}\right|^2 \leq B \norm{f}_2^2,
\end{align*}
\vspace{-1cm}

\noindent where 
\vspace{-.3cm}
\begin{align*}
A=N\cdot \min_{\lambda \in \sigma(\L)} G(\lambda)\hbox{~~and~~}B=N\cdot \max_{\lambda \in \sigma(\L)} G(\lambda).
\end{align*}
In particular, if $G(\lambda)$ is constant on $\sigma(\L)$, 
then ${\cal D}$ is a tight frame. 
\end{lemma}

\section{Uniform Translates} \label{Se:uniform_translates}
Our objective in this section is to develop a method to generate a system of filters such that (i) the filters are translated versions of each other in the graph spectral domain, and (ii) the $M\cdot N$ dictionary atoms constructed by applying each generalized translation operator $T_i$ to each filter form a tight frame.
More precisely, given an upper bound, $\lambda_{\max}$, on the spectrum and a desired number of filters, $M$, we want to find a kernel $\widehat{g^{U}}(\cdot)$ and constants 
 $a$ and $A$ such that
\begin{align}\label{Eq:kernel_criterion}
G(\lambda)=\sum_{m=1}^{M} \bigl[\widehat{g^U}(\lambda-ma)\bigr]^2=A,~~\forall \lambda \in [0,\lambda_{\max}].
\end{align} 
The following theorem and corollaries, proofs of which are included in the appendix, show one method to construct a parametrized family of kernels satisfying \eqref{Eq:kernel_criterion}. 
\begin{theorem}   \label{Th:uniform_translates}
  Let $K\in {\mathbb N}$ and $a_k \in \Rbb$ for $k \in \{0,1,\ldots,K\}$, and define
  \begin{align}\label{Eq:window_form}
   q(t) := \sum_{k=0}^K a_k \cos\Biggl(2\pi k \Bigl(t-\frac{1}{2}\Bigr)\Biggr) \Identity_{\{0 \leq t < 1\}}.
  \end{align}
  Then for any $R \in {\mathbb N}$ satisfying $R > 2K$,
  \begin{equation*}
    \sum_{m\in {\mathbb Z}} \left|q\left(t-\frac{m}{R}\right)\right|^2 = R a_0^2 + \frac{R}{2} \sum_{k=1}^K a_k^2, ~~\forall t \in \Rbb;
  \end{equation*}
  i.e. the squares of a system of regular translates sum 
  up to a constant function.
\end{theorem}
\begin{corollary}\label{Co:undilated_translates}
Given a desired number of filters, $M$, 
let $R$ and $K$ be any integers satisfying $2<R\leq M$ and $K<\frac{R}{2}$, and 
define the 
kernel
\begin{align*}
\hat{h}(y):=
\sum_{k=0}^K a_k \cos\left(2\pi k \left(y-\frac{1}{2}\right)\right) \Identity_{\left\{0 \leq y < 1 \right\}},
\end{align*}
where $\left\{a_k\right\}_{k=0,1,\ldots,K}$ is a real sequence of coefficients satisfying
\begin{align}\label{Eq:a_constraint}
\sum_{k=0}^K (-1)^k a_k = 0.
\end{align}
Then
\begin{align}\label{Eq:co1_statement}
H(y)=\sum_{m=1-R}^{M-R}\left[\hat{h}\left(y-\frac{m}{R}\right)\right]^2=
Ra_0^2+\frac{R}{2}\sum_{k=1}^K a_k^2,
~~\forall y\in \left[0,\frac{M+1-R}{R}\right].
\end{align}
\end{corollary}
Note that condition \eqref{Eq:a_constraint} is equivalent to requiring the kernel $\hat{h}(\cdot)$ to be continuous.
In the examples in this paper, we always take the kernel to be a shifted Hann kernel, with $K=1$ and $a_0=a_1=\frac{1}{2}$, in which case the right-hand side of \eqref{Eq:co1_statement} is equal to $\frac{3R}{8}$, where $R$ is a parameter controlling the overlap of the shifted kernels. The Blackman window also satisfies \eqref{Eq:window_form} and \eqref{Eq:a_constraint}, with $K=2$, $a_0=0.42$, $a_1=0.5$, and $a_2=0.08$.
\begin{example} \label{Ex:uniform}
In Figure \ref{Fig:hann}, we show three different sets of translated Hann kernels, for 
different 
$R$ and $M$.
\begin{figure}[h]
\centering
\begin{minipage}[b]{.32\linewidth}
\centerline{$~~~~R=3,~M=3$}
\centerline{\includegraphics[width=\linewidth]{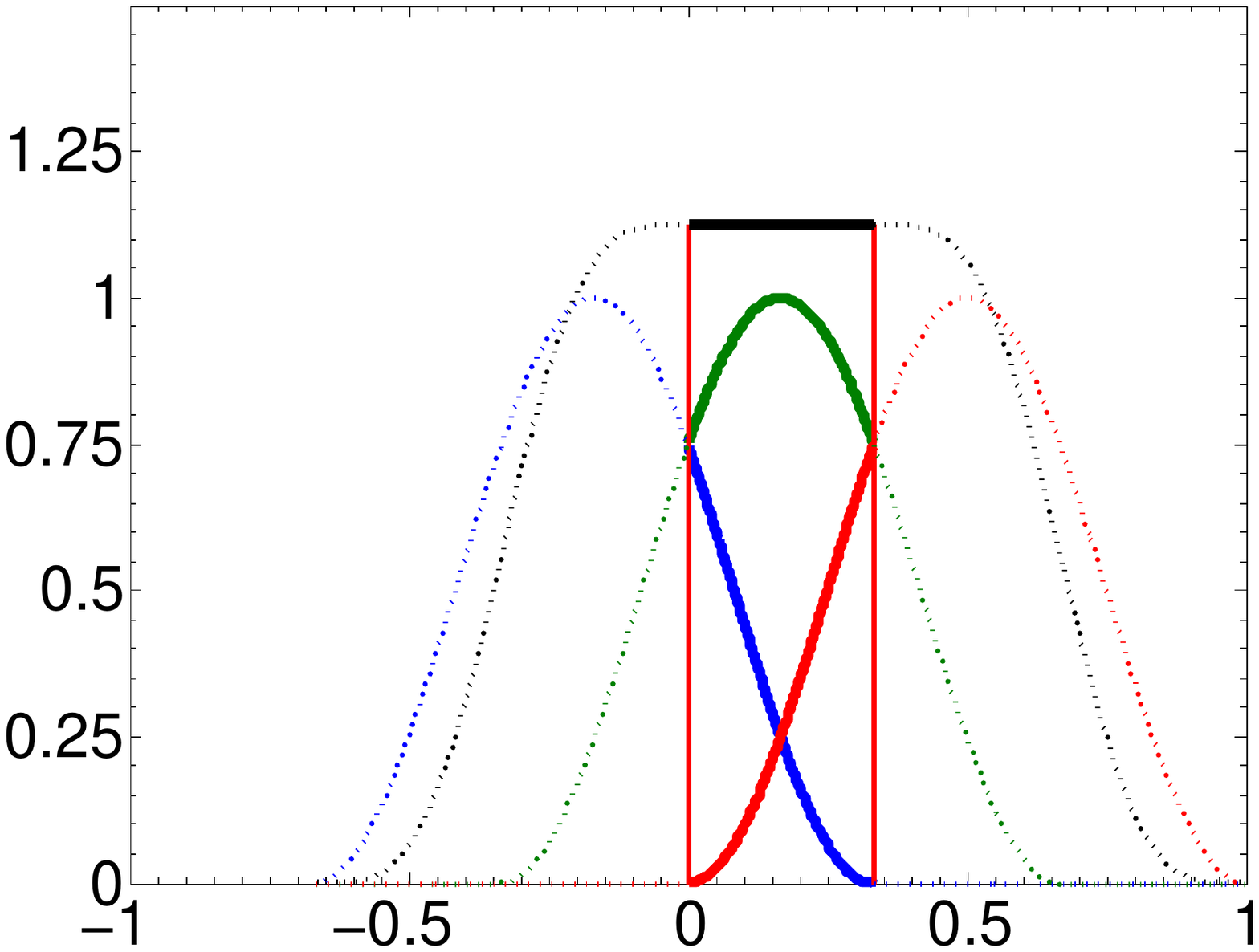}} 
\centerline{\small{~~~~$y$}}
\centerline{\small{~~~~(a)}}
\end{minipage}
\hfill
\begin{minipage}[b]{.32\linewidth}
\centerline{$~~~~R=3,~M=9$}
\centerline{\includegraphics[width=\linewidth]{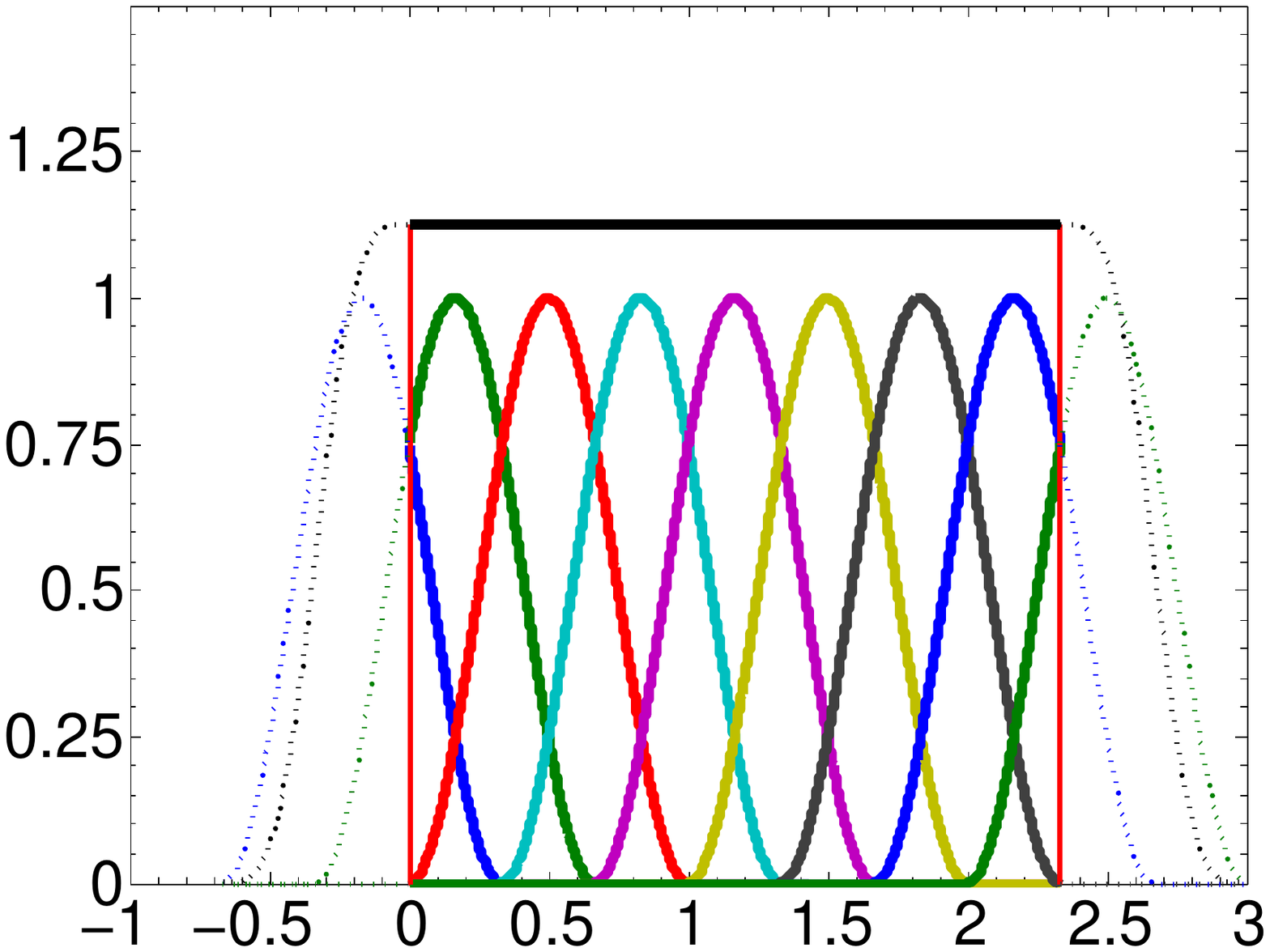}} 
\centerline{\small{~~~~$y$}}
\centerline{\small{~~~~(b)}}
\end{minipage}
\hfill
\begin{minipage}[b]{.32\linewidth}
\centerline{$~~~~R=5,~M=9$}
\centerline{\includegraphics[width=\linewidth]{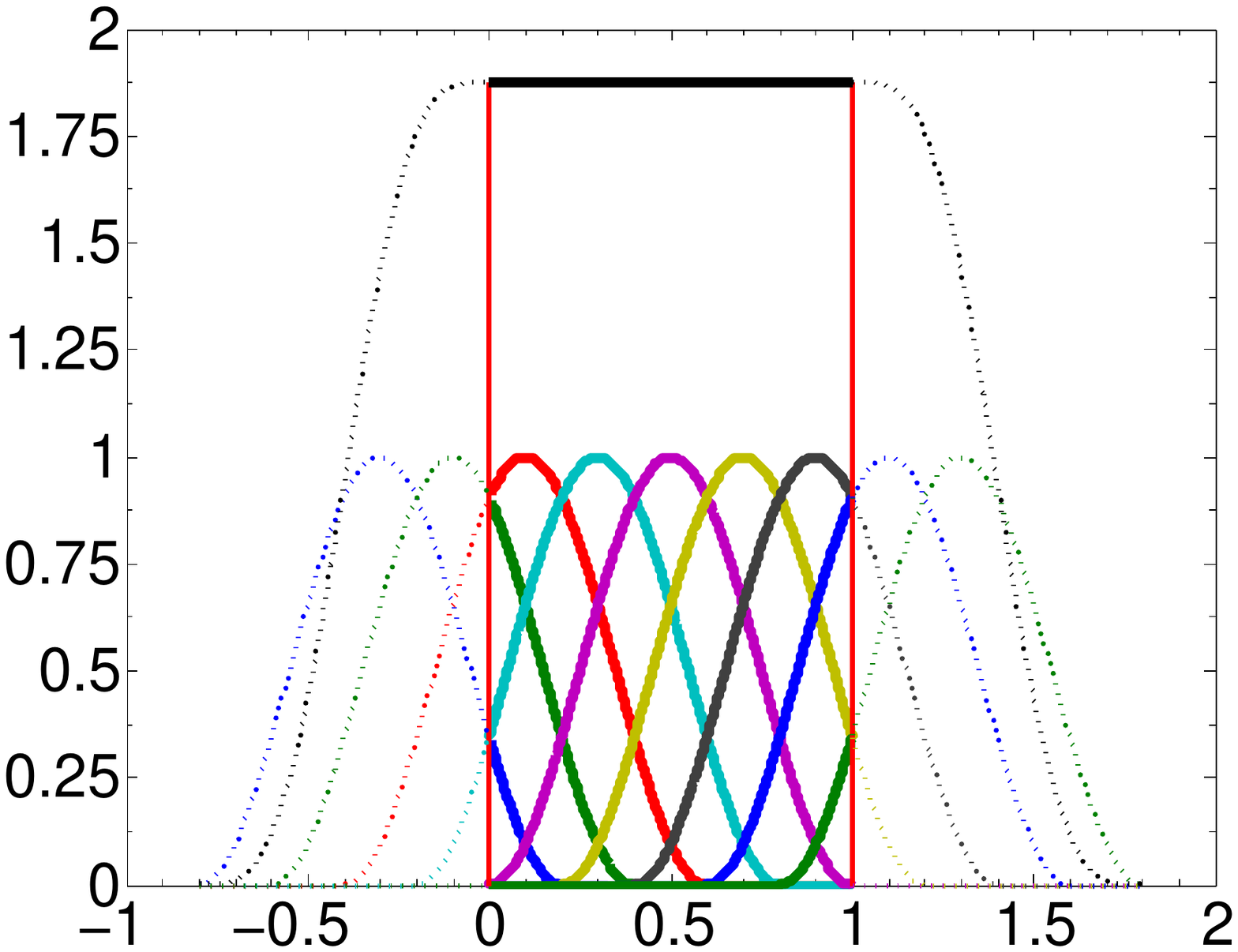}} 
\centerline{\small{~~~~$y$}}
\centerline{\small{~~~~(c)}}
\end{minipage} 
\caption {Translated versions of the shifted Hann kernel, as described in Corollary \ref{Co:undilated_translates}. 
The vertical red lines show the bounds $\left[0,\frac{M+1-R}{R}\right]$, and the top black lines show $H(y)$. } 
 \label{Fig:hann}
\end{figure}
\end{example}
\begin{corollary}\label{Co:uniform}
Given an upper bound, $\gamma$,
on the spectrum and a desired number of filters, $M$, let $R$ and $K$ be any integers satisfying $2<R\leq M$ and $K<\frac{R}{2}$, and for a real sequence $\left\{a_k\right\}_{k=0,1,\ldots,K}$ satisfying \eqref{Eq:a_constraint},
 define 
\begin{align}\label{Eq:gu_kernel}
\widehat{g^U}(\lambda):=
\sum_{k=0}^K a_k \cos\left(2\pi k \left(\frac{M+1-R}{R \gamma}\lambda+\frac{1}{2}\right)\right) 
\Identity_{\left\{-\frac{R \gamma}{M+1-R}\leq \lambda < 0\right\}}.
\end{align}
Then 
\begin{align*}
G(\lambda)=\sum_{m=1}^{M} \left[\widehat{g^U_m}(\lambda)\right]^2=
Ra_0^2+\frac{R}{2}\sum_{k=1}^K a_k^2,
~~\forall \lambda \in [0,\gamma], 
\end{align*} 
where 
\begin{align}\label{Eq:uniform_construction}
\widehat{g^U_m}(\lambda):=\widehat{g^U}\left(\lambda-m\frac{\gamma}{M+1-R}\right).
\end{align}
\end{corollary}
\begin{proof}
Take $\widehat{g^U}(\lambda)=\hat{h}\left(\Bigl[\frac{\lambda}{\gamma}\Bigr]\Bigl[\frac{M+1-R}{R}\Bigr]+1\right)$, with $\hat{h}(\cdot)$ from Corollary \ref{Co:undilated_translates}, and then perform a change of variable $m^{\prime}=m-R$. 
\end{proof}
\addtocounter{example}{-1}
\begin{example}[cont.] In Figure \ref{Fig:hann2}, we take $\gamma=\lambda_{\max}=12$ and stretch the filters of Figure \ref{Fig:hann} to fit the spectrum $[0,\lambda_{\max}]$, for each of the three different pairs of $R$ and $M$. In each case, $G(\lambda)$ is a constant, and therefore, by Lemma \ref{Le:frame}, the dictionary $\left\{T_i g^U_m\right\}_{i=1,2,\ldots,N;~m=1,2,\ldots,M}$ is a tight frame. 
\begin{figure}[h]
\centering
\begin{minipage}[b]{.32\linewidth}
\centerline{$~~~~R=3,~M=3$}
\centerline{\includegraphics[width=\linewidth]{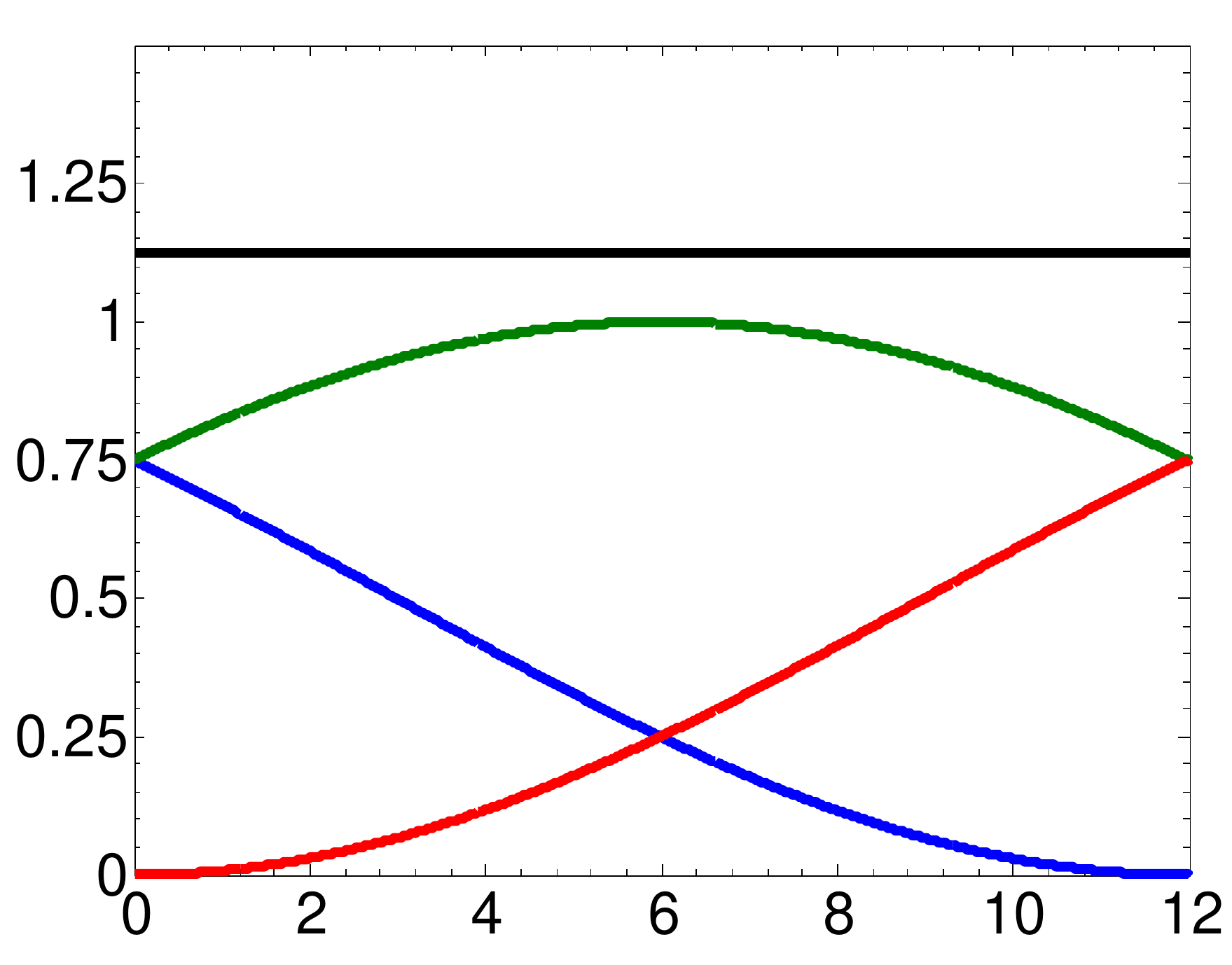}} 
\centerline{\small{~~~~$\lambda$}}
\centerline{\small{~~~~(a)}}
\end{minipage}
\hfill
\begin{minipage}[b]{.32\linewidth}
\centerline{$~~~~R=3,~M=9$}
\centerline{\includegraphics[width=\linewidth]{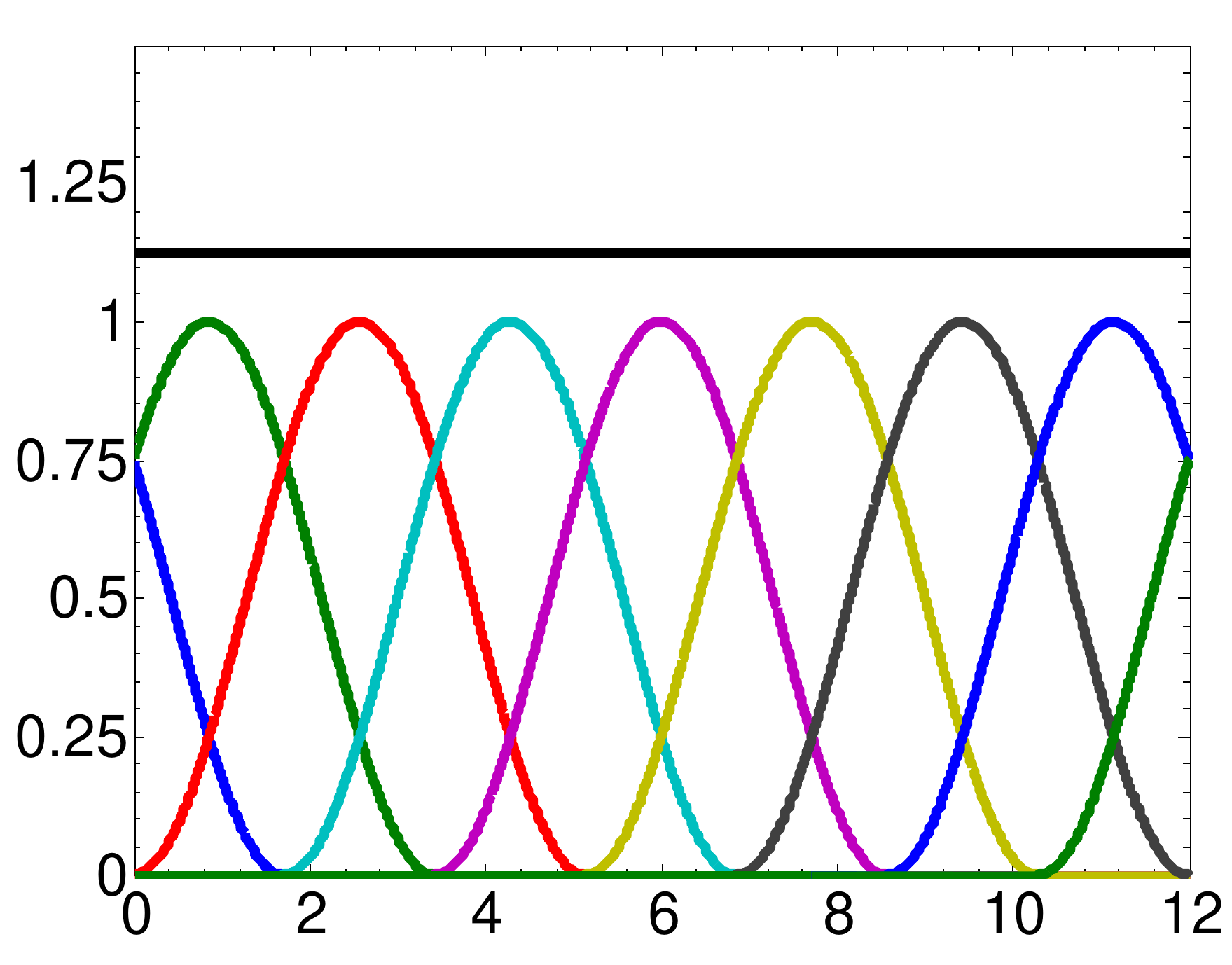}} 
\centerline{\small{~~~~$\lambda$}}
\centerline{\small{~~~~(b)}}
\end{minipage}
\hfill
\begin{minipage}[b]{.32\linewidth}
\centerline{$~~~~R=5,~M=9$}
\centerline{\includegraphics[width=\linewidth]{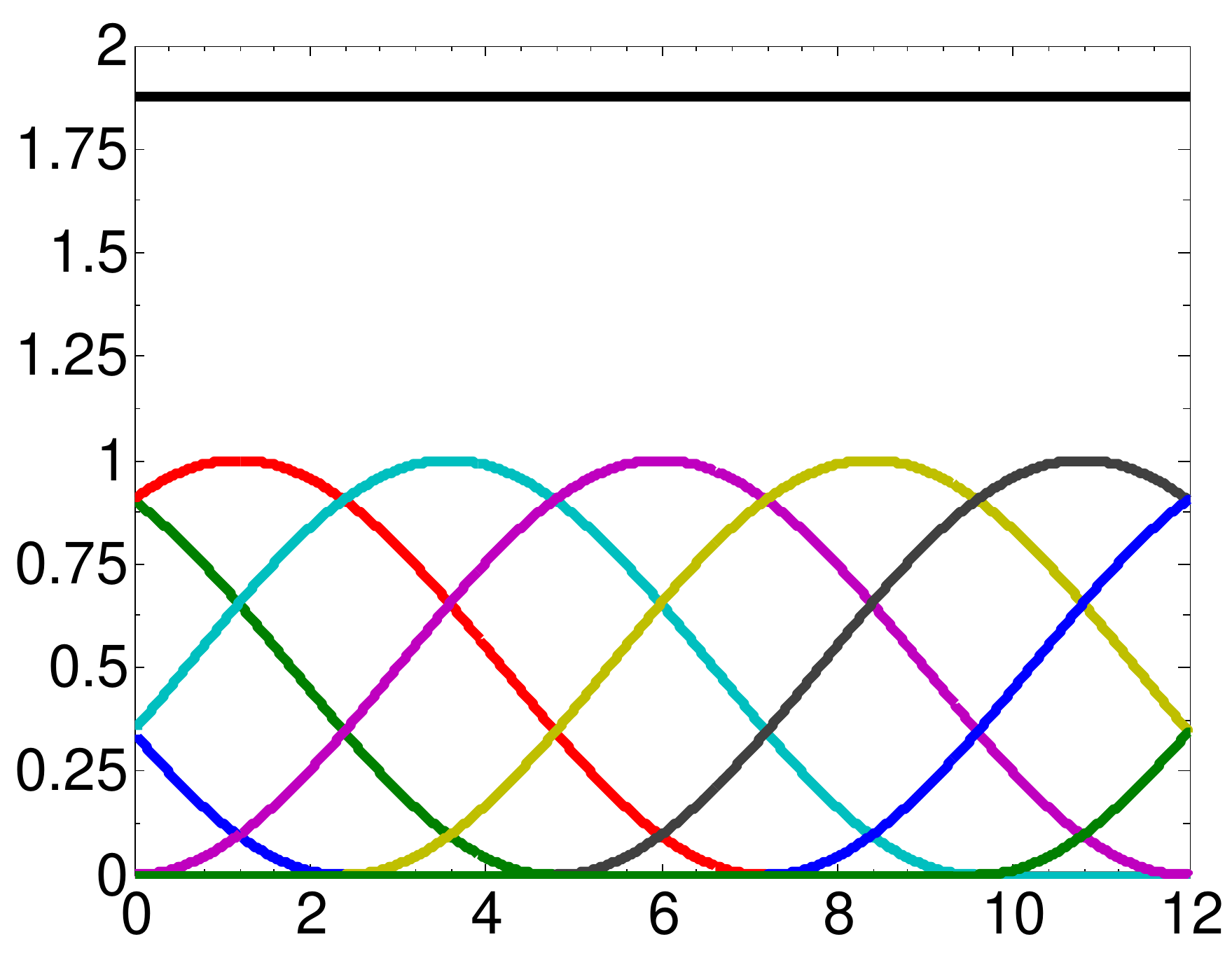}} 
\centerline{\small{~~~~$\lambda$}}
\centerline{\small{~~~~(c)}}
\end{minipage} 
\caption {Spectral graph filter banks of uniform translates, $\left\{\widehat{g^U_m}(\lambda)\right\}_{m=1,2,\ldots,M}$, of a shifted and scaled Hann kernel, $\widehat{g^U}(\lambda)$. The top black line in each figure shows $G(\lambda)$. Comparing (b) and (c), we see that for a fixed number of filters $M$, the overlap of the filters increases as the parameter $R$ increases.} 
 \label{Fig:hann2}
\end{figure}
\end{example}

\begin{remark}
For any fixed number $K$ of cosine terms in the definition of the kernel $\widehat{g^U}(\cdot)$ in \eqref{Eq:gu_kernel}, we can choose a nonzero coefficient sequence $\left\{a_k\right\}_{k=0,1,\ldots,K}$ such that $\widehat{g^U}(\cdot) \in {\cal C}^{2K-1}$. This can be seen by differentiating the kernel to find a linear system of equations for the coefficients \cite{nu81-1,ha78}. To satisfy this system of equations, the coefficient sequence must be in the kernel of a $K\times (K+1)$ matrix. Since this kernel is never trivial, we can always find nonzero coefficient sequences yielding the desired degree of smoothness.
\end{remark}

\section{Warping}
To generate systems of filter banks in the graph spectral domain, we now warp the uniform translates constructed in the previous section. Specifically, for a given warping function $\omega:[0,\lambda_{\max}]\rightarrow \Rbb$, we consider filters of the form
\begin{align}\label{Eq:warp_def}
\widehat{g_m}(\lambda) = \widehat{g_m^U}\left(\omega(\lambda)\right),~m=1,2,\ldots,M.
\end{align}
The role of the warping function $\omega(\cdot)$ is to scale the spectrum, and for different applications, different scalings of the spectrum can be desirable. In Section \ref{Se:log_wavelet}, we use a logarithmic scaling to generate a tight graph wavelet frame that is only adapted to the length of the spectrum. In Sections \ref{Se:adapted} and \ref{Se:random},
we leverage this same warping idea to scale the spectrum according to the distribution of eigenvalues over the spectrum,
in order to generate spectrum-adapted vertex-frequency frames. In Section \ref{Se:wavelets}, we compose the logarithmic and spectrum-based warping functions to generate spectrum-adapted graph wavelet frames.  In all cases, the warping function should be nondecreasing, 
and it is also desirable for the warping function to be smooth in order that the warped filters are smooth.
As detailed in the following remark, 
the sum of the squared magnitudes of the warped filters defined in \eqref{Eq:warp_def} is the same as the sum of the squared magnitudes of the initial system of translates,
and therefore the warping method is a good choice for constructing tight frames.
\begin{remark} If $\omega:[0,\lambda_{\max}] \rightarrow [0,\gamma]$, then, by Corollary \ref{Co:uniform} and Lemma \ref{Le:frame}, $\left\{T_i g_m\right\}_{i=1,2,\ldots,N;~m=1,2,\ldots,M}$ is also a tight frame, because
\begin{align*}
\sum_{m=1}^M |\widehat{g_m}(\lambda)|^2=\sum_{m=1}^M |\widehat{g_m^U}(\omega(\lambda))|^2=
Ra_0^2+\frac{R}{2}\sum_{k=1}^K a_k^2,
~\forall \lambda \in [0,\lambda_{\max}].
\end{align*}
\end{remark}

\subsection{Example: Tight Graph Wavelet Frames} \label{Se:log_wavelet}
Recently, \cite{christensen_warp} 
demonstrates 
that 
wavelets on the real line can be constructed by warping Gabor systems with a logarithmic warping function. In the same spirit, we now present a new method to construct tight wavelet frames in the graph setting by using a logarithmic function to warp systems of uniform translates in the graph spectral domain. 

To construct a set of $M-1$ wavelet kernels and one scaling kernel, we proceed as follows. First, define the warping function $\omega(x):=\log(x)$.\footnote{We take $\omega(0)$ to be $-\infty$ so that $\widehat{g_m}(0):=\widehat{g_{m-1}^U}\bigl(-\infty\bigr)=0$. Alternatively, in numerical implementations, we can define $\omega(x):=\log(x)+\epsilon$, where $\epsilon$ is an arbitrarily small constant.}
Second, as described in Corollary \ref{Co:uniform}, choose $2<R\leq M$ and $K\leq \frac{R}{2}$ and construct a set of uniform translates, $\left\{\widehat{g_m^U}(\cdot)\right\}_{m=1,2,\ldots,M-1}$, with $\gamma=\omega(\lambda_{\max})$. Finally, define the $M-1$ wavelet kernels as 
\begin{align} \label{Eq:logwarp0}
\widehat{g_m}(\lambda):=\widehat{g_{m-1}^U}\bigl(\omega(\lambda)\bigr),~m=2,3,\ldots,M, 
\end{align} 
and the scaling kernel as
\begin{align} \label{Eq:logwarp1}
\widehat{g_1}(\lambda):=\sqrt{Ra_0^2+\frac{R}{2}\sum_{k=1}^K a_k^2
-\sum_{m=2}^M |\widehat{g_m}(\lambda)|^2}.
\end{align}
Note that for some values of $\lambda$ in $[0,\lambda_{\max}]$, $\omega(\lambda)\notin[0,\gamma]$; however, 
the form of the scaling kernel \eqref{Eq:logwarp1} and Lemma \ref{Le:frame} ensure that $\left\{T_i g_m\right\}_{i=1,2,\ldots,N;~m=1,2,\ldots,M}$ is still a tight wavelet frame.

\begin{example} \label{Ex:wavelets}
In Figure \ref{Fig:wavelets}(c), we show an example of graph wavelet and scaling kernels generated in the above fashion, using Hann kernels ($K=1$ and $a_0=a_1=\frac{1}{2}$) with $\lambda_{\max}=12$, $R=3$, and $M=8$. Comparing this system to the corresponding kernels used for the spectral graph wavelet transform (SGWT) \cite{sgwt} and Meyer-like graph wavelet frame \cite{leonardi_fmri,leonardi_multislice}, we see that, similar to the Meyer-like kernels, the log-warped kernels lead to a tight frame and 
the support of each wavelet kernel is a strict subset of the spectrum $[0,\lambda_{\max}]$
(analogously to bandlimited wavelets on the real line);  
however, the overlap and shape of the wavelet kernels is closer to the spline-based SGWT wavelet kernels.

\begin{figure}[h]
\centering
\begin{minipage}[b]{.32\linewidth}
\centerline{~~~Spectral Graph}
\centerline{~~~Wavelet Frame \cite{sgwt}}
\centerline{\includegraphics[width=\linewidth]{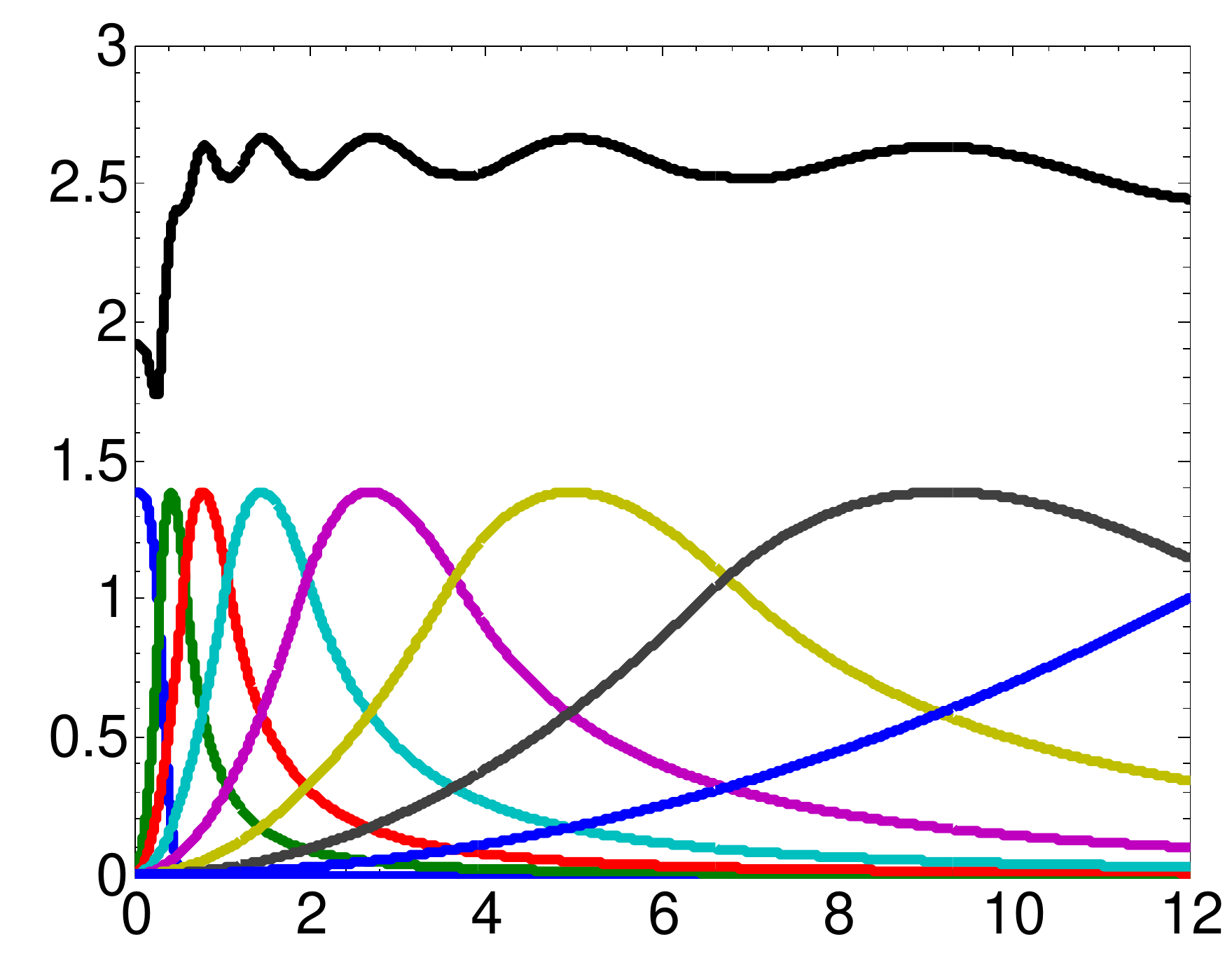}} 
\centerline{\small{~~~~$\lambda$}}
\centerline{\small{~~~~(a)}}
\end{minipage}
\hfill
\begin{minipage}[b]{.32\linewidth}
\centerline{~~~Meyer-Like Tight}
\centerline{~~~Graph Wavelet Frame \cite{leonardi_fmri,leonardi_multislice}}
\centerline{\includegraphics[width=\linewidth]{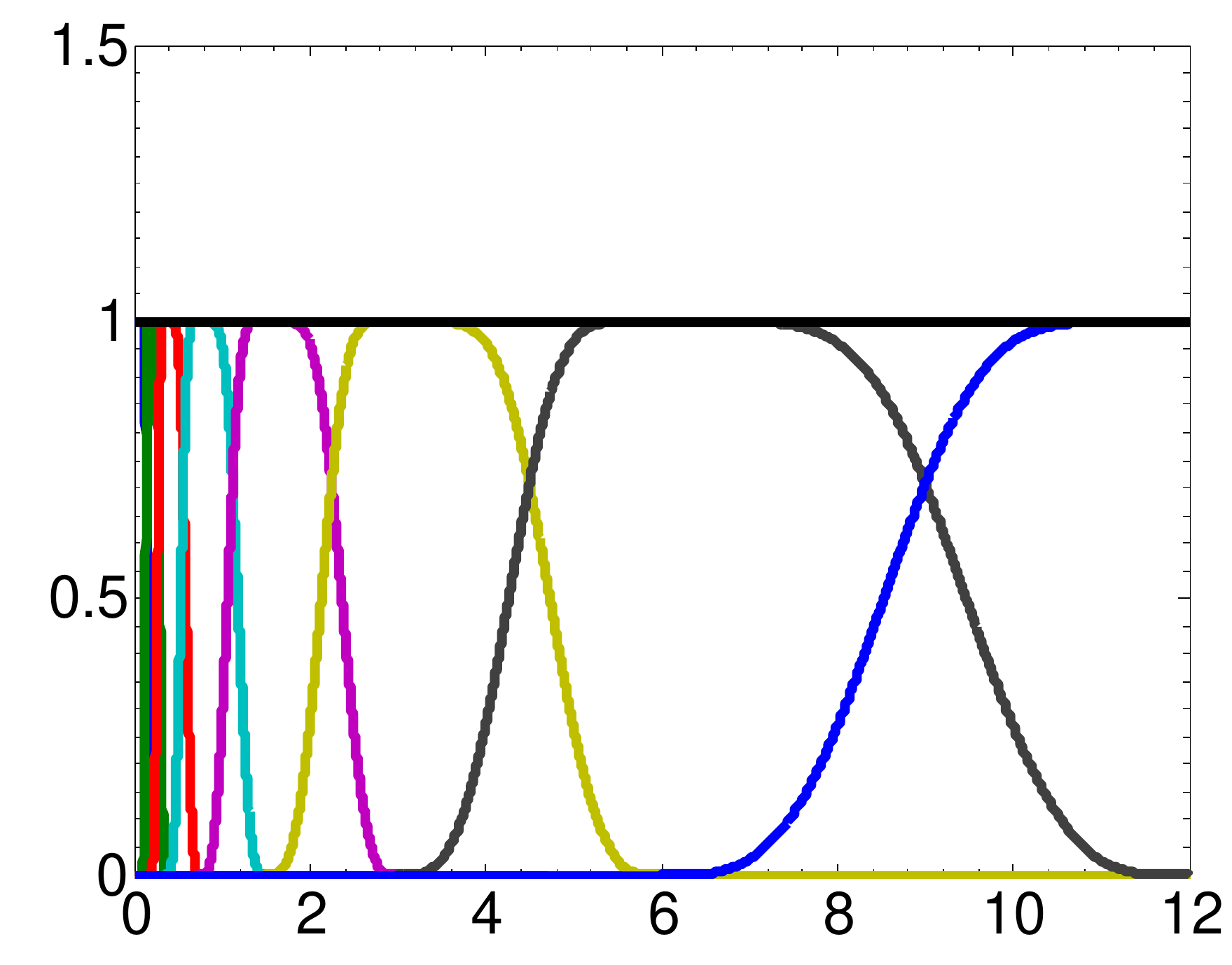}} 
\centerline{\small{~~~~$\lambda$}}
\centerline{\small{~~~~(b)}}
\end{minipage}
\hfill
\begin{minipage}[b]{.32\linewidth}
\centerline{~~~Log-Warped Tight}
\centerline{~~~Graph Wavelet Frame}
\centerline{\includegraphics[width=\linewidth]{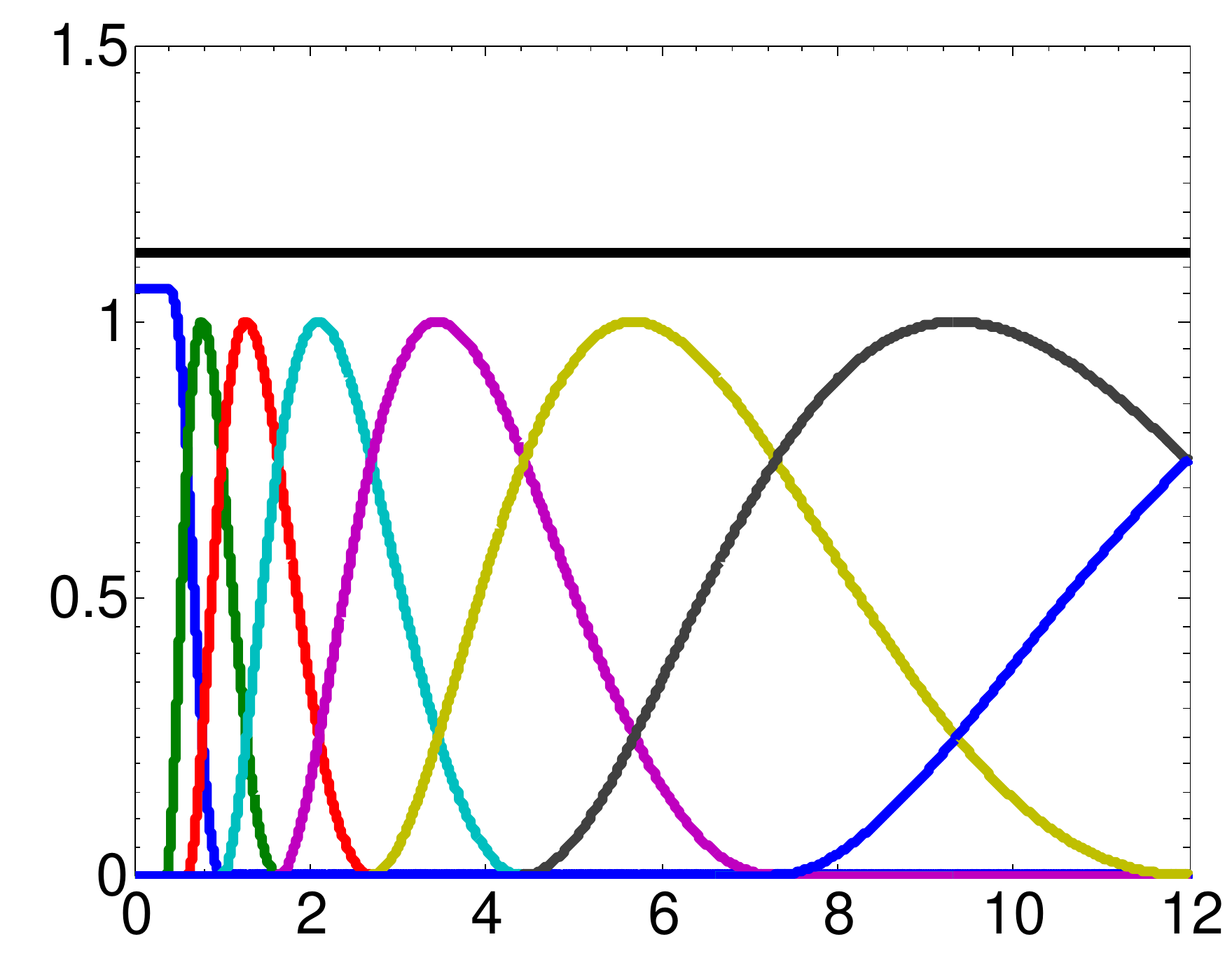}} 
\centerline{\small{~~~~$\lambda$}}
\centerline{\small{~~~~(c)}}
\end{minipage} 
\caption {Three different sets of wavelet and scaling kernels in the graph spectral domain. The top black line in each figure is $G(\lambda)$.} 
 \label{Fig:wavelets}
\end{figure}
\end{example}

\section{Spectrum-Adapted Filters} \label{Se:adapted}
While each atom of the form $T_i g_m=\sqrt{N}\widehat{g_m}(\L)\delta_i$ generated from the filters in Examples \ref{Ex:uniform} and \ref{Ex:wavelets} is adapted to the particular graph spectrum through the matrix function $\widehat{g_m}(\L)$, the filters themselves are only adapted to the length of the discrete spectrum, and not to the specific locations of the eigenvalues. As discussed in \cite{shuman_ACHA_2013}, in order to extract information from signals with oscillations that are localized on the graph, it is useful to develop atoms that are simultaneously localized in both the vertex domain and the graph spectral domain. In classical continuous-time or discrete-time time-frequency analysis, we can form such atoms by modulating and then translating a window, where the modulation is a translation in the Fourier domain. In the graph setting, however, the Laplacian spectrum is not only finite, but it is not uniformly distributed. Therefore,   
as Example \ref{Ex:eigen_loc} below demonstrates, simply shifting filters in the graph spectral domain is not the ideal way to break the spectrum up into different frequency bands for analysis. 
\begin{example}\label{Ex:eigen_loc}
In Figure \ref{Fig:three_graphs}, we show three different graphs with $N=64$ vertices. 
In Figure \ref{Fig:uniform_eigen_loc}, we plot systems of eight uniform translates of the form \eqref{Eq:uniform_construction}, with $\gamma=\lambda_{\max}$ for the three different graphs.  
The filters are only adapted to the length of the spectrum, $\lambda_{\max}$; however, we also show the locations of the graph Laplacian eigenvalues with ``x'' marks on the horizontal axis. Throughout the paper, we mark the eigenvalues locations that are used in the design of the filters in red, and those that are not known or not used in the design of the filters in black.

\begin{figure}[h]
\centering
\begin{minipage}[b]{.32\linewidth}
\centerline{Path Graph}
\centerline{\includegraphics[width=\linewidth]{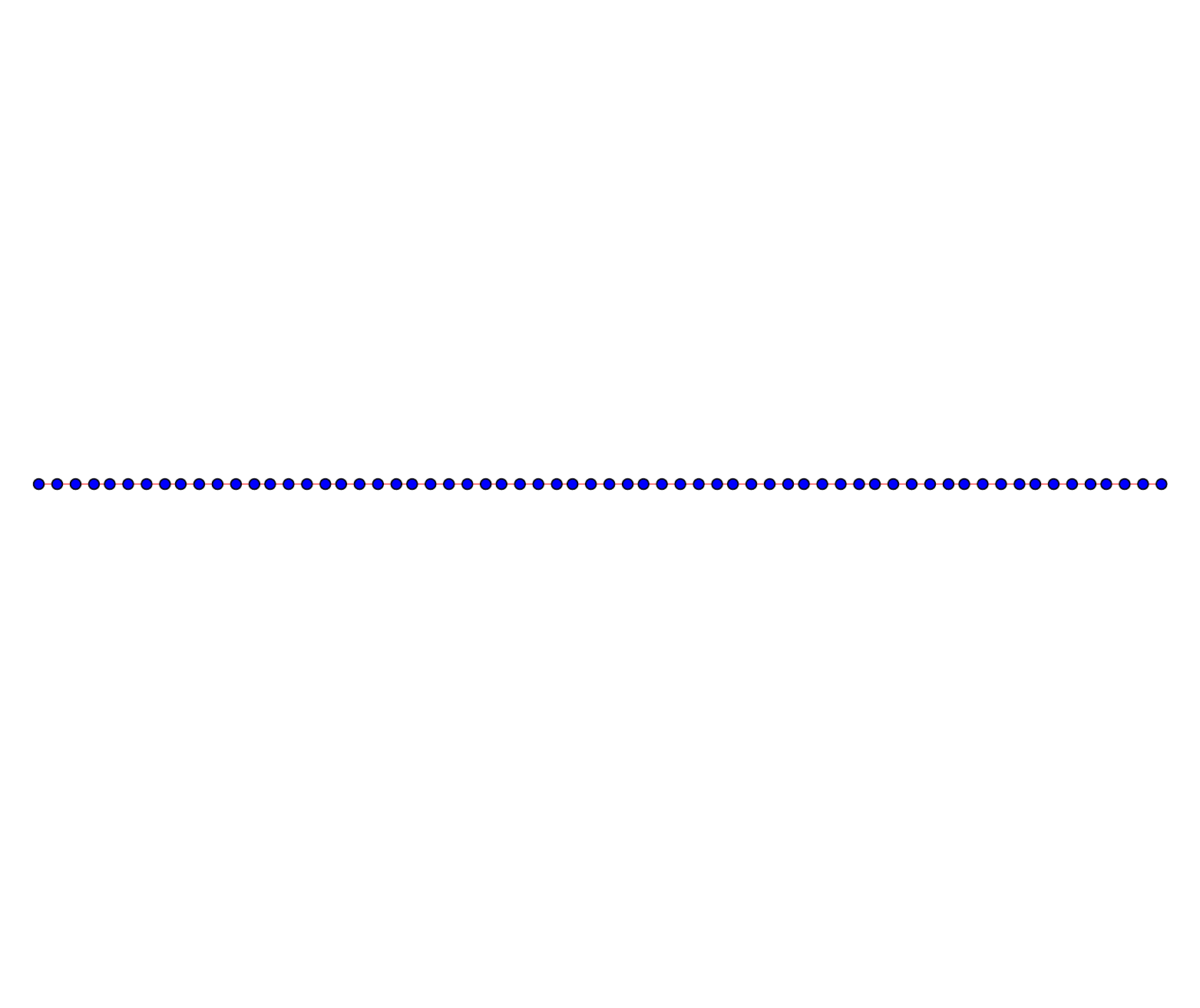}} 
\centerline{\small{(a)}}
\end{minipage}
\hfill
\begin{minipage}[b]{.32\linewidth}
\centerline{Sensor Network}
\centerline{\includegraphics[width=\linewidth]{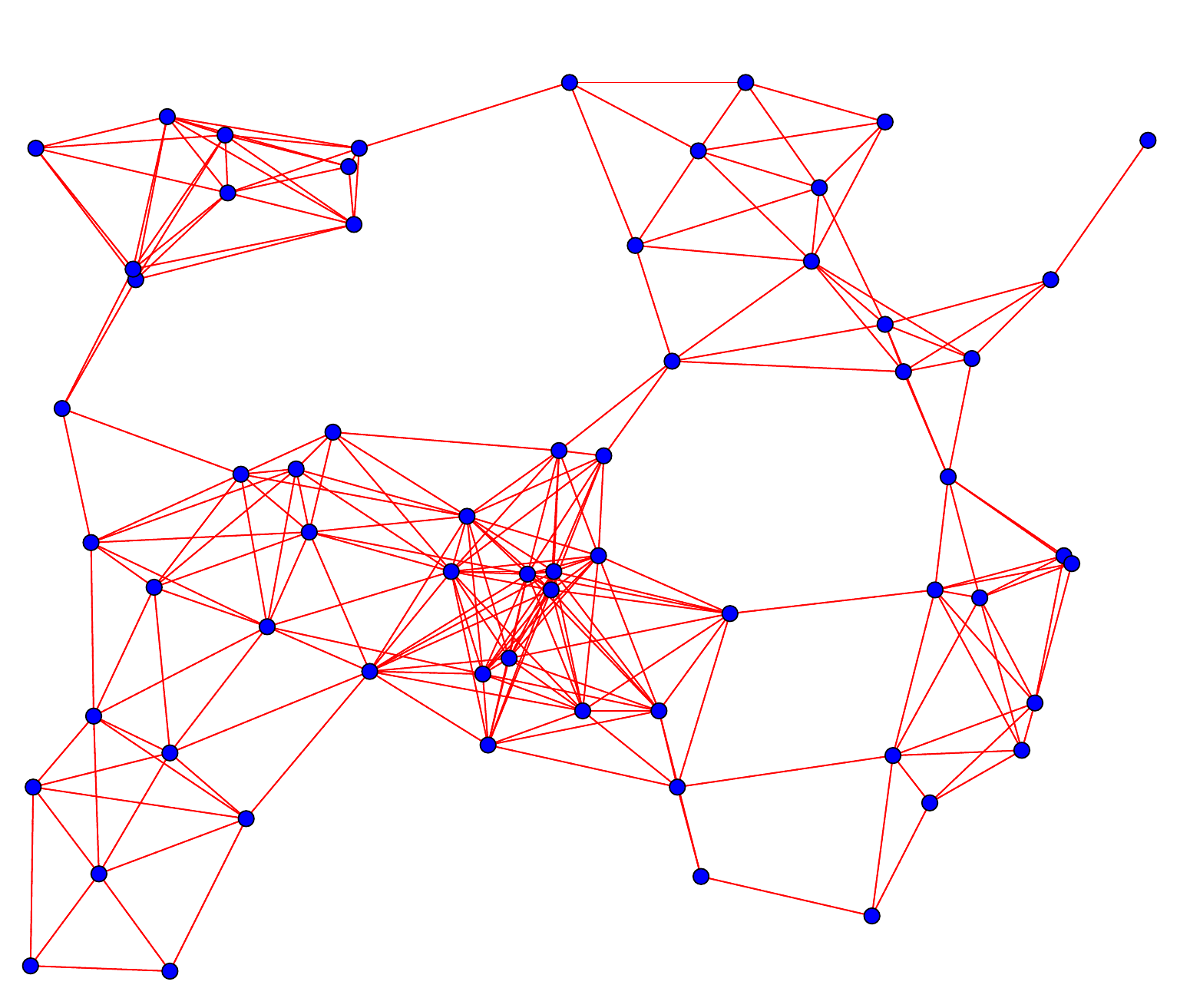}} 
\centerline{\small{(b)}}
\end{minipage}
\hfill
\begin{minipage}[b]{.32\linewidth}
\centerline{Comet Graph}
\centerline{\includegraphics[width=\linewidth]{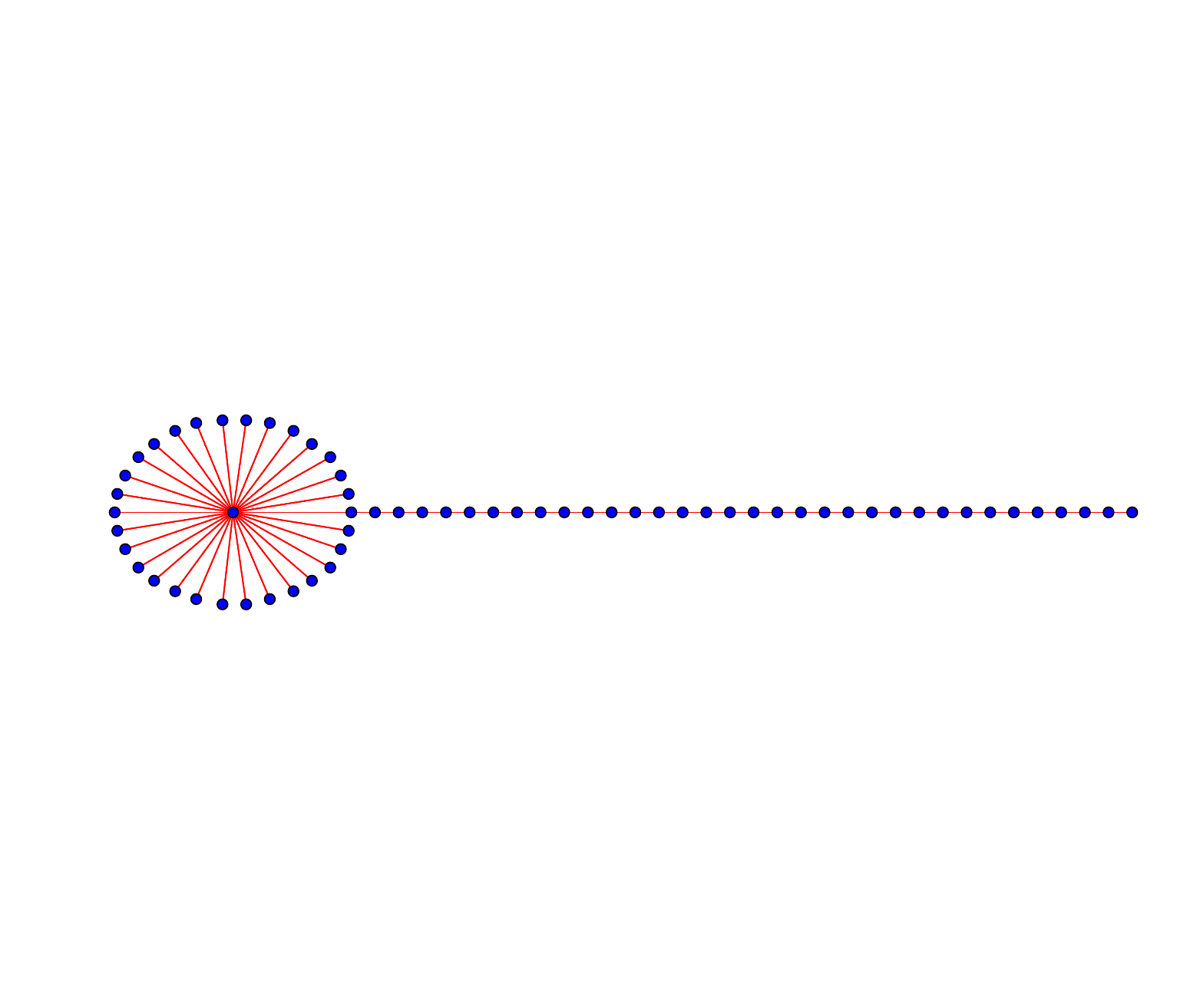}} 
\centerline{\small{(c)}}
\end{minipage} 
\caption {Three different graphs with 64 vertices. The degree of the center vertex in the comet graph in (c) is 30. The non-zero edge weights in (a) and (c) are all equal to 1. The edge weights in the  
sensor network in (b) are assigned based on physical distance via a thresholded Gaussian kernel weighting function (see, e.g., \cite[Equation (1)]{shuman_SPM}).} 
 \label{Fig:three_graphs}
\end{figure}
\begin{figure}[h]
\centering
\begin{minipage}[b]{.32\linewidth}
\centerline{~~~Path Graph}
\centerline{\includegraphics[width=\linewidth]{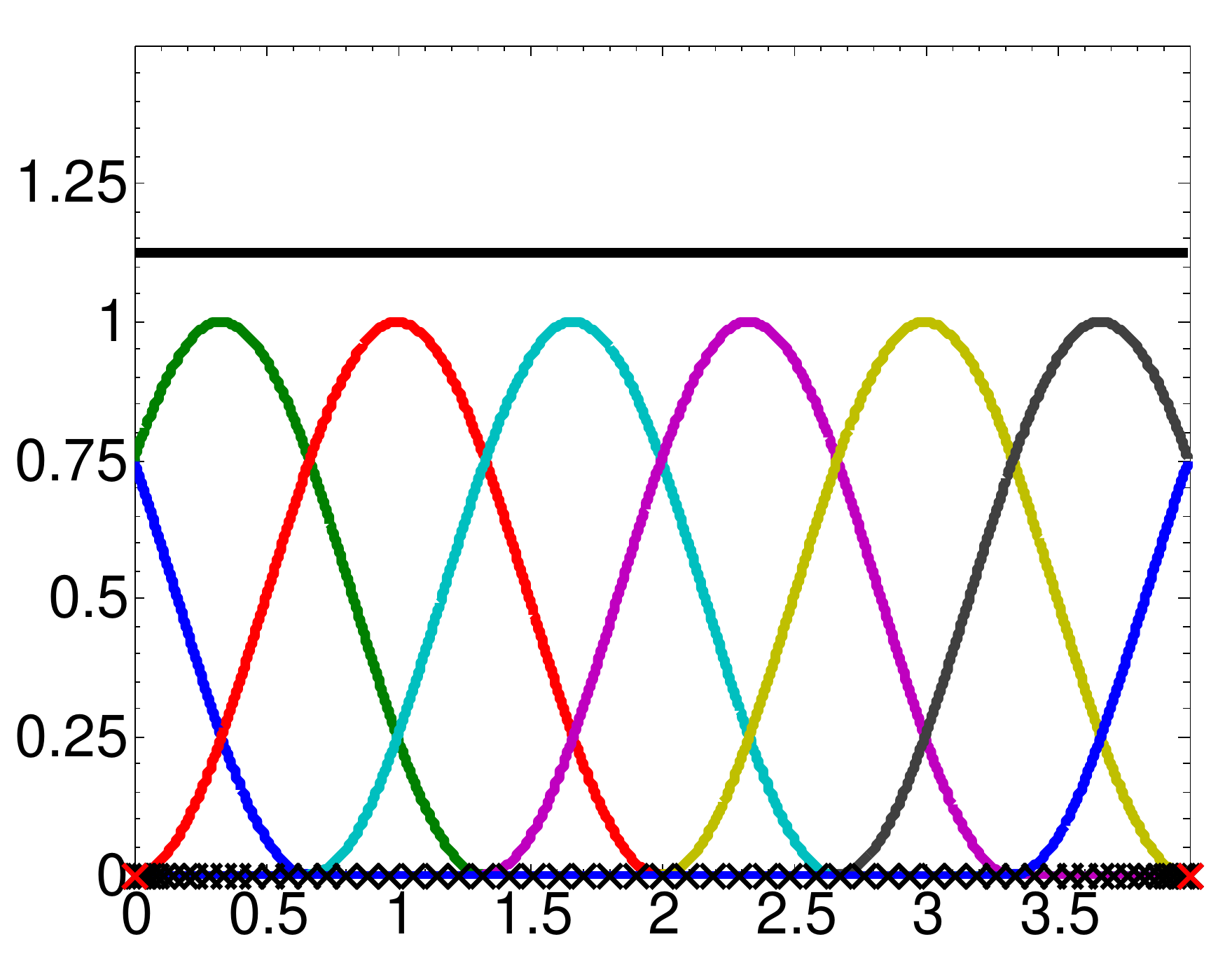}} 
\centerline{\small{~~~~$\lambda$}}
\centerline{\small{~~~~(a)}}
\end{minipage}
\hfill
\begin{minipage}[b]{.32\linewidth}
\centerline{~~~Sensor Network}
\centerline{\includegraphics[width=\linewidth]{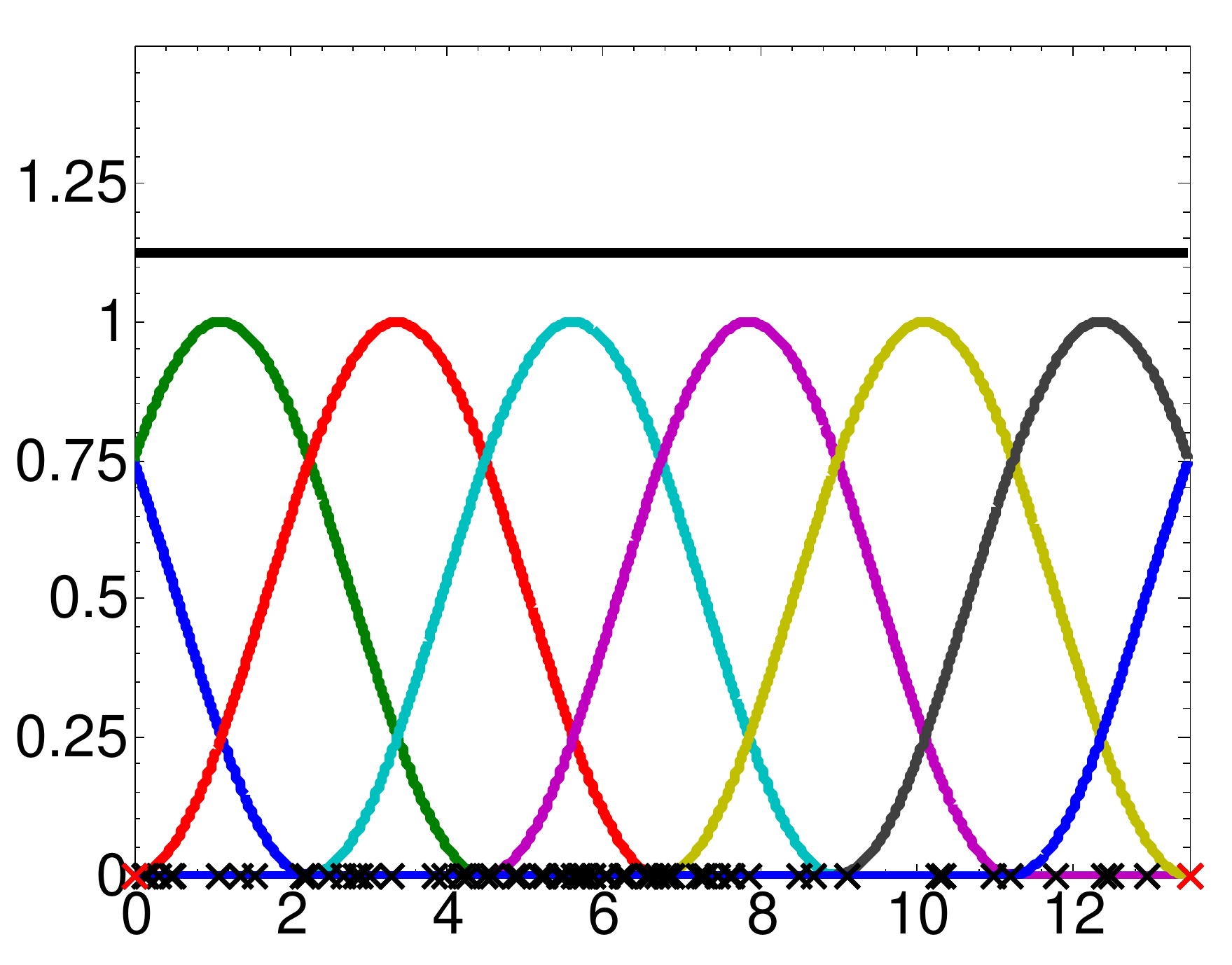}} 
\centerline{\small{~~~~$\lambda$}}
\centerline{\small{~~~~(b)}}
\end{minipage}
\hfill
\begin{minipage}[b]{.32\linewidth}
\centerline{~~~Comet Graph}
\centerline{\includegraphics[width=\linewidth]{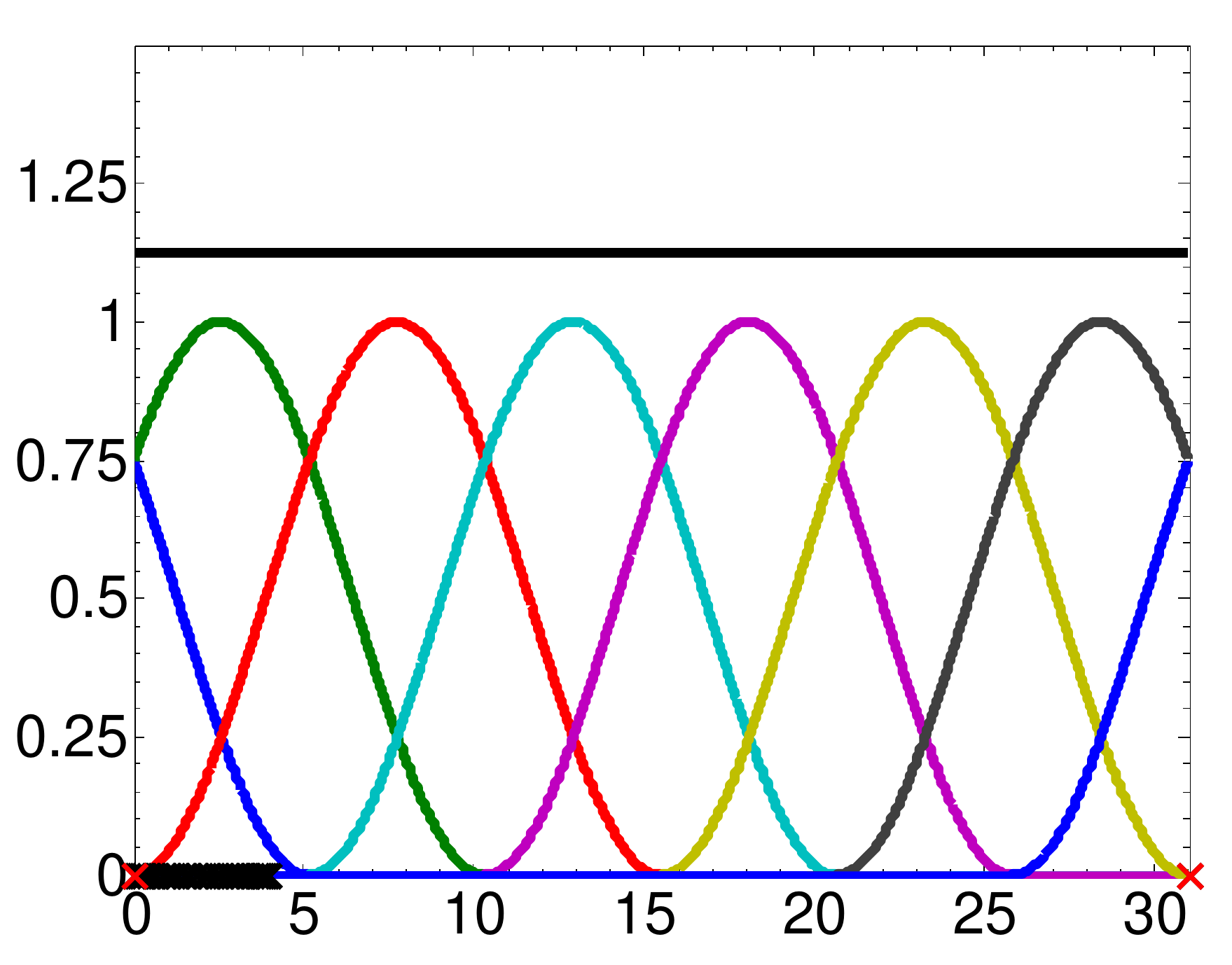}} 
\centerline{\small{~~~~$\lambda$}}
\centerline{\small{~~~~(c)}}
\end{minipage} 
\caption {Systems of uniformly translated filters, $\left\{\widehat{g_m^U}\right\}_{m=1,2,\ldots,8}$, adapted to the length, $\lambda_{\max}$, of the graph Laplacian spectrum for three different graphs, each with $N=64$ vertices. The locations of the graph Laplacian eigenvalues are marked on the horizontal axis. Only those shown in red ($\lambda_0$ and $\lambda_{\max}$) are used in the design of the filters.} 
 \label{Fig:uniform_eigen_loc}
\end{figure}

We see that simply shifting filters in the graph spectral domain may lead to a disparity in the number of graph Laplacian eigenvalues (frequencies) in each frequency band, which is not ideal for information extraction.
As an extreme example, for the comet graph, because $\widehat{g_5^U}(\lambda)=0$ for all $\lambda \in \sigma(\L)$ in Figure \ref{Fig:uniform_eigen_loc}(c), $\ip{f}{T_i g_5^U}=0$ for all $i\in\{1,2,\ldots,N\}$ and every signal $f\in \Rbb^N$. Therefore, given a fixed number of filters and knowledge about the locations of the eigenvalues, $\widehat{g_5^U}(\cdot)$ (shown in magenta) is not a good choice of a filter, because $\ip{f}{T_i g_5^U}$ provides no additional information about \emph{any} signal $f$ on this comet graph.
\end{example}

In the remainder of this section, we present a method to incorporate some knowledge about the locations of the graph Laplacian eigenvalues into the design of the system of filters, in a manner such that the resulting analysis coefficients $\ip{f}{T_i g_m}$ provide more information about the signal $f$. Our general approach is to estimate the density of the graph Laplacian eigenvalues, and then warp the spectrum accordingly.

In Section \ref{Se:warp_exact}, 
we assume that we know all of the eigenvalues exactly; however, in extremely large graphs, it is computationally prohibitive to compute this full spectrum, and therefore in Section \ref{Se:approx_spectrum}, we discuss how to approximate the 
density of the graph Laplacian eigenvalues in a more efficient manner.

\subsection{Spectrum-Based Warping Functions} \label{Se:warp_exact}
The \emph{spectral density function} (see, e.g., \cite[Chapter 6]{van_mieghem}) or \emph{empirical spectral distribution} (see, e.g., \cite[Chapter 2.4]{tao_random_matrix}) of the graph Laplacian eigenvalues of a given graph $\G$ with $N$ vertices is the probability measure 
\vspace{-.1cm}
\begin{align*}
p_{\lambda}(s):=\frac{1}{N}\sum_{\l=0}^{N-1} \Identity_{\left\{\lambda_{\l}=s\right\}}.
\end{align*}
Similarly, we can define a \emph{cumulative spectral density function} or \emph{empirical spectral cumulative distribution} as
\vspace{-.1cm}
\begin{align}\label{Eq:cdf}
P_{\lambda}(z):=\frac{1}{N}\sum_{\l=0}^{N-1} \Identity_{\left\{\lambda_{\l}\leq z\right\}}.
\end{align}

One method to adapt the uniform translates of Example \ref{Ex:eigen_loc} so that the support of each filter includes a similar number of eigenvalues is, as in \eqref{Eq:warp_def}, to let the filters be of the form $\widehat{g_m}(\lambda) = \widehat{g_m^U}\left(\omega(\lambda)\right)$, with the cumulative spectral density function \eqref{Eq:cdf} used as the warping function $\omega(\lambda)$.  
However, for finite deterministic graphs, doing so results in discontinuous filters, as the cumulative spectral density function is discontinuous. We prefer smooth filters, because (i) results characterizing the localization of  $T_i g_m$ in the vertex domain (see, e.g., \cite{shuman_ACHA_2013}) depend on smoothness of $\widehat{g_m}(\cdot)$ in the graph spectral domain; and (ii) smooth kernels can be better approximated by low-order polynomials, which is relevant for approximate computational approaches (see, e.g., \cite[Section 6]{sgwt}). 

Rather, we build a continuous warping function that approximates the cumulative spectral density function by interpolating the points\footnote{By setting the first interpolation point to $(0,0)$ and the last to $(\lambda_{\max},1)$, we ensure that the support, $[0,\lambda_{\max}]$, of the warped filters is mapped to the full support of the uniform translates. In the case of a repeated eigenvalue $\lambda_{\l}=\lambda_{\l+1}=\ldots=\lambda_{\l+k}$, we just include the single point $\left(\lambda_{\l+k},\frac{\l+k}{N-1}\right)$ in the set of interpolation points.} 
\begin{align}\label{Eq:interp_points}
\left\{\left(\lambda_{\l},\frac{\l}{N-1}\right)\right\}_{\l=0,1,\ldots,N-1}. 
\end{align}
We consider two interpolation methods: simple linear interpolation and monotonic cubic interpolation \cite{fritsch}.

\begin{figure}[ht]
\centering
\begin{minipage}[b]{.32\linewidth}
\centerline{~~~Path Graph}
\centerline{\includegraphics[width=\linewidth]{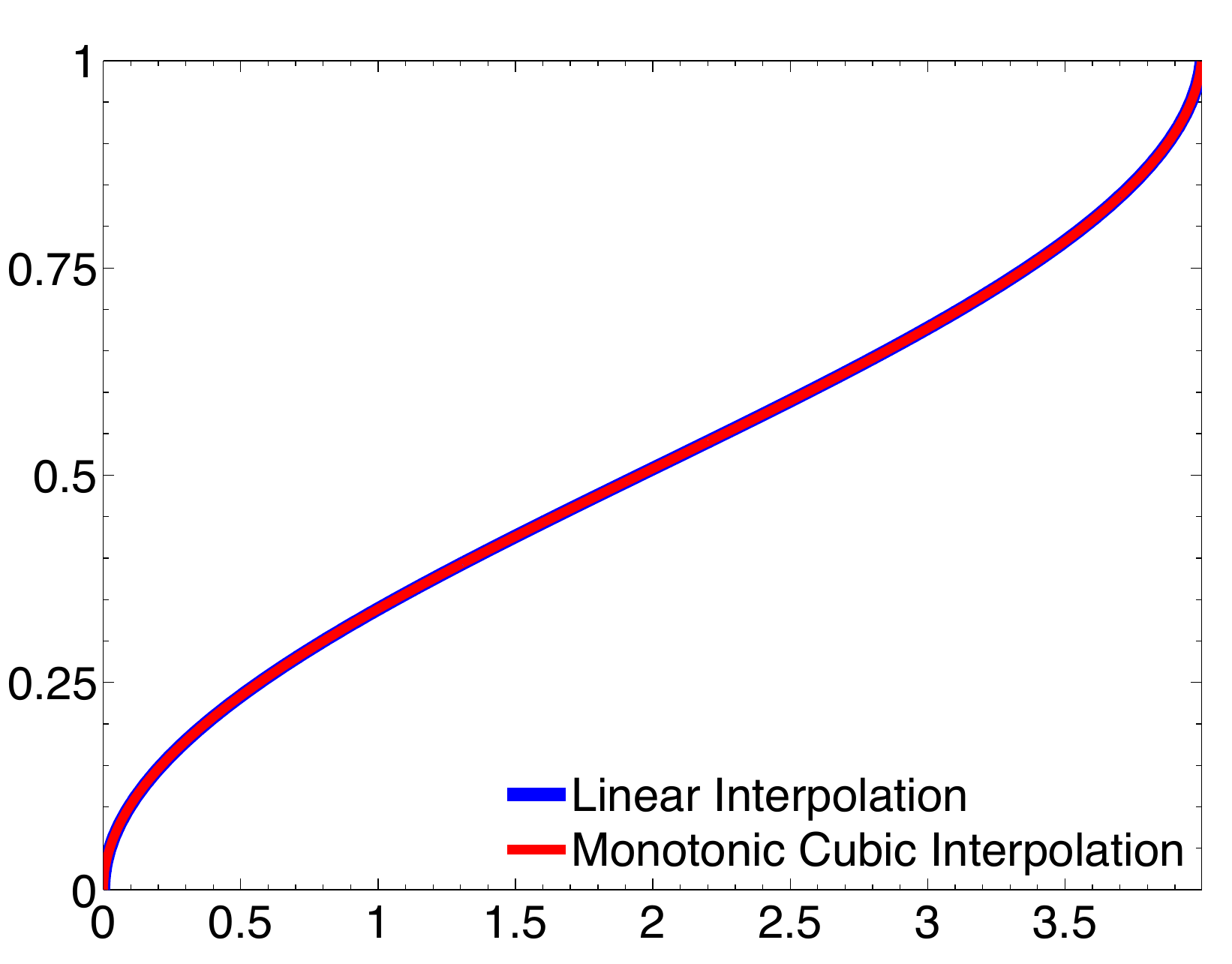}} 
\centerline{\small{~~~~$\lambda$}}
\centerline{\small{~~~~(a)}}
\end{minipage}
\hfill
\begin{minipage}[b]{.32\linewidth}
\centerline{~~~Sensor Network}
\vspace{.025in}
\centerline{\includegraphics[width=\linewidth]{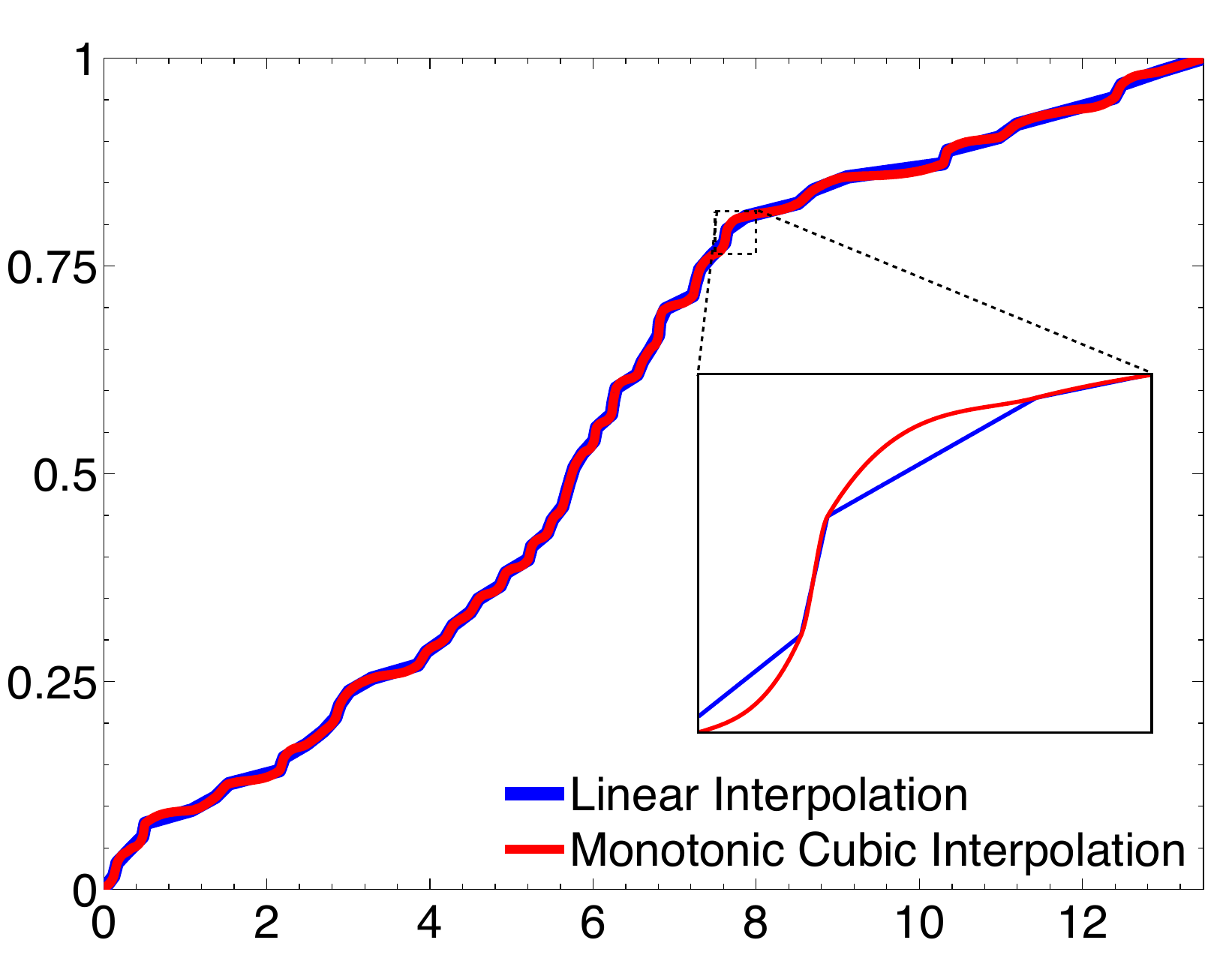}} 
\centerline{\small{~~~~$\lambda$}}
\centerline{\small{~~~~(b)}}
\end{minipage}
\hfill
\begin{minipage}[b]{.32\linewidth}
\centerline{~~~Comet Graph}
\centerline{\includegraphics[width=\linewidth]{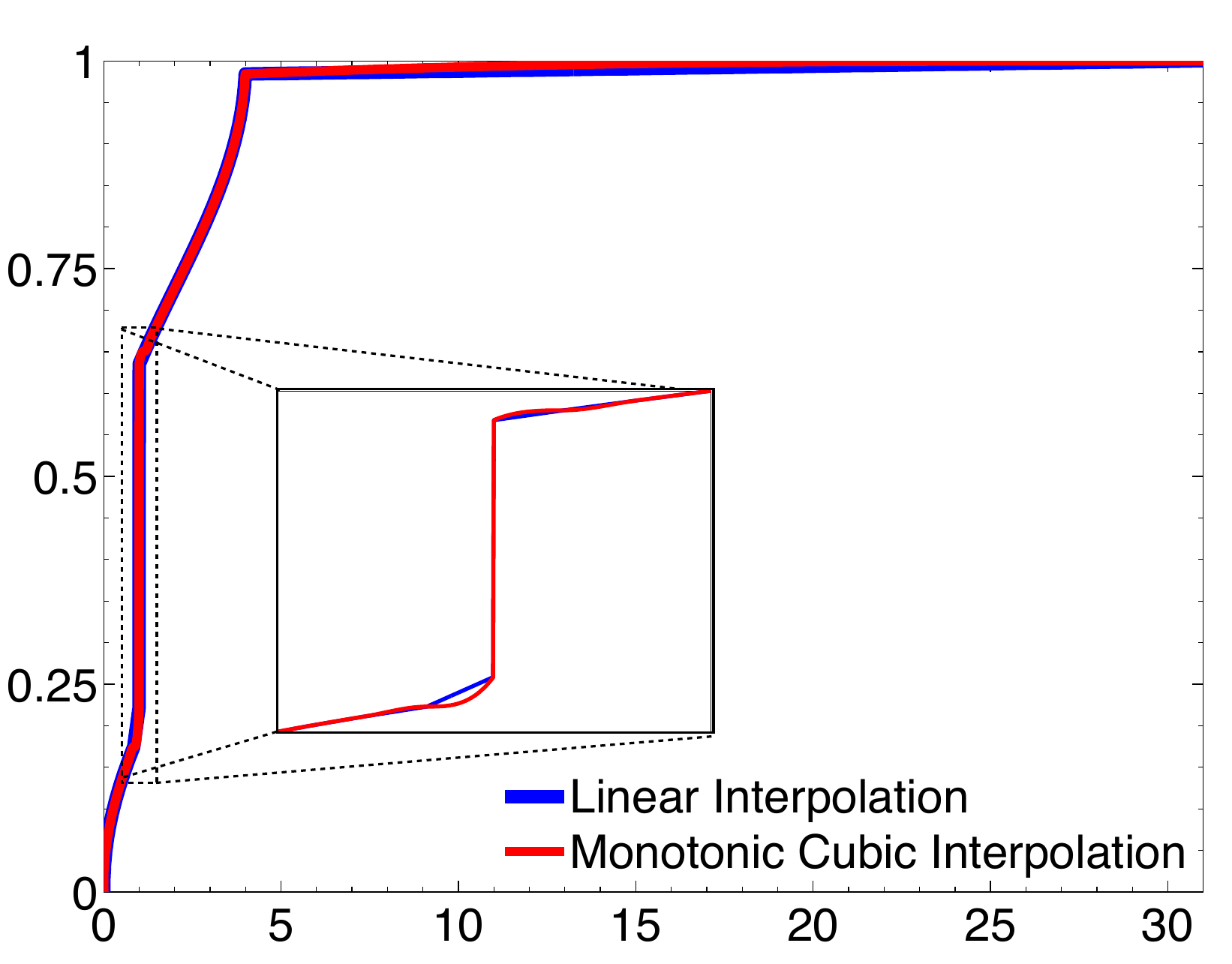}} 
\centerline{\small{~~~~$\lambda$}}
\centerline{\small{~~~~(c)}}
\end{minipage}  \\
\vspace{.1in}
\begin{minipage}[b]{.32\linewidth}
\centerline{\includegraphics[width=\linewidth]{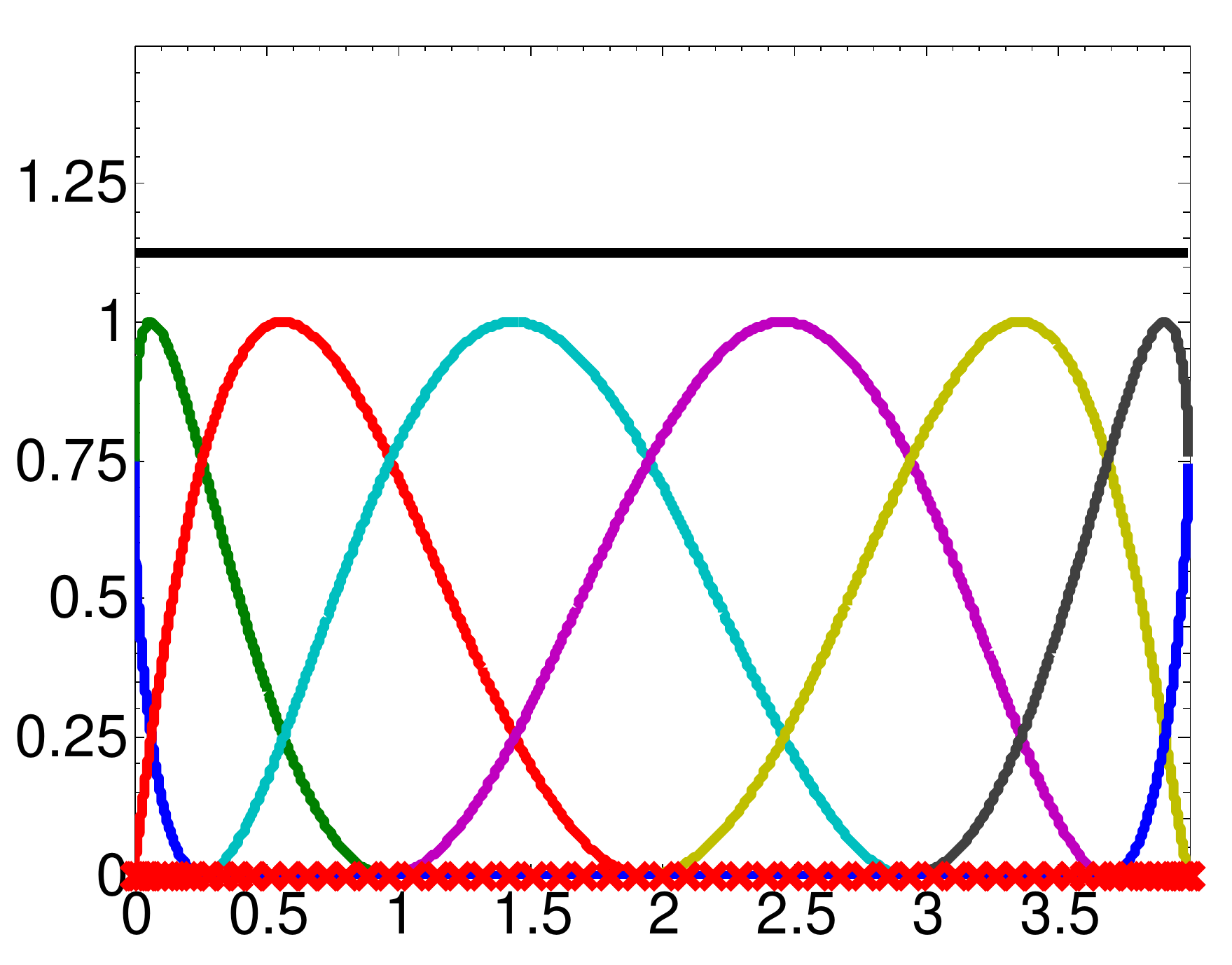}} 
\centerline{\small{~~~~$\lambda$}}
\centerline{\small{~~~~(d)}}
\end{minipage}
\hfill
\begin{minipage}[b]{.32\linewidth}
\centerline{\includegraphics[width=\linewidth]{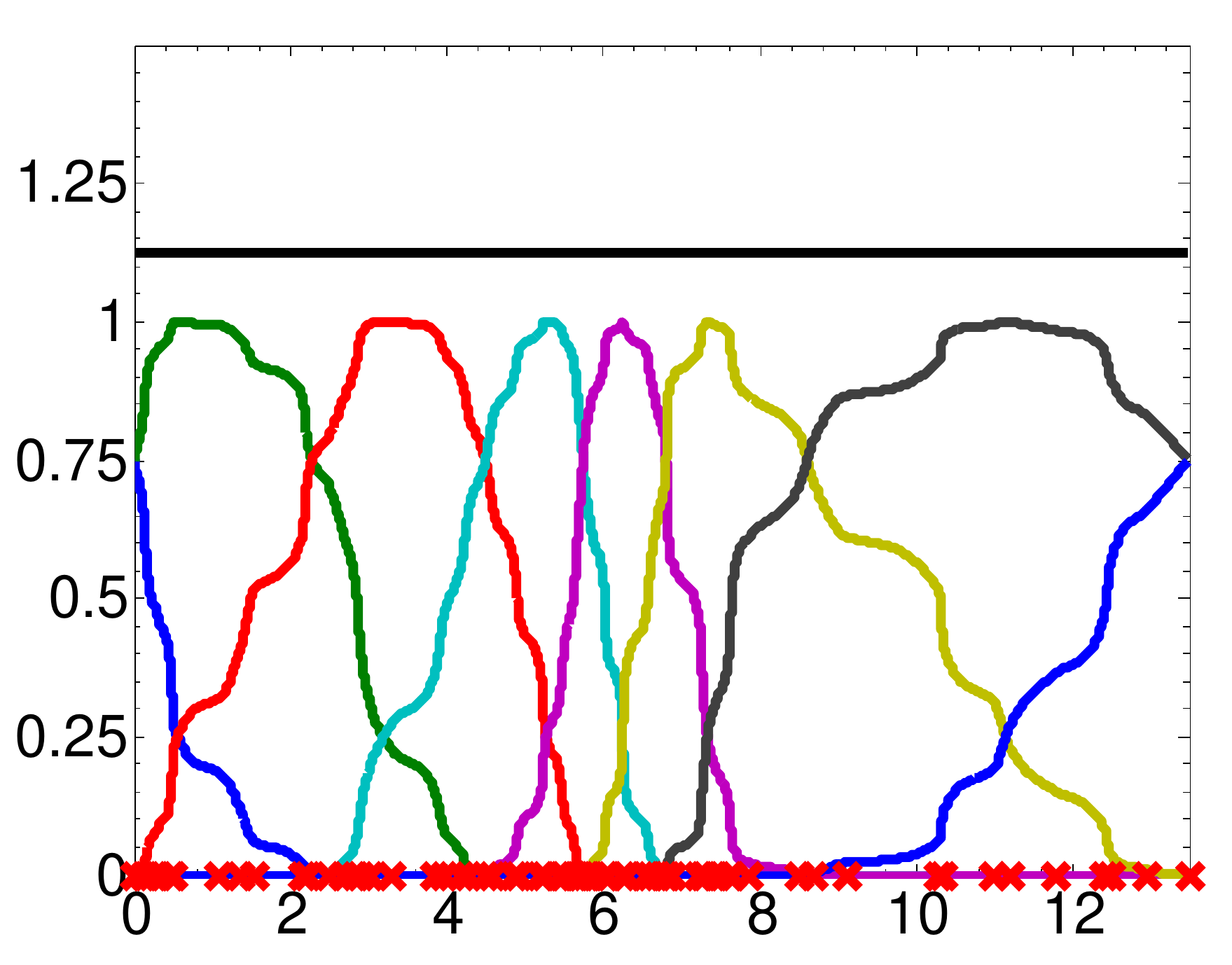}} 
\centerline{\small{~~~~$\lambda$}}
\centerline{\small{~~~~(e)}}
\end{minipage}
\hfill
\begin{minipage}[b]{.32\linewidth}
\centerline{\includegraphics[width=\linewidth]{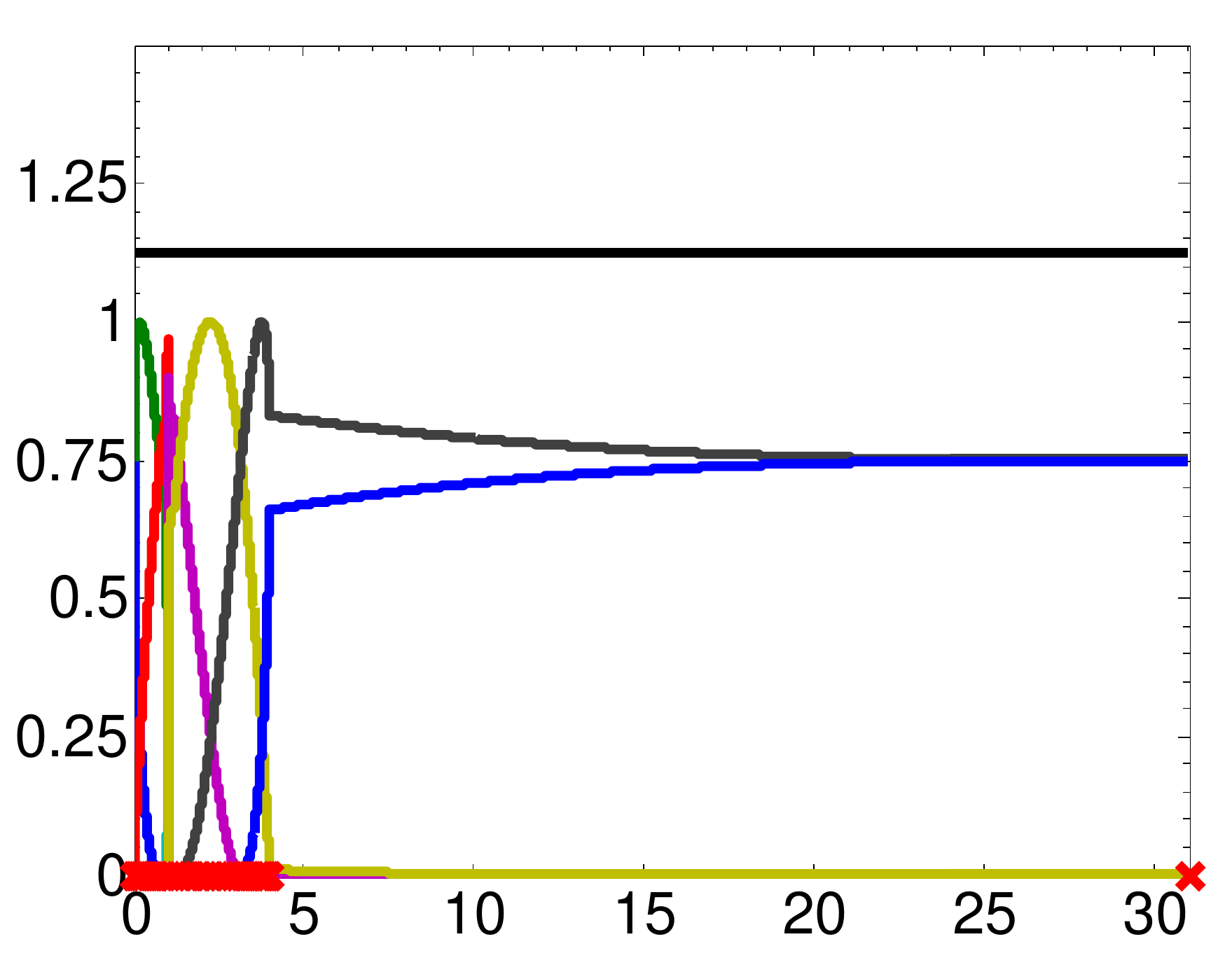}} 
\centerline{\small{~~~~$\lambda$}}
\centerline{\small{~~~~(f)}}
\end{minipage} \\
\vspace{.1in}
\begin{minipage}[b]{.32\linewidth}
\centerline{\includegraphics[width=\linewidth]{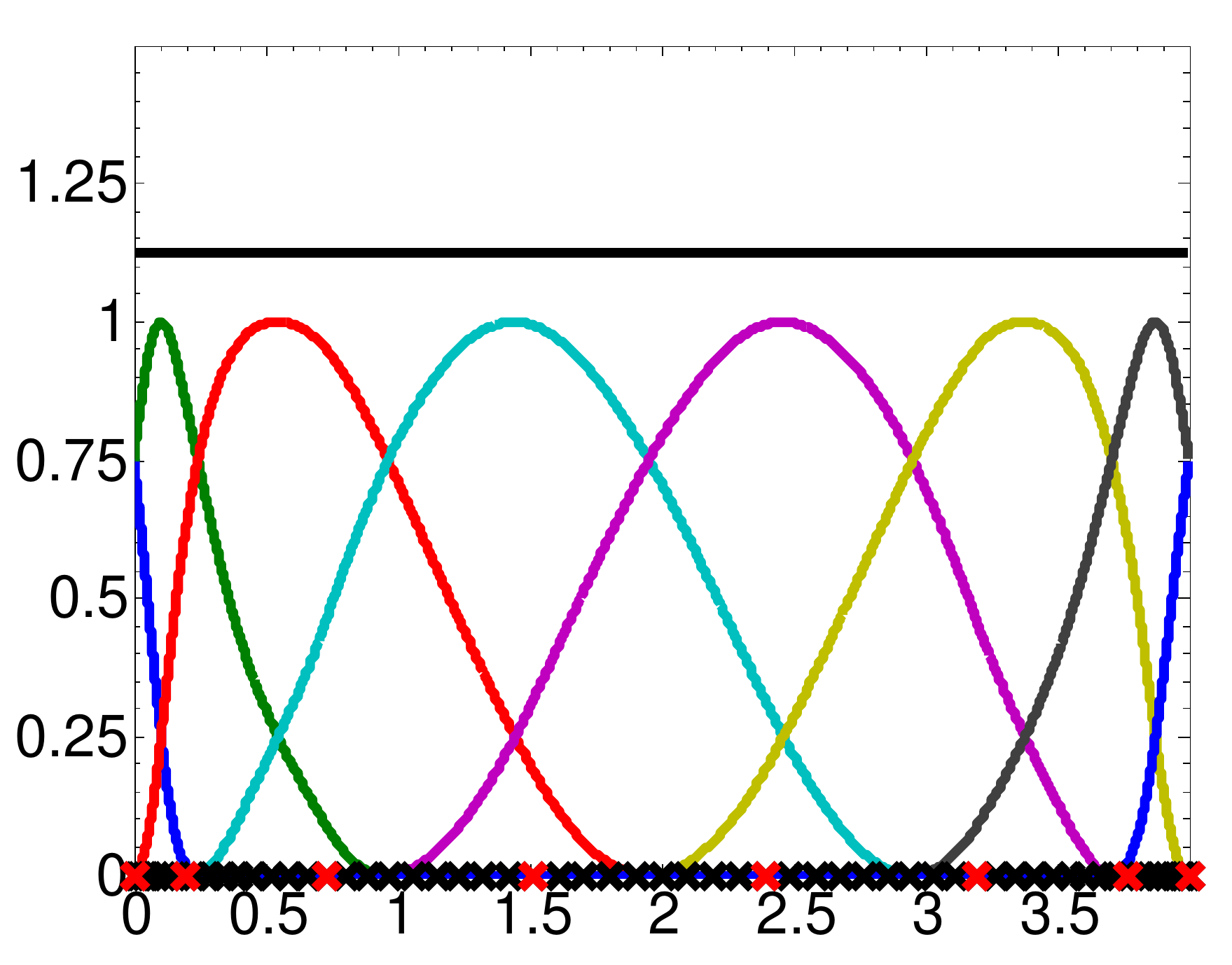}} 
\centerline{\small{~~~~$\lambda$}}
\centerline{\small{~~~~(g)}}
\end{minipage}
\hfill
\begin{minipage}[b]{.32\linewidth}
\centerline{\includegraphics[width=\linewidth]{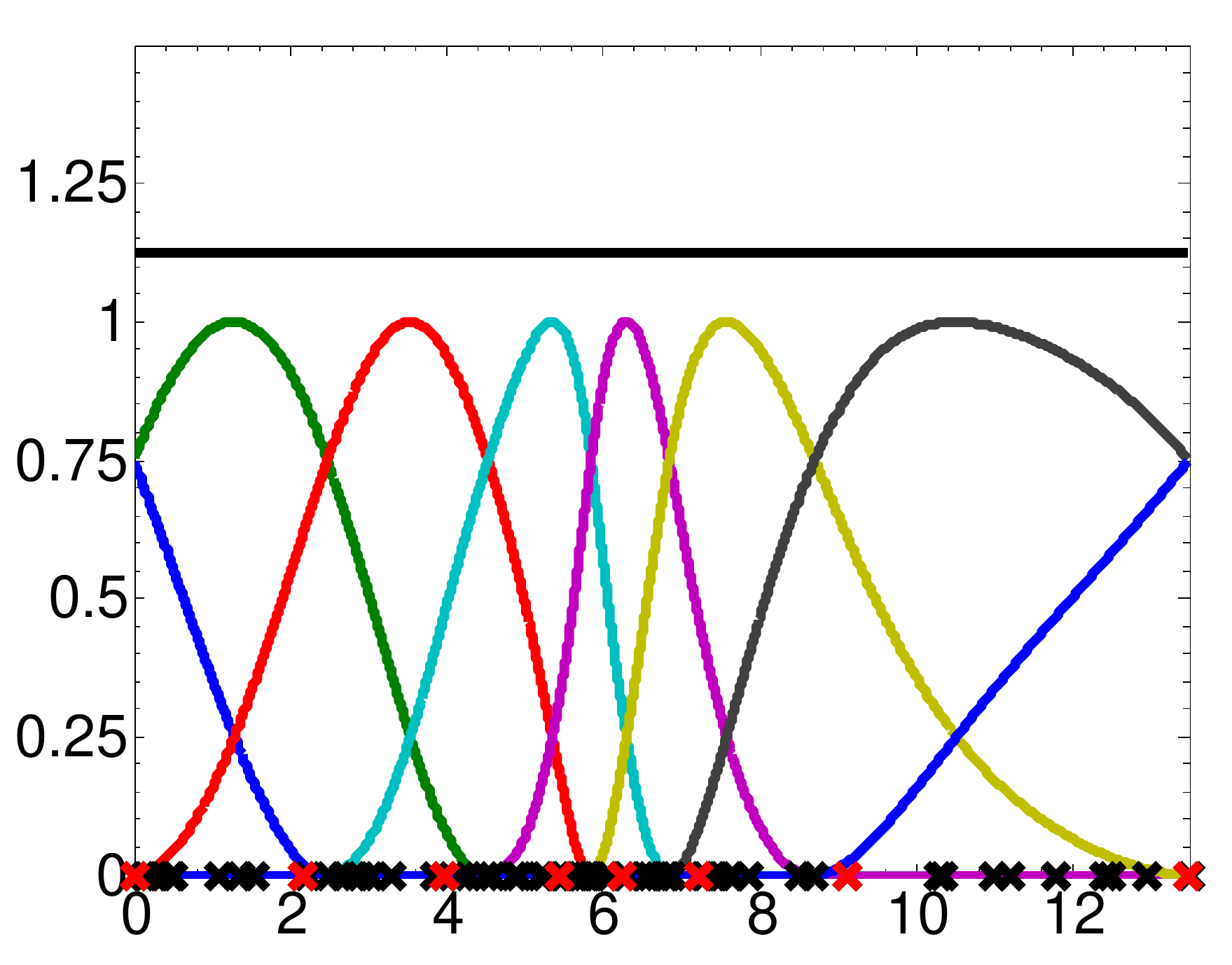}} 
\centerline{\small{~~~~$\lambda$}}
\centerline{\small{~~~~(h)}}
\end{minipage}
\hfill
\begin{minipage}[b]{.32\linewidth}
\centerline{\includegraphics[width=\linewidth]{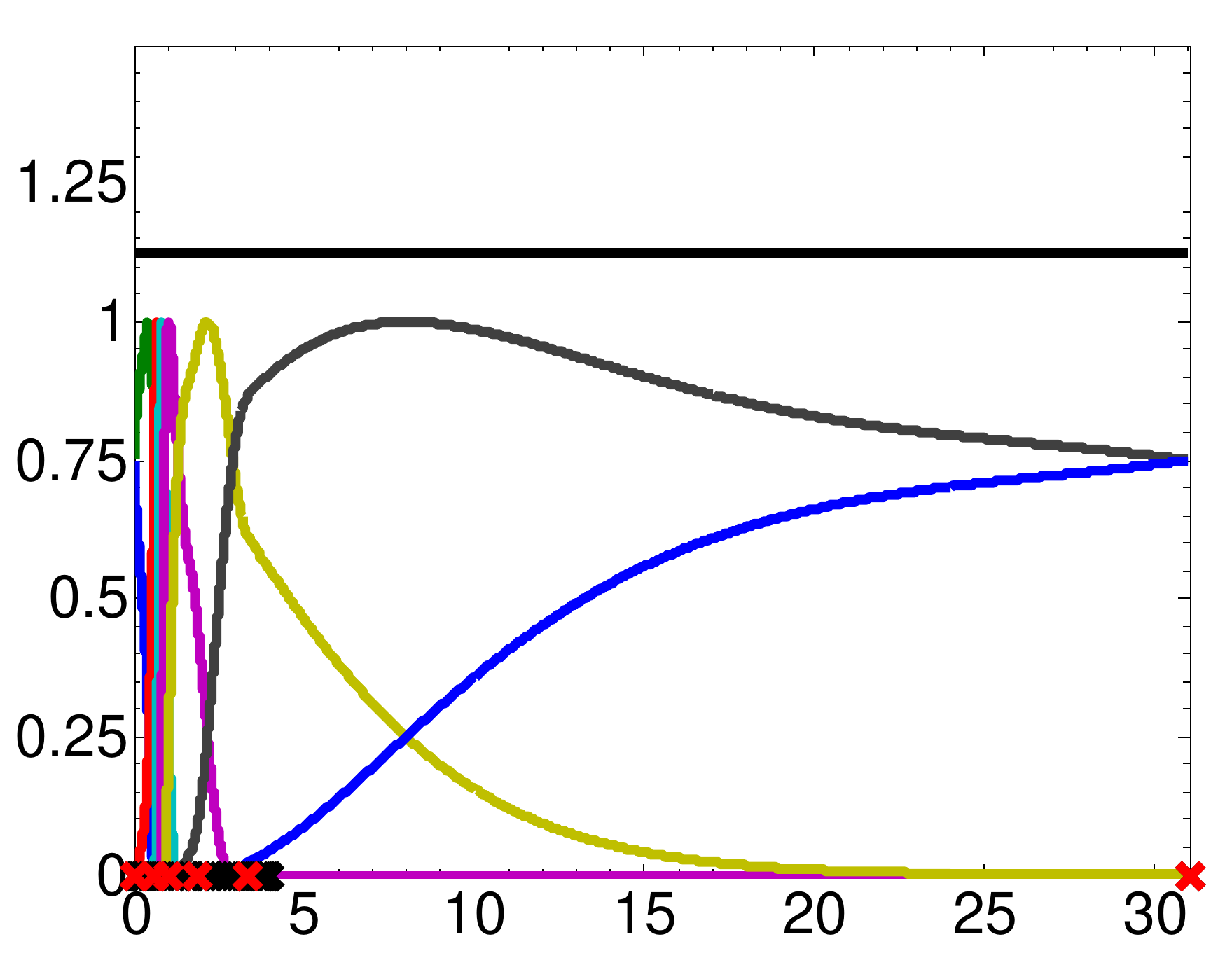}} 
\centerline{\small{~~~~$\lambda$}}
\centerline{\small{~~~~(i)}}
\end{minipage} 
\caption {(a)-(c) Spectrum-based warping functions, $\omega(\lambda)$, based on full knowledge of the graph Laplacian spectrum, $\sigma(\L)$. (d)-(f) Systems of warped filters, $\left\{\widehat{g_m}\right\}_{m=1,2,\ldots,8}$, where $\widehat{g_m}(\lambda)=\widehat{g_m^U}(\omega(\lambda))$, and $\omega(\lambda)$ are the spectrum-adapted warping functions constructed with monotonic cubic interpolation in (a)-(c).
(g)-(i) Systems of warped filters arising from a warping function generated by interpolating a subset of 8 of the 64 Laplacian eigenvalues. Specifically, we use the interpolation points $\left\{\left(\bar{\lambda}_{\l},\frac{\l}{\bar{N}-1}\right)\right\}_{\l=0,1,\ldots,\bar{N}-1}$, where $\bar{N}=8$ and $\bar{\sigma}(\L)=\left\{\bar{\lambda}_{\l}\right\}_{\l=0,1,\ldots,7}=\left\{\lambda_0,\lambda_9,\lambda_{18},\lambda_{27},\lambda_{36},\lambda_{45},\lambda_{54},\lambda_{63}\right\}$.
} 
 \label{Fig:pwl_mono}
\end{figure}
\vspace{-.4cm}
\begin{example} \label{Ex:pwl_mono_warp}
In Figure \ref{Fig:pwl_mono}(a)-(c), we show the warping functions generated by interpolating the points 
\eqref{Eq:interp_points}
with each of these two methods, for each of the graphs in Example \ref{Ex:eigen_loc}. We then show the resulting systems of spectrum-adapted warped filters in Figure \ref{Fig:pwl_mono}(d)-(f). 
We see that the warped filters are narrower where the eigenvalue density is higher -- each end of the spectrum for the path graph, the middle of the spectrum for the  
sensor network, and the very low end of the spectrum for the comet graph.
If we interpolate a subset of the Laplacian eigenvalues, there is not a huge discrepancy in the warping functions, and the resulting spectrum-adapted filters, shown in Figure \ref{Fig:pwl_mono}(g)-(i), are smoother.
\end{example}

\begin{remark} \label{Re:tight_comparison}
Interestingly, in the tight spectral graph wavelet frame construction of \cite{leonardi_multislice}, Leonardi and Van De Ville include a warping function of $\arccos\left(1-\frac{\lambda_{\l}}{d_{\max}}\right)$ in order that the resulting spectral graph wavelet on a ring graph coincides with the classical Meyer wavelet. When applied to the special case of a ring graph, our general warping method shown in Example \ref{Ex:pwl_mono_warp} almost exactly yields the warping function $\frac{1}{\pi}\arccos\left(1-\frac{\lambda_{\l}}{2}\right)$. However, as shown later in Example \ref{Ex:wavelet_comparison}, the two methods are quite different in general.
\end{remark}

\subsection{Approximation of the Cumulative Spectral Density Function} \label{Se:approx_spectrum}
From Example \ref{Ex:pwl_mono_warp}, we see that we do not need to compute all of the Laplacian eigenvalues to generate a warping function that provides a reasonable approximation of the cumulative spectral density function.
However, we are not aware of a scalable method to draw graph Laplacian eigenvalues according to the spectral density function of a high-dimensional graph Laplacian. For example, it is known that the Lanczos approximate eigendecomposition method does not accurately predict the density of the spectrum  \cite{kuijlaars}. In general, the problem of approximating the spectral density function of a large graph Laplacian matrix is an open question. 

For the purposes of this paper, we take a simple approach to this problem. Starting with an upper bound, $\lambda_{upper}$, on the Laplacian spectrum, we take $Q+1$ evenly spaced points on the interval $[0,\lambda_{upper}]$, and then compute the number of eigenvalues of $\L$ below each point using the spectrum slicing method of \cite[Section 3.3]{parlett}. Specifically, for every $q \in \{0,1,\ldots,Q\}$, we compute a triangular factorization $\L-\frac{q \lambda_{upper}}{Q} I = L_q \Delta_q L_q^*$, where $\Delta_q$ is a diagonal matrix and $L_q$ is a lower triangular matrix (not to be confused with the Laplacian matrix $\L$). By a corollary of Sylvester's law of inertia, the number of negative eigenvalues of the diagonal matrix $\Delta_q$ is equal to the number of  negative eigenvalues of $\L-\frac{q \lambda_{upper}}{Q} I$, and therefore equal to the number of eigenvalues of $\L$ less than $\frac{q \lambda_{upper}}{Q}$ \cite[Theorem 3.3.1]{parlett}. Finally, to form a smooth warping function that estimates the cumulative spectral density function, we once again use the monotonic cubic polynomial interpolation routine of \cite{fritsch} with the interpolation points $\left\{\left(\frac{q \lambda_{upper}}{Q},\frac{\mu_q}{N-1} \right)\right\}_{q=0,1,\ldots,Q}$, where $\mu_q$ is the number of diagonal elements of $\Delta_q$ less than zero.\footnote{We take $\mu_0=0$ and $\mu_{Q}=N-1$ without actually performing the triangular factorization for these two values of $q$.} 

To perform the triangular factorizations of $\L-\frac{q \lambda_{upper}}{Q} I$, we use the \texttt{MATLAB} LDL sparse Cholesky package, which is written by Timothy Davis and included in the SuiteSparse package \cite{ldl2}. 
To improve the computational times slightly, we modify the LDL package so that it performs the symbolic analysis step only once, since the result is the same for all $q$. We also include a sparsity preserving permutation using the \texttt{MATLAB} routine \texttt{symamd}. On a 1.8GHz Intel Core i7 MacBook Air laptop with 4GB of memory, for very sparse mesh-like graphs with mean degree 3 and the number of vertices $N$ equal to 50k, 200k, 500k, and 1m, the computation of the warping function using this method took approximately 1.5 seconds, 15 seconds, 2 minutes, and 5 minutes, respectively.

\section{Filters Adapted to Classes of Large Random Graphs}\label{Se:random}

In this section, rather than estimating the spectral density function for a deterministic graph, we consider classes of large random graphs for which the asymptotic (as the number of vertices $N$ goes to infinity) empirical spectral distribution of the graph Laplacian eigenvalues is known. For a given class of random graphs, the empirical spectral distribution of a random graph realization with $N$ vertices is a random measure; however, for certain classes of random graphs, the sequence of random measures for each possible $N$ actually converges to a deterministic probability measure \cite[Chapter 2.4]{tao_random_matrix}. We use these deterministic distributions (or some approximation of them) as the warping functions in \eqref{Eq:warp_def}.

The methods we develop in this section are not only useful when we are considering random graphs with known empirical spectral distributions; 
they are also potentially useful when 
dealing with a large deterministic graph whose structural properties (e.g., degree distribution, diameter, clustering coefficient) 
are similar to those of a particular class of random graphs. Then the 
empirical 
spectral distribution  
of that class of random graphs may be used as an approximation to the 
spectral density of the deterministic graph, in order to construct an appropriate warping function.

\subsection{Graph Laplacian Spectrum of Large Random Regular Graphs} \label{Se:random_regular}
The first class of random graphs we consider is  
that of random regular graphs. For integers $r$ and $N$ satisfying $3 \leq r < N$ and $rN$ is even, let $\G_{N,r}$ be the set of all unweighted graphs with $N$ vertices and with the degree of every vertex equal to $r$. A random regular graph is one chosen uniformly at random from the set $\G_{N,r}$ (see, e.g., \cite[Chapter 2.4]{bollobas} for more on random regular graphs).

The asymptotic behavior of the empirical spectral distribution of the graph Laplacian eigenvalues of a large random regular graph is given by the following result, which is commonly referred to as McKay's Law.
\begin{theorem}[McKay \cite{mckay}]
In the limit as the number of vertices $N$ goes to infinity, the empirical spectral distribution, $p_{\lambda,N}^{RR}(s)$ of the graph Laplacian eigenvalues of a large random regular graph with degree $r$ converges in probability to the deterministic probability  density function
\begin{align}\label{Eq:rr_density}
p_{\lambda,\infty}^{RR}(s):=\lim_{N \rightarrow \infty} p_{\lambda,N}^{RR}(s)=\frac{r\sqrt{4(r-1)-(s-r)^2}}{2\pi \left(r^2-(s-r)^2\right)}\Identity_{\left\{r-2\sqrt{r-1} \leq s \leq r+2\sqrt{r-1}\right\}},
\end{align}
and the empirical spectral cumulative distribution, $P_{\lambda,N}^{RR}(z)$, converges in probability to
\begin{align*}
P_{\lambda,\infty}^{RR}(z)&=\int_{0}^{z} p_{{\lambda},\infty}^{RR}(s)~ds \\
&=
\begin{cases}
0,&\hbox{ if } 0\leq z < r-2\sqrt{r-1} \\
\left[
\begin{array}{l}
\frac{1}{2} + \frac{r}{2\pi}\arcsin \left(\frac{z-r}{2\sqrt{r-1}}\right) \\
-\frac{r-2}{2\pi} \arctan \left(\frac{(r-2)(z-r)}{r\sqrt{4(r-1)-(z-r)^2}} \right)
\end{array}
\right]  
,&\hbox{ if } r-2\sqrt{r-1} \leq z < r+2\sqrt{r-1} \\
1, &\hbox{ if } r+2\sqrt{r-1} \leq z
\end{cases}~.
\end{align*}
\end{theorem}

Now, given a large but finite random regular graph with $N$ vertices and degree $r$, we (i) compute an upper bound $\lambda_{upper}$ on $\lambda_{\max}$, either via the power method or by using a simple upper bound such as \cite{anderson_morley}
\begin{align*}
\lambda_{\max}\leq \max_{i \sim j}\left\{d_i+d_j\right\}=2r\hbox{ for a random regular graph};
\end{align*} 
(ii) take the warping function to be the empirical spectral cumulative distribution on the interval $\left[0,\lambda_{upper}\right]$:
\begin{align}\label{Eq:rr_warp}
\omega^{RR}(z):=P_{\lambda,\infty}^{RR}(z),~z\in\left[0,\lambda_{upper}\right];
\end{align}
and (iii) take the warped filters to be of the form \eqref{Eq:warp_def}, with $\gamma=\omega^{RR}(\lambda_{upper})=1$ for the design of the uniform translates.\footnote{Note that for a random regular graph with a finite number of vertices, $\lambda_{\max}$ may be greater than $r+2\sqrt{r-1}$. Thus, in order to ensure that $\omega^{RR}(\cdot)$ is well-defined on the entire spectrum $\sigma(\L)$ of the random graph realization, we need to only restrict the empirical spectral cumulative distribution to $\left[0,\lambda_{upper}\right]$, rather than $\left[0,r+2\sqrt{r-1}\right]$.}

\begin{example}
We choose a realization from the class of random regular graphs with $N=3000$ vertices and degree $r=3$, and take $\lambda_{upper}=2r=6$. In Figure \ref{Fig:rr}, we compare the normalized histogram of the graph Laplacian eigenvalues to the expected asymptotic spectral density, $p_{\lambda,\infty}^{RR}(s)$, for large $N$, and plot the warping function, $\omega^{RR}(\lambda)$, and warped system of filters. Note that while the warping function is adapted to the class of random regular graphs with degree 3, the actual graph Laplacian eigenvalues shown in Figure \ref{Fig:rr}(a) are not used to construct the warping function in Figure \ref{Fig:rr}(b) or the warped filters in Figure \ref{Fig:rr}(c). Therefore, we can apply the same set of filters to any realization of a random regular graph of degree 3 with a larger number of vertices $N$ without running into scalability issues.
\begin{figure}[h]
\centering
\begin{minipage}[b]{.32\linewidth}
\centerline{~~~Empirical Spectral Distribution}
\centerline{\includegraphics[width=\linewidth]{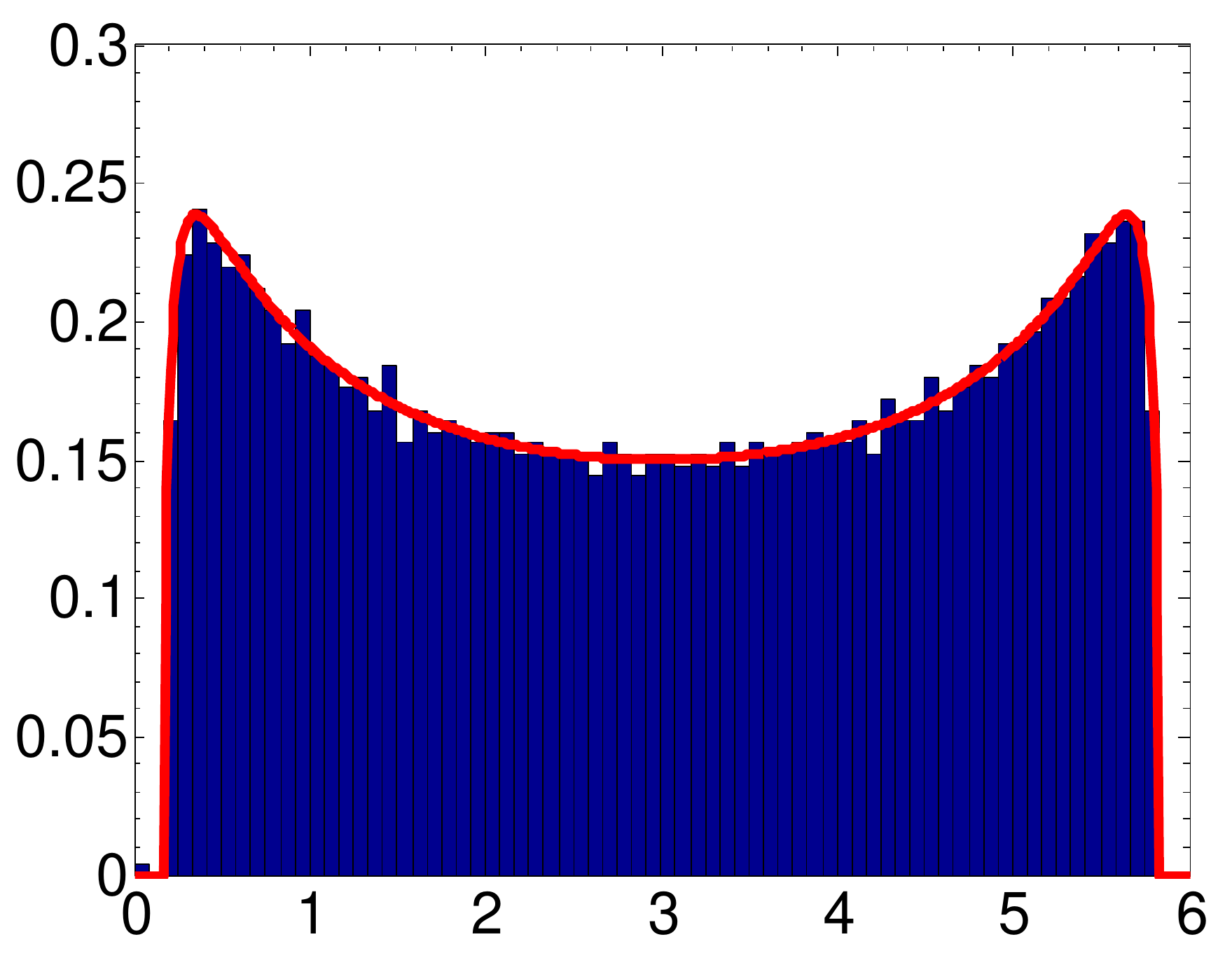}} 
\centerline{\small{~~~$\lambda$}}
\centerline{\small{~~~(a)}}
\end{minipage}
\hfill
\begin{minipage}[b]{.318\linewidth}
\centerline{~~~Warping Function}
\centerline{\includegraphics[width=\linewidth]{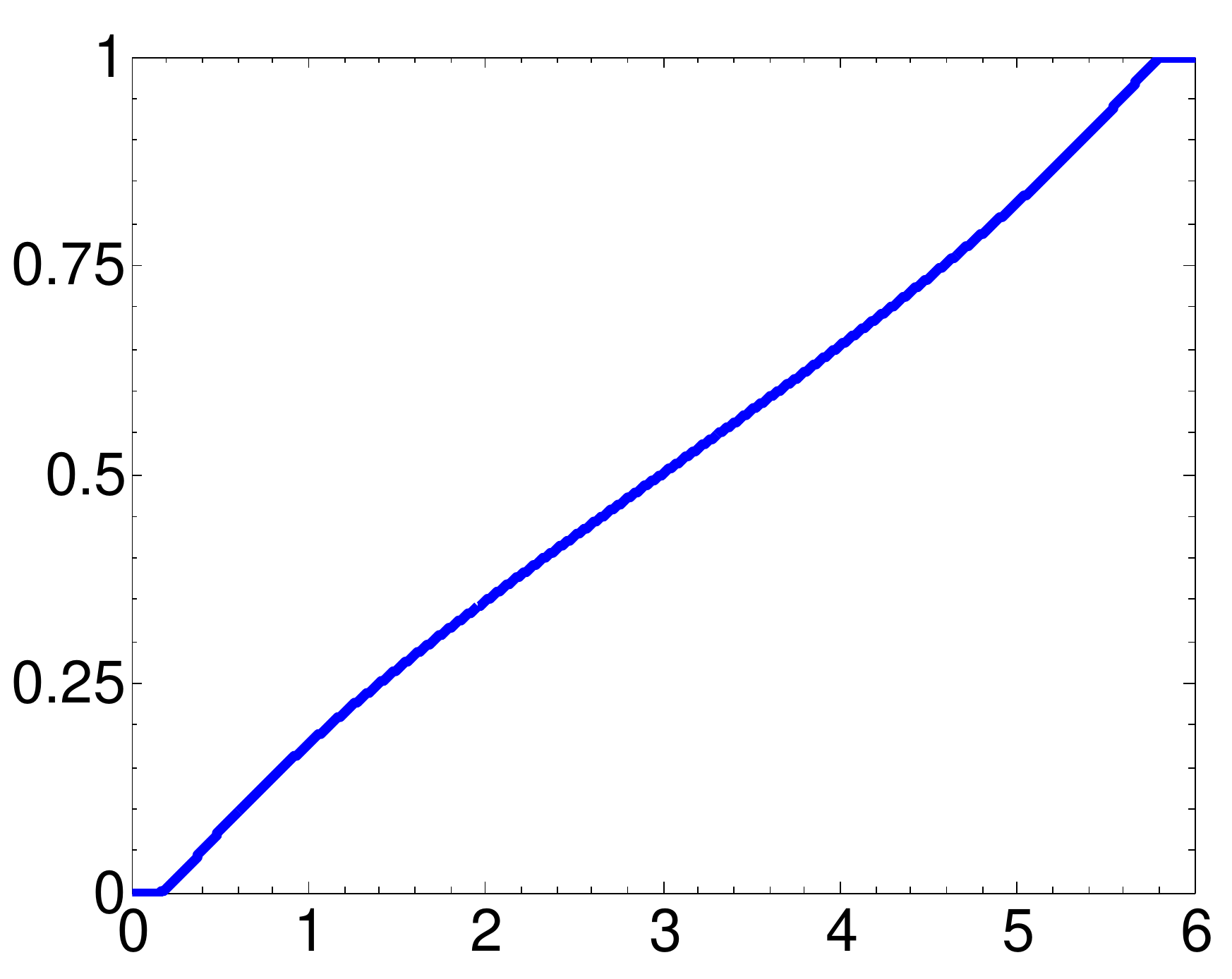}} 
\centerline{\small{~~~$\lambda$}}
\centerline{\small{~~~(b)}}
\end{minipage}
\hfill
\begin{minipage}[b]{.32\linewidth}
\centerline{~~~Warped Filters}
\centerline{\includegraphics[width=\linewidth]{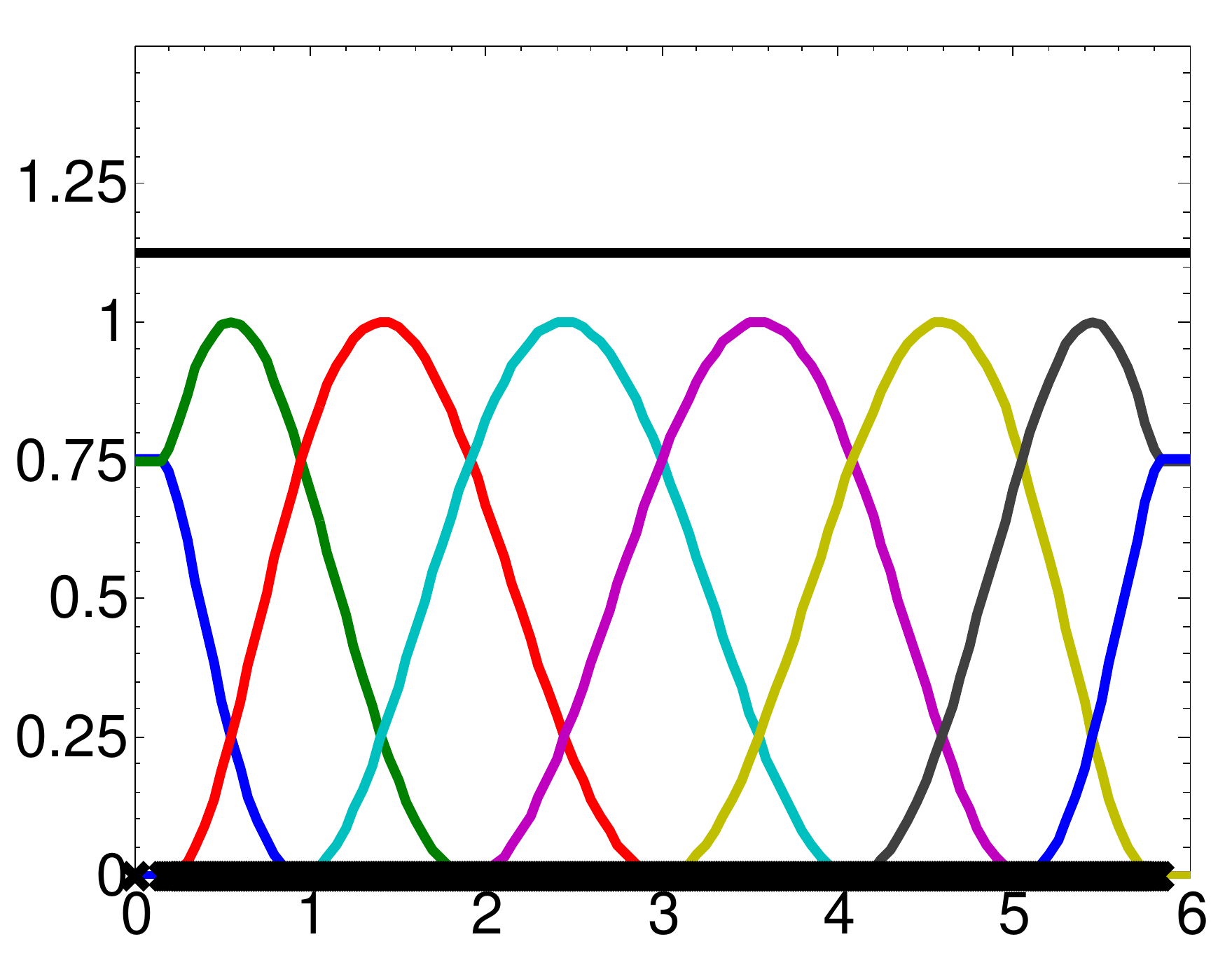}} 
\centerline{\small{~~~$\lambda$}}
\centerline{\small{~~~(c)}}
\end{minipage} 
\caption {Construction of a system of filters adapted to the graph Laplacian spectrum of the class of random regular graphs with degree $r=3$. (a) The normalized histogram of the graph Laplacian eigenvalues of a single realization of a random regular graph with $N=3000$ vertices, compared to the asymptotic empirical spectral distribution $p_{\lambda,\infty}^{RR}(s)$ in \eqref{Eq:rr_density}. (b) The warping function $\omega^{RR}(\lambda)$ defined in \eqref{Eq:rr_warp}. (c) The resulting system of warped filters. The black marks on the horizontal axis represent the eigenvalues of the single realization. 
While the filters are not adapted to that specific realization, they are narrower in the regions of the spectrum where the eigenvalue density is higher.} 
 \label{Fig:rr}
\end{figure}
\end{example}

\subsection{Graph Laplacian Spectrum of Erd\H{o}s-R{\'e}nyi Random Graphs}
In the Erd\H{o}s-R{\'e}nyi $G(N,p)$ random graph model \cite{erdos_renyi}, \cite[Chapter 5]{chung_complex}, an edge connects each possible pair of the $N$ vertices with probability $p$, with $0<p<1$; that is, for all $i,j \in \{1,2,\ldots,N\}$ with $i \neq j$, 
\begin{align*}
W_{ij}=W_{ji}=
\begin{cases}
1,&\hbox{with probability } p \\
0,&\hbox{with probability } 1-p
\end{cases}
,
\end{align*}
independently of $W_{i^{\prime}j^{\prime}}$ for $(i^{\prime},j^{\prime}) \neq (i,j)$. The following theorem of Ding and Jiang characterizes the asymptotic empirical spectral cumulative distribution of the graph Laplacian eigenvalues of Erd\H{o}s-R{\'e}nyi random graphs.
\begin{theorem}[Ding and Jiang, Theorem 2, \cite{ding}]
In the limit as the number of vertices $N$ goes to infinity, 
with probability one, 
the shifted and scaled empirical spectral cumulative distribution 
\begin{align}\label{Eq:ER_shifted_scaled}
\bar{P}_{{\lambda},N}^{ER}(z):=\frac{1}{N} \sum_{\l=0}^{N-1} \Identity_{\left\{\frac{\lambda_{\l}-pN}{\sqrt{pN(1-p)}}\leq z\right\}}
\end{align}
of the graph Laplacian eigenvalues of a large random  Erd\H{o}s-R{\'e}nyi graph with edge probability $p$ converges weakly to the  
measure $\mu=\mu_A \boxplus \mu_B$,  
the free additive convolution\footnote{For more details about free probability theory and the free additive convolution, see \cite[Chapter 2.5]{tao_random_matrix} or \cite{voiculescu}.} of the standard normal distribution with density
\begin{align*}
d\mu_A:=\frac{1}{\sqrt{2\pi}}e^{\frac{-x^2}{2}}~dx,
\end{align*} 
and the semi-circular distribution with 
density
\begin{align}\label{Eq:semicircular}
d\mu_B:=\frac{1}{2\pi}\sqrt{4-x^2}~\Identity_{\left\{-2 \leq x \leq 2\right\}}~dx.
\end{align} 
\end{theorem}

Given a large but finite Erd\H{o}s-R{\'e}nyi random graph with $N$ vertices and edge probability $p$, we can approximate the empirical spectral cumulative distribution by rearranging \eqref{Eq:ER_shifted_scaled} to get
\begin{align}\label{Eq:ER_cdf}
P_{{\lambda},N}^{ER}(z):=\sqrt{\frac{1}{pN(1-p)}}~\mu\left(\left(-\infty,\frac{z-pN}{\sqrt{pN(1-p)}}\right]\right)
=\sqrt{\frac{1}{pN(1-p)}}~\int_{-\infty}^{z} d\mu\left(\frac{s-pN}{\sqrt{pN(1-p)}}\right),
\end{align}
and then proceed as in Section \ref{Se:random_regular} with $\omega^{ER}(z):=P_{\lambda,N}^{ER}(z)$ for $z\in\left[0,\lambda_{upper}\right]$.
We should comment on a few technical issues. First, as mentioned earlier, for a fixed $N$, the empirical spectral cumulative distribution is a random measure, while the sequence of distributions converges to a deterministic measure asymptotically as $N$ increases. Nonetheless, we are taking the deterministic approximation \eqref{Eq:ER_cdf} as the warping function. Second, in general, computing free convolutions is non-trivial. In Example \ref{Ex:er} below, we use the numerical computational method presented in \cite{olver} to compute the density $d\mu$ in \eqref{Eq:ER_cdf}. Third, the support of the density function of the free convolution of the standard normal distribution and the semi-circular distribution is the entire real line. Therefore, unlike the case of the random regular graph above, $\omega^{ER}(0)$ is not exactly equal to zero; however, for large $N$, it is quite small (e.g., on the order of $\frac{1}{1000}$ for Example \ref{Ex:er} below). Another consequence of the non-compact support of $p_{{\lambda},N}^{ER}(s)$ is that we cannot choose a strict upper bound $\lambda_{upper}$. Rather, for any given $\epsilon>0$, we can choose a $\lambda_{upper}$ such that the probability that an eigenvalue is bigger than $\lambda_{upper}$ is less than $\epsilon$.

\begin{example} \label{Ex:er}
We choose a realization from the class of Erd\H{o}s-R{\'e}nyi random graphs with $N=3000$ vertices and edge probability $p=0.05$, and take $\lambda_{upper}=pN+4\sqrt{pN(1-p)}=197.75$. In Figure \ref{Fig:er}, we compare the normalized histogram of the graph Laplacian eigenvalues to $p_{\lambda,N}^{ER}(s)$
and plot the warping function and warped system of filters.

\begin{figure}[h]
\centering
\begin{minipage}[b]{.325\linewidth}
\centerline{~~~Empirical Spectral Distribution}
\centerline{\includegraphics[width=\linewidth]{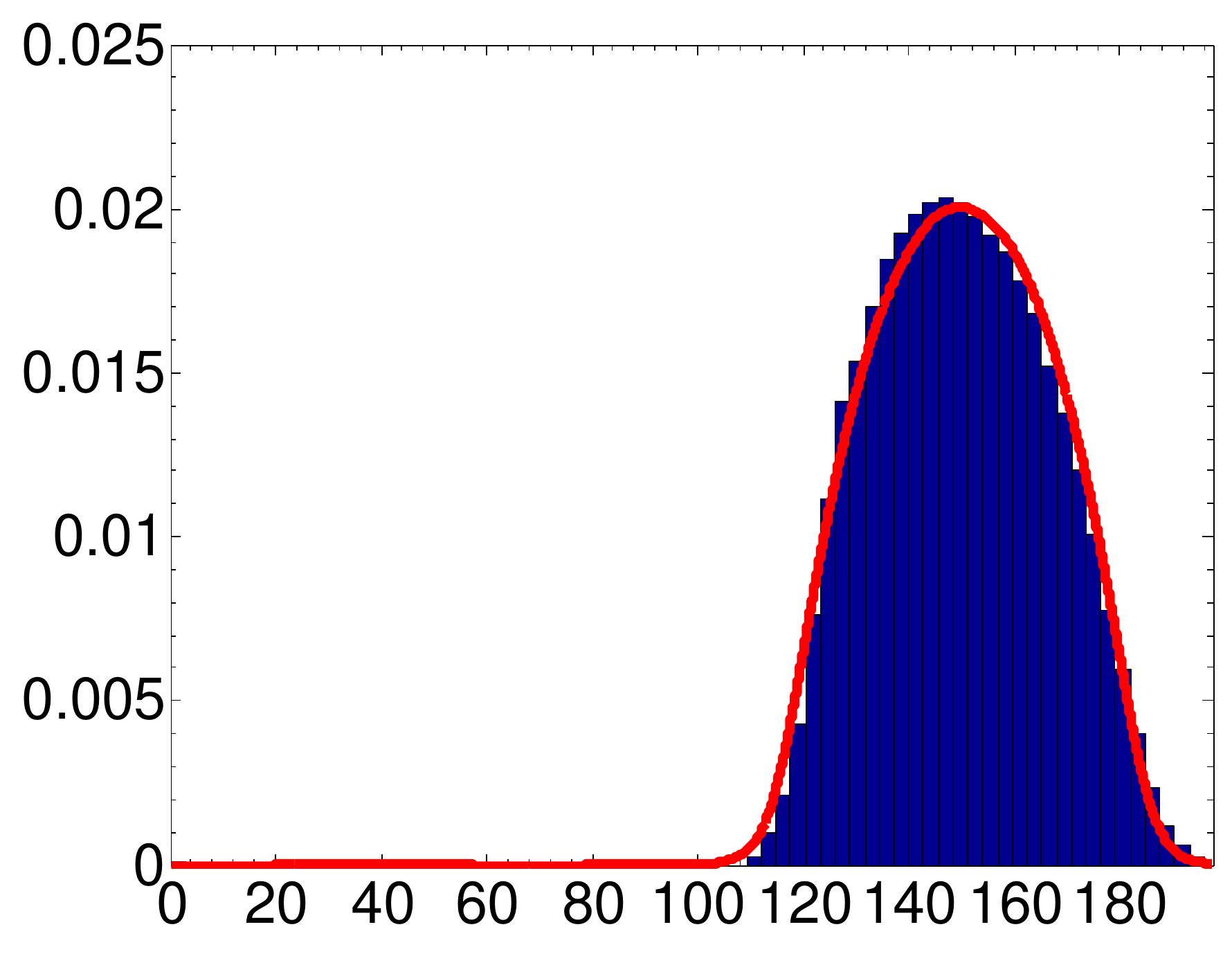}} 
\centerline{\small{~~~$\lambda$}}
\centerline{\small{~~~(a)}}
\end{minipage}
\hfill
\begin{minipage}[b]{.318\linewidth}
\centerline{~~~Warping Function}
\centerline{\includegraphics[width=\linewidth]{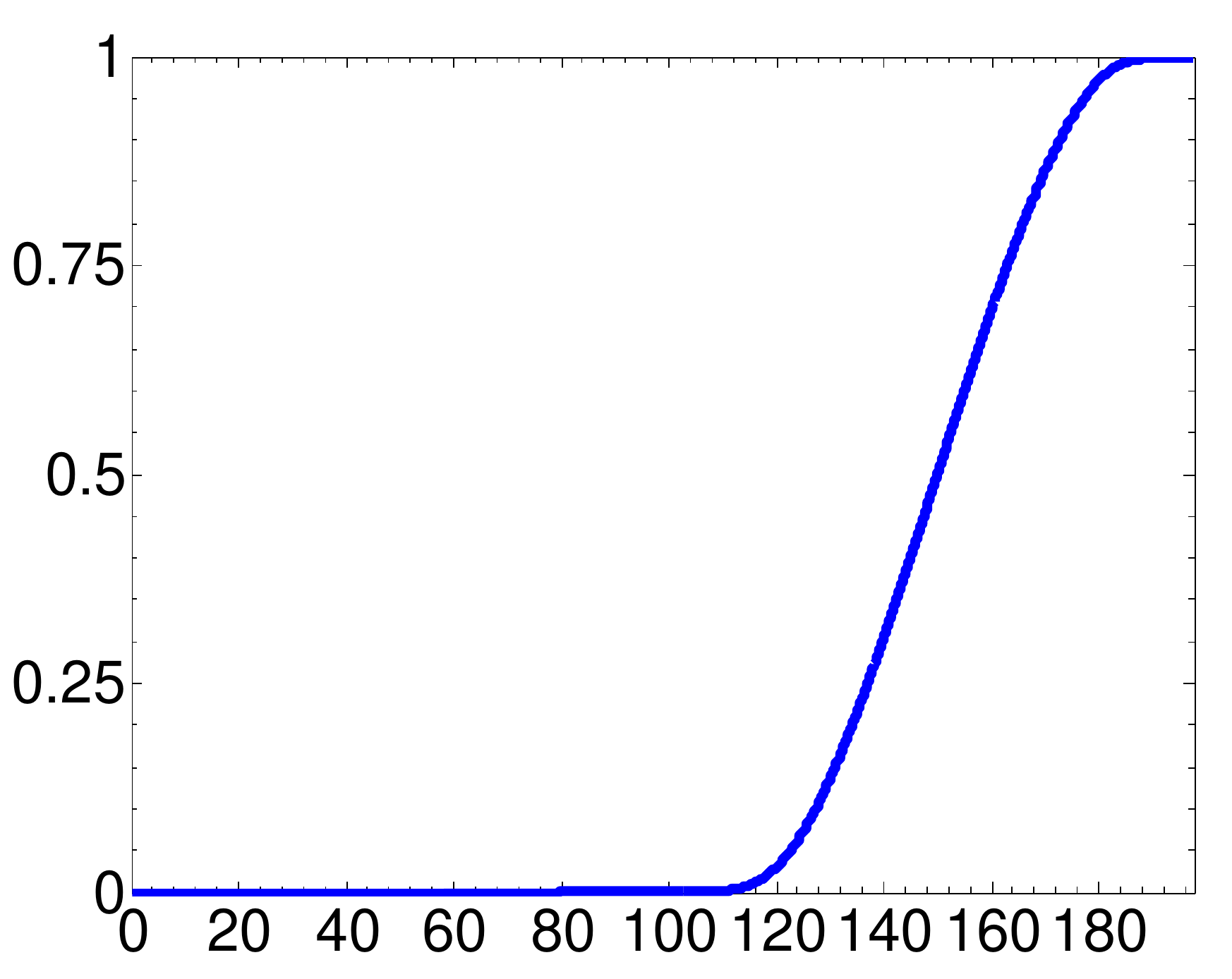}} 
\centerline{\small{~~~$\lambda$}}
\centerline{\small{~~~(b)}}
\end{minipage}
\hfill
\begin{minipage}[b]{.322\linewidth}
\centerline{~~~Warped Filters}
\centerline{\includegraphics[width=\linewidth]{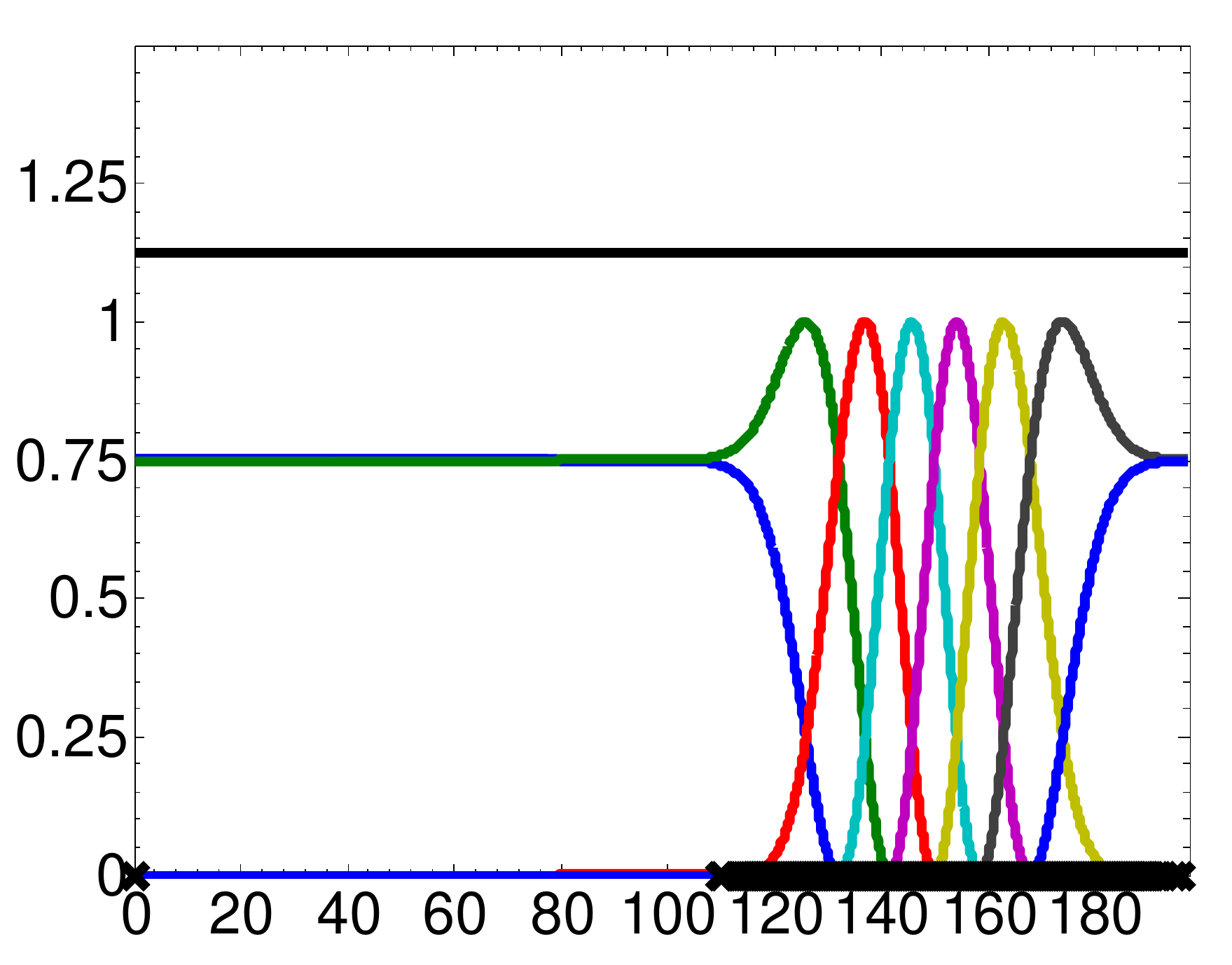}} 
\centerline{\small{~~~$\lambda$}}
\centerline{\small{~~~(c)}}
\end{minipage} 
\caption {Construction of a system of filters adapted to the graph Laplacian spectrum of the class of Erd\H{o}s-R{\'e}nyi random graphs with $N=3000$ vertices and edge probability $p=0.05$. (a) The normalized histogram of the graph Laplacian eigenvalues of a single graph realization from this class, compared to the approximate empirical spectral distribution $p_{\lambda,N}^{ER}(s)$. (b) The warping function $\omega^{ER}(\lambda)$ defined in \eqref{Eq:ER_cdf}. (c) The resulting system of warped filters. Once again, 
the filters, although not adapted to that specific realization, are narrower in the regions of the spectrum where the eigenvalue density is higher.} 
 \label{Fig:er}
\end{figure}
\end{example}

\subsection{Normalized Graph Laplacian Spectrum of Erd\H{o}s-R{\'e}nyi Random Graphs} \label{Se:norm_ER}
As discussed in \cite{shuman_SPM}, it may be beneficial to use the normalized graph Laplacian eigenvectors as a graph spectral filtering basis in some applications. Therefore, we continue to consider Erd\H{o}s-R{\'e}nyi random graphs, and now derive filters adapted to the normalized graph Laplacian spectrum $\sigma(\tilde{L})$. The asymptotic behavior of the empirical spectral cumulative distribution of these eigenvalues is characterized in the following theorem.

\begin{theorem}[Fan, Lu, and Vu, Theorem 6, \cite{chung_power_law} and Jiang, Corollary 1.3, \cite{jiang}]
In the limit as the number of vertices $N$ goes to infinity, 
with probability one, 
the shifted and scaled empirical spectral cumulative distribution 
\begin{align*}
\bar{P}_{{\tilde{\lambda}},N}^{ER}(z):=\frac{1}{N} \sum_{\l=0}^{N-1} \Identity_{\left\{\sqrt{\frac{pN}{(1-p)}}\left(1-\tilde{\lambda}_{\l}\right)\leq z\right\}}
\end{align*}
of the normalized graph Laplacian eigenvalues of a large random  Erd\H{o}s-R{\'e}nyi graph with edge probability $p$ converges weakly to the 
semi-circular distribution 
\eqref{Eq:semicircular}.
\end{theorem}
 
We want to take the warping function for a random graph with $N$ vertices to be the (deterministic) approximate empirical spectral cumulative distribution. Substituting $x=\sqrt{\frac{pN}{1-p}}(1-s)$ into \eqref{Eq:semicircular} yields
\begin{align}\label{Eq:er_norm_density}
p_{\tilde{\lambda},N}^{ER}(s)
=\frac{1}{2\pi} \sqrt{\frac{pN}{1-p}} \sqrt{4-\frac{pN}{1-p}(1-s)^2}~\Identity_{\left\{1-2\sqrt{\frac{1-p}{pN}}\leq s \leq 1+2\sqrt{\frac{1-p}{pN}}\right\}}.
\end{align}
Integrating \eqref{Eq:er_norm_density}, we find that for large $N$, the empirical spectral cumulative distribution is approximately
\begin{align*} 
P_{\tilde{\lambda},N}^{ER}(z)&=\int_{0}^{z} p_{\tilde{\lambda},N}^{ER}(s)~ds \\
&=
\begin{cases}
0,&\hbox{ if } 0\leq z < 1-2\sqrt{\frac{1-p}{pN}} \\
\left[
\begin{array}{l}
\pi \sqrt{\frac{1-p}{pN}} + \left(\frac{z-1}{2}\right) \sqrt{4-\frac{pn}{1-p}(1-z)^2}\\
-2 \sqrt{\frac{1-p}{pN}} \arcsin\left(\sqrt{\frac{pN}{1-p}}\left(\frac{1-z}{2}\right)\right)
\end{array}
\right]  
,&\hbox{ if } 1-2\sqrt{\frac{1-p}{pN}} \leq z < 1+2\sqrt{\frac{1-p}{pN}} \\
1, &\hbox{ if } 1+2\sqrt{\frac{1-p}{pN}} \leq z \leq 2
\end{cases}~, \nonumber
\end{align*}
where we 
use formulas from \cite[Section 2.26, pp. 94-95]{gradshteyn} to evaluate the integral. 

Thus, given a large but finite Erd\H{o}s-R{\'e}nyi random graph with $N$ vertices and edge probability $p$, we proceed as in Section \ref{Se:random_regular}, with $\tilde{\lambda}_{upper}$ either computed more precisely or simply set to 2, and 
\begin{align}\label{Eq:er_norm_warp}
\tilde{\omega}^{ER}(z):=P_{\tilde{\lambda},\infty}^{ER}(z),~z\in\left[0,\tilde{\lambda}_{upper}\right].
\end{align}

\begin{example} \label{Ex:ern}
We now consider the same class of Erd\H{o}s-R{\'e}nyi random graphs and specific graph realization as in Example \ref{Ex:er}, but we adapt the filters to the normalized graph Laplacian spectrum. We use the trivial upper bound $\lambda_{upper}=2$. 
Figure \ref{Fig:ern} shows 
the approximate empirical spectral distribution, warping function, and resulting system of warped filters.
\end{example}
\begin{figure}[h]
\centering
\begin{minipage}[b]{.32\linewidth}
\centerline{~~~Empirical Spectral Distribution}
\centerline{\includegraphics[width=\linewidth]{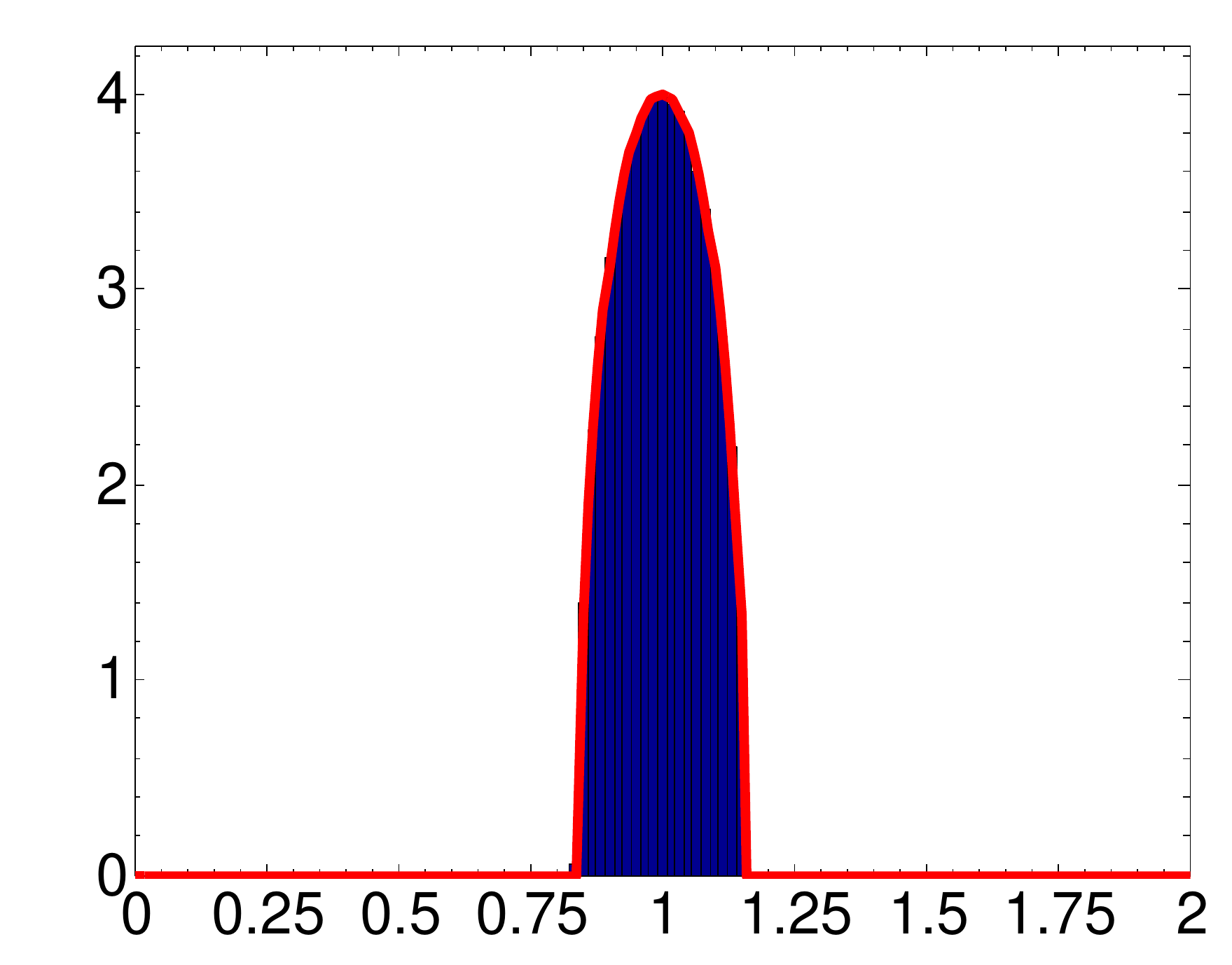}} 
\centerline{\small{~~~$\lambda$}}
\centerline{\small{~~~(a)}}
\end{minipage}
\hfill
\begin{minipage}[b]{.318\linewidth}
\centerline{~~~Warping Function}
\centerline{\includegraphics[width=\linewidth]{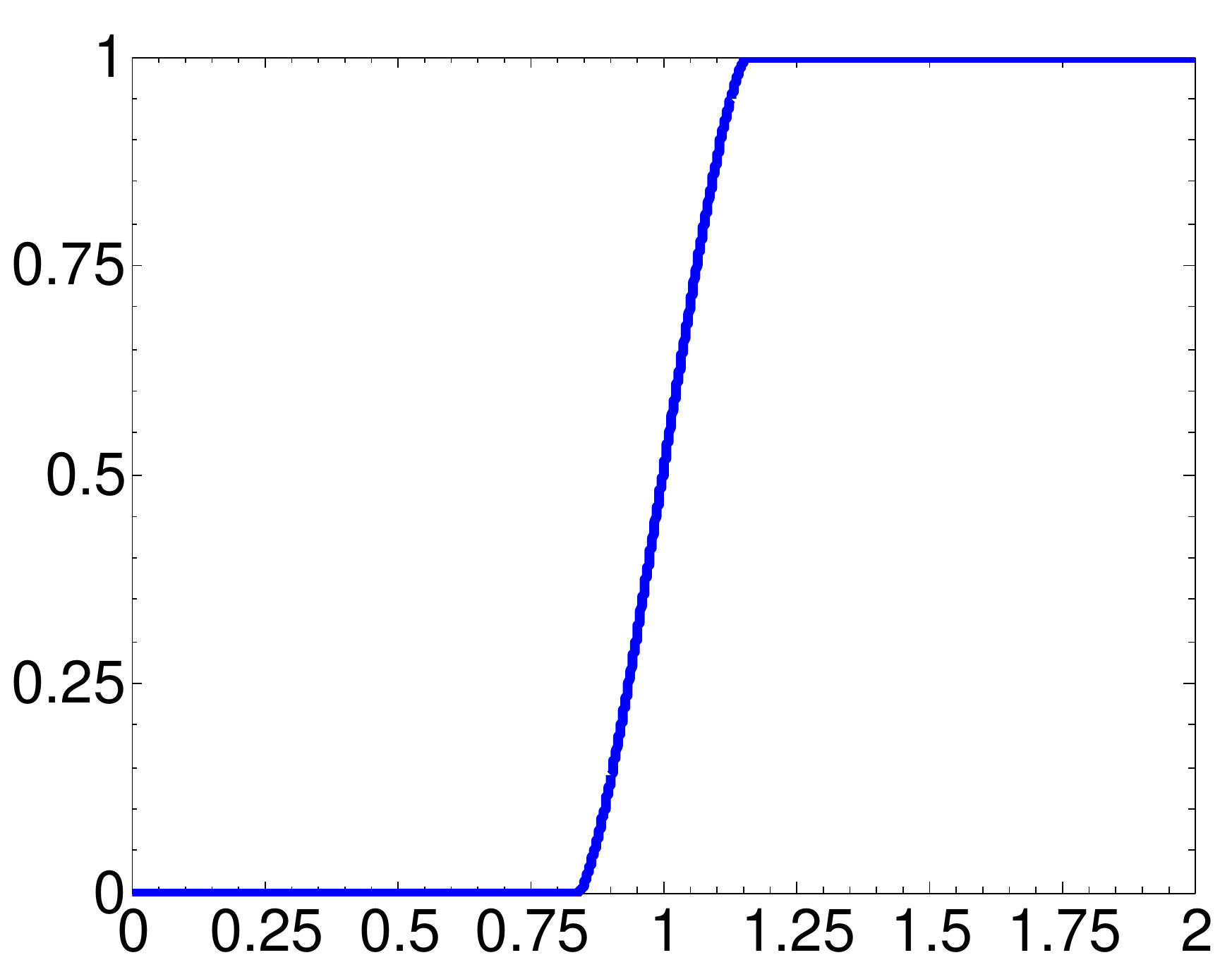}} 
\centerline{\small{~~~$\lambda$}}
\centerline{\small{~~~(b)}}
\end{minipage}
\hfill
\begin{minipage}[b]{.32\linewidth}
\centerline{~~~Warped Filters}
\centerline{\includegraphics[width=\linewidth]{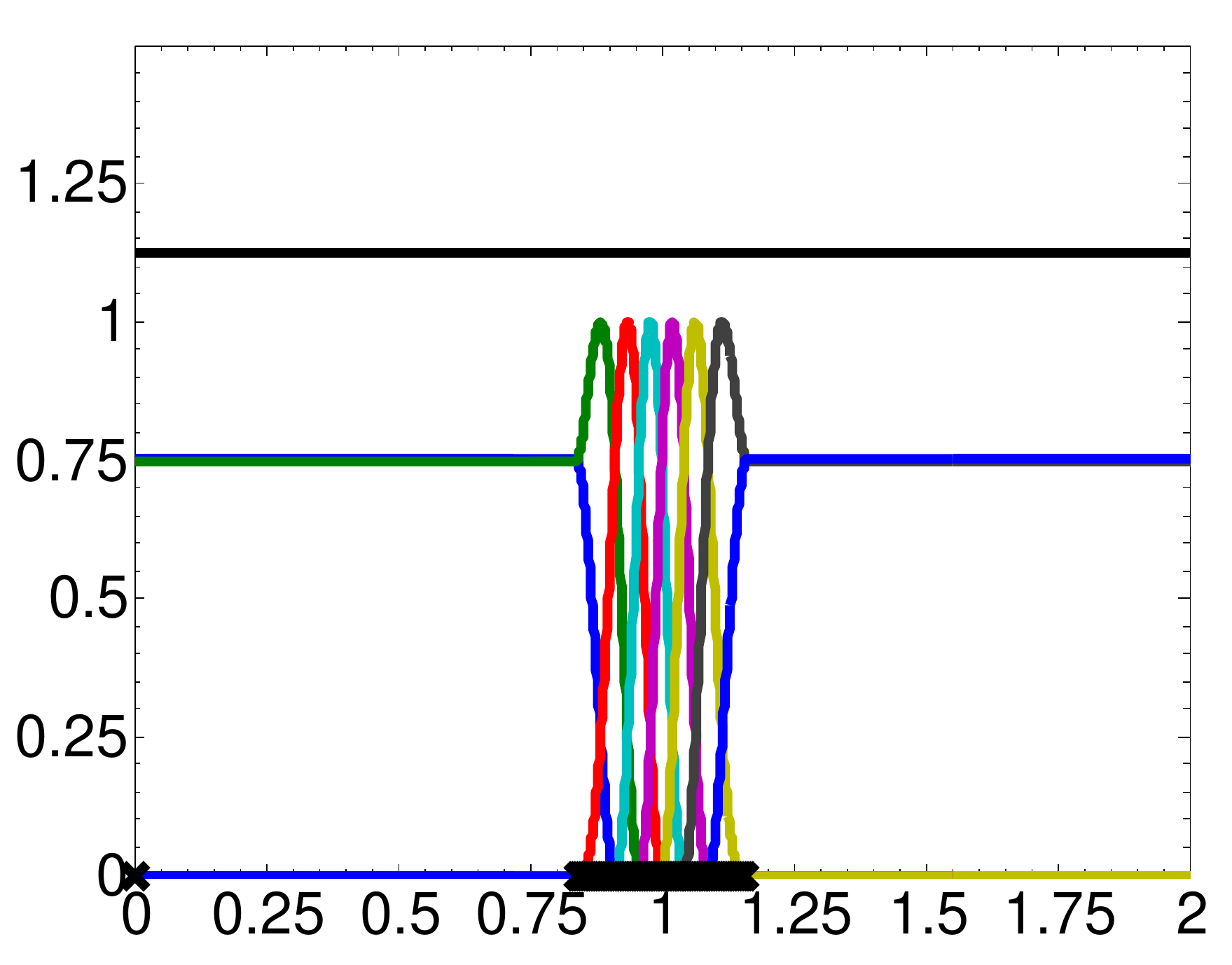}} 
\centerline{\small{~~~$\lambda$}}
\centerline{\small{~~~(c)}}
\end{minipage} 
\caption {Construction of a system of filters adapted to the normalized graph Laplacian spectrum of the class of Erd\H{o}s-R{\'e}nyi random graphs with $N=3000$ vertices and edge probability $p=0.05$. (a) The normalized histogram of the normalized graph Laplacian eigenvalues of a single graph realization from this class, compared to the approximate empirical spectral distribution $p_{\tilde{\lambda},N}^{ER}(s)$ given in \eqref{Eq:er_norm_density}. (b) The warping function $\tilde{\omega}^{ER}(\tilde{\lambda})$ defined in \eqref{Eq:er_norm_warp}. (c) The 
system of warped filters.}
 \label{Fig:ern}
\end{figure}

\section{Spectrum-Adapted Tight Graph Wavelet Frames}\label{Se:wavelets}
We can now combine the logarithmic warping from Section \ref{Se:log_wavelet} with the spectrum-adapted warping functions from Sections \ref{Se:adapted} and \ref{Se:random} to generate spectrum-adapted tight wavelet frames. Namely, we take the warping function to be 
\begin{align}\label{Eq:composite_warp}
\omega(\lambda):=\log\bigl(\omega_0(\lambda)
\bigr),
\end{align}
where $\omega_0(\cdot)$ is a normalizing constant times some approximation of the empirical spectral cumulative distribution. 
Then we can once again generate the wavelet and scaling kernels according to \eqref{Eq:logwarp0} and \eqref{Eq:logwarp1}.

\begin{example}\label{Ex:wavelet_comparison}
We consider the same class of Erd\H{o}s-R{\'e}nyi random graphs from Example \ref{Ex:er}, and take $\omega_0(\lambda)=\lambda_{upper}\cdot\omega^{ER}(\lambda)$. In Figure \ref{Fig:wavelet_comparison}, we compare the wavelet and scaling kernels generated from the spectral graph wavelet transform, Meyer-like tight wavelet frames, and log-warped tight wavelet frame from Section \ref{Se:log_wavelet} to the warped filters generated by the composite warping function \eqref{Eq:composite_warp}.
\end{example}

In Example \ref{Ex:eigen_loc}, we saw that a filter whose support does not overlap any Laplacian eigenvalues leads to atoms with zero norm, which are not helpful in analysis. More generally, it is desirable that the wavelet atoms are not too correlated with each other. To quantify 
these correlations, we examine the cumulative coherence function \cite{tropp}, which, for a given sparsity level $k$ is defined as 
\begin{align*}
\mu_1(k):=\max_{|\Theta|=k}~\max_{\psi \in {\cal D}_{\{1,2,\ldots,N\cdot M\}\setminus \Theta}}~\sum_{\theta \in \Theta} \frac{|\ip{\psi}{{\cal D}_{\theta}}|}{\norm{\psi}_2\norm{{\cal D}_{\theta}}_2}.
\end{align*}
In Table \ref{Ta:cumulative_coherences}, we compare the cumulative coherences for different graph wavelet constructions. When 
an atom has a norm of 0, we list the cumulative 
coherence as N/A. We also show $\sigma_{\norm{g_{i,m}}}$, the standard deviation of the norms of the atoms of each dictionary. The four graphs have $N=256,500,64,$ and $1000$ vertices, respectively. We see that in all cases, the spectrum-adapted tight wavelet frame has the smallest cumulative coherence and standard deviation of the atom norms.

\begin{table}[htb]
{\footnotesize
\hspace{-.3in}
\tabcolsep=0.11cm
\begin{tabular}{l|ccc|ccc|ccc|ccc|}
\cline{2-13}
 & \multicolumn{3}{ c| }{Path Graph} & \multicolumn{3}{ c| }{Sensor Network} & \multicolumn{3}{ c| }{Comet Graph} & \multicolumn{3}{ c| }{Random Erd\H{o}s-R{\'e}nyi}\\ 
\cline{2-13}
& $\mu_1(\sqrt{N})$ & $\mu_1(N)$ & $\sigma_{\norm{g_{i,m}}}$ & $\mu_1(\sqrt{N})$ & $\mu_1(N)$ & $\sigma_{\norm{g_{i,m}}}$ & $\mu_1(\sqrt{N})$ & $\mu_1(N)$ & $\sigma_{\norm{g_{i,m}}}$ & $\mu_1(\sqrt{N})$ & $\mu_1(N)$ & $\sigma_{\norm{g_{i,m}}}$ \\ 
\cline{1-13}
\multicolumn{1}{ |l| }{Spectral Graph} & 13.3 & 48.0 & 0.18 & 21.7 & 139.4 & 0.33 & 8.0 & 63.5 & 0.38 & 32.0 & 999.0 & 0.43 \\ 
\cline{1-13}
\multicolumn{1}{ |l| }{Meyer-Like} & 15.5 & 70.1 & 0.14 & 21.9 & 178.8 & 0.25 & N/A & N/A & 0.28 & N/A & N/A & 0.31 \\ 
\cline{1-13}
\multicolumn{1}{ |l| }{Degree-Adapted Meyer} & 16.0 & 130.2 & 0.18 & N/A & N/A & 0.27 & N/A & N/A & 0.27 & N/A & N/A & 0.31 \\ 
\cline{1-13}
\multicolumn{1}{ |l| }{Log-Warped} & 13.3 & 43.7 & 0.12 & 21.6 & 138.1 & 0.24 & N/A & N/A & 0.28 & N/A & N/A & 0.32 \\
\cline{1-13}
\multicolumn{1}{ |l| }{Spectrum-Adapted} & 12.9 & 34.0 & 0.10 & 21.6 & 127.0 & 0.23 & 8.0 & 55.2 & 0.25 & 31.7 & 829.5 & 0.25 \\
\cline{1-13} 
\end{tabular}
}
\caption{Comparison of the cumulative coherences of the normalized dictionary atoms of five different graph wavelet frames adapted to four different graphs, for sparsity levels $\sqrt{N}$ (rounded to the nearest integer) and $N$.} 
\label{Ta:cumulative_coherences}
\end{table}
\begin{figure}[htb]
\centering
\begin{minipage}[b]{.32\linewidth}
\centerline{~~~Spectral Graph}
\centerline{~~~Wavelet Frame \cite{sgwt}}
\centerline{\includegraphics[width=\linewidth]{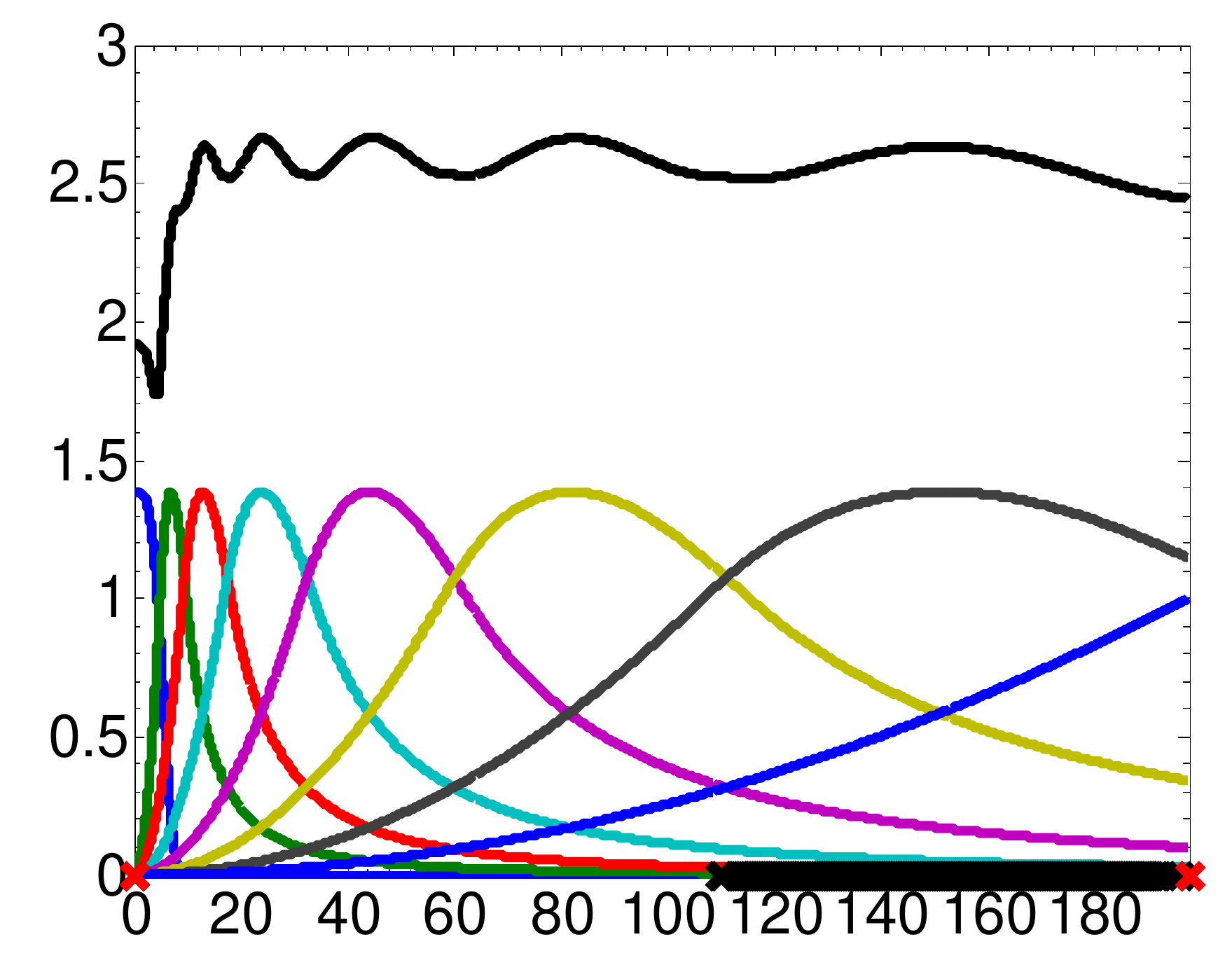}} 
\centerline{\small{~~~~$\lambda$}}
\centerline{\small{~~~~(a)}}
\end{minipage}
\hfill
\begin{minipage}[b]{.32\linewidth}
\centerline{~~~Meyer-Like Tight}
\centerline{~~~Graph Wavelet Frame \cite{leonardi_fmri,leonardi_multislice}}
\centerline{\includegraphics[width=\linewidth]{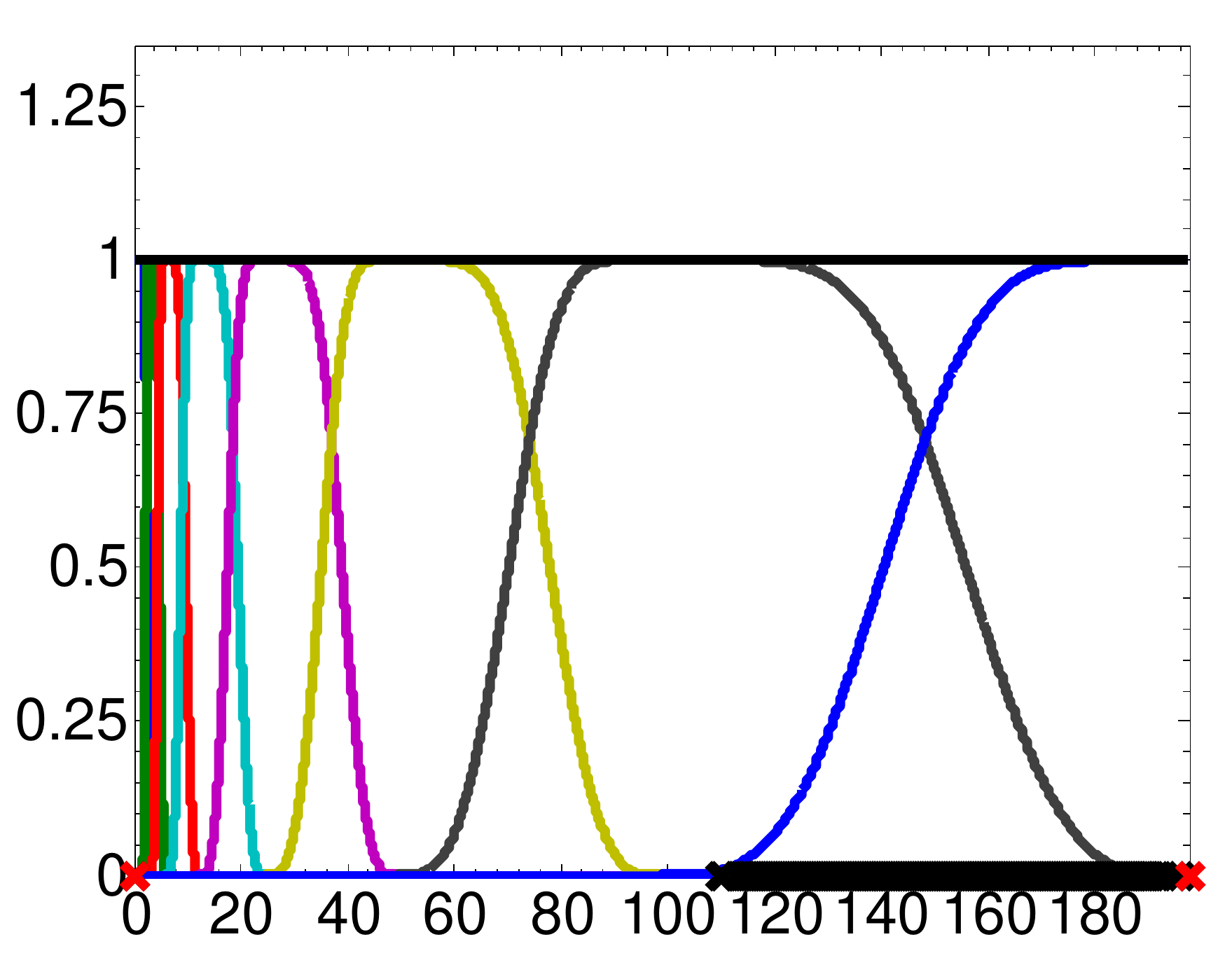}} 
\centerline{\small{~~~~$\lambda$}}
\centerline{\small{~~~~(b)}}
\end{minipage}
\hfill
\begin{minipage}[b]{.32\linewidth}
\centerline{~~~Max Degree-Adapted Meyer-Like}
\centerline{~~~Tight Graph Wavelet Frame \cite{leonardi_multislice}}
\centerline{\includegraphics[width=\linewidth]{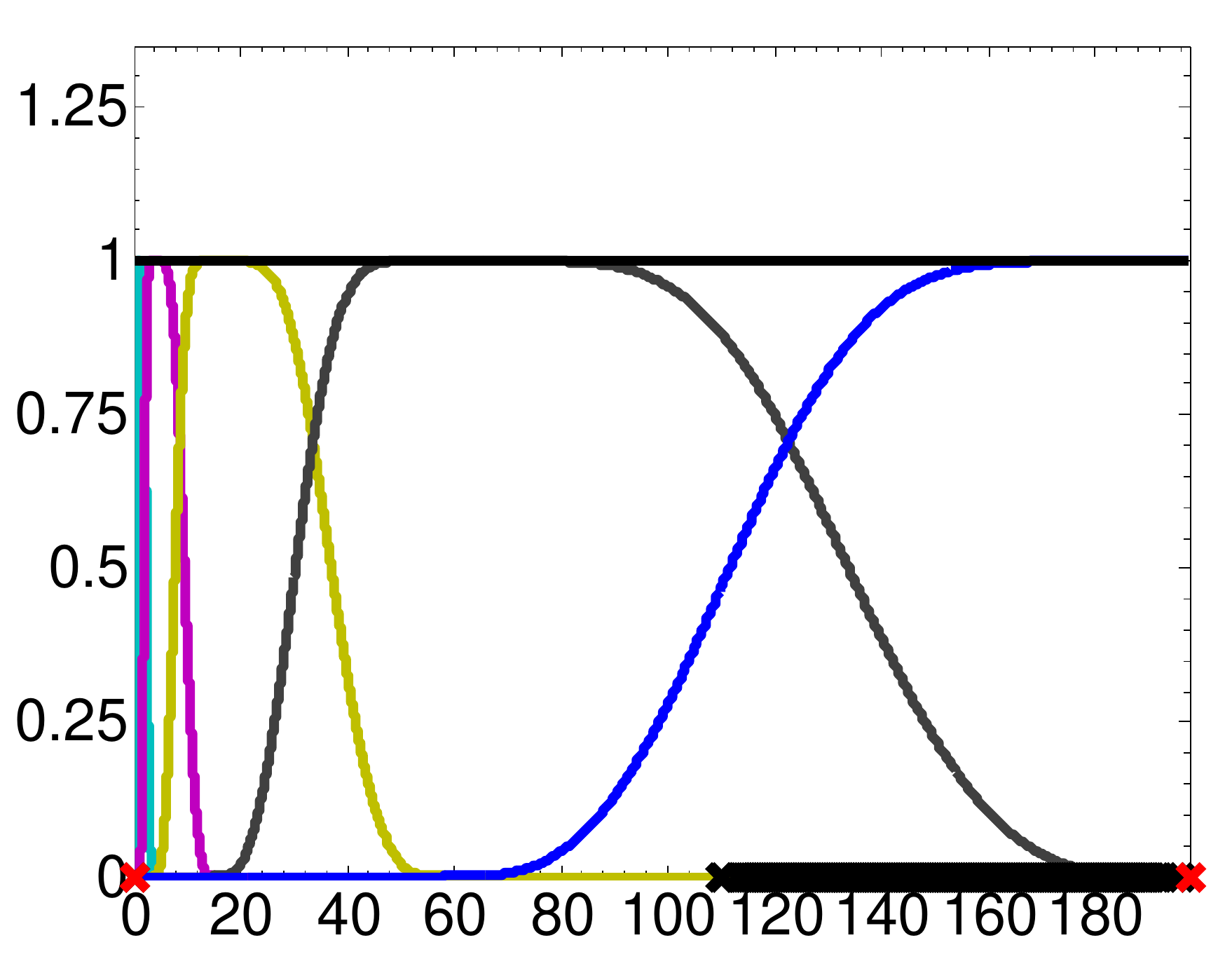}} 
\centerline{\small{~~~~$\lambda$}}
\centerline{\small{~~~~(c)}}
\end{minipage} \\
\vspace{.15in}
\begin{minipage}[b]{.32\linewidth}
\centerline{~~~Warping Functions}
\centerline{\includegraphics[width=\linewidth]{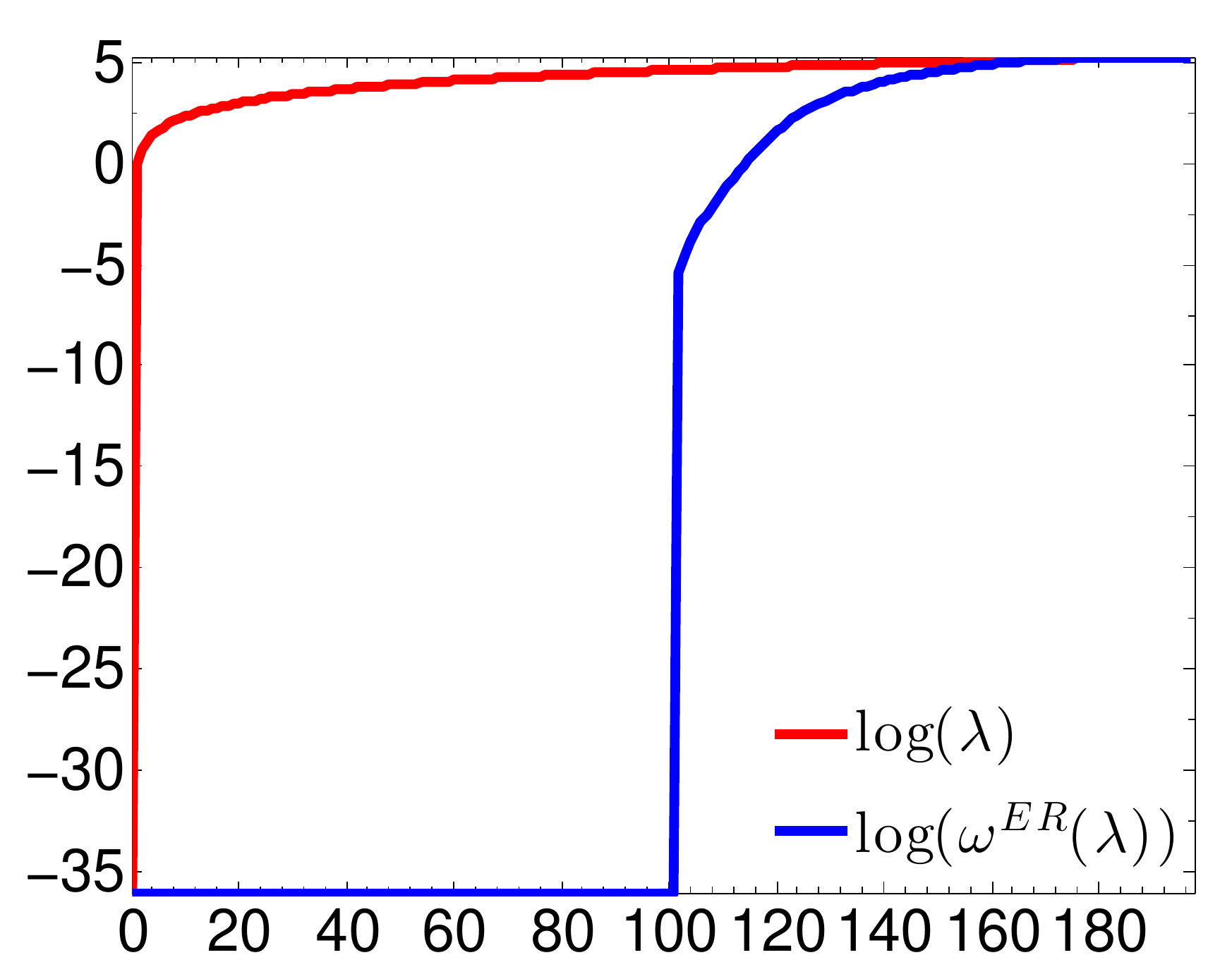}} 
\centerline{\small{~~~~$\lambda$}}
\centerline{\small{~~~~(d)}}
\end{minipage}
\hfill
\begin{minipage}[b]{.32\linewidth}
\centerline{~~~Log-Warped Tight}
\centerline{~~~Graph Wavelet Frame}
\centerline{\includegraphics[width=\linewidth]{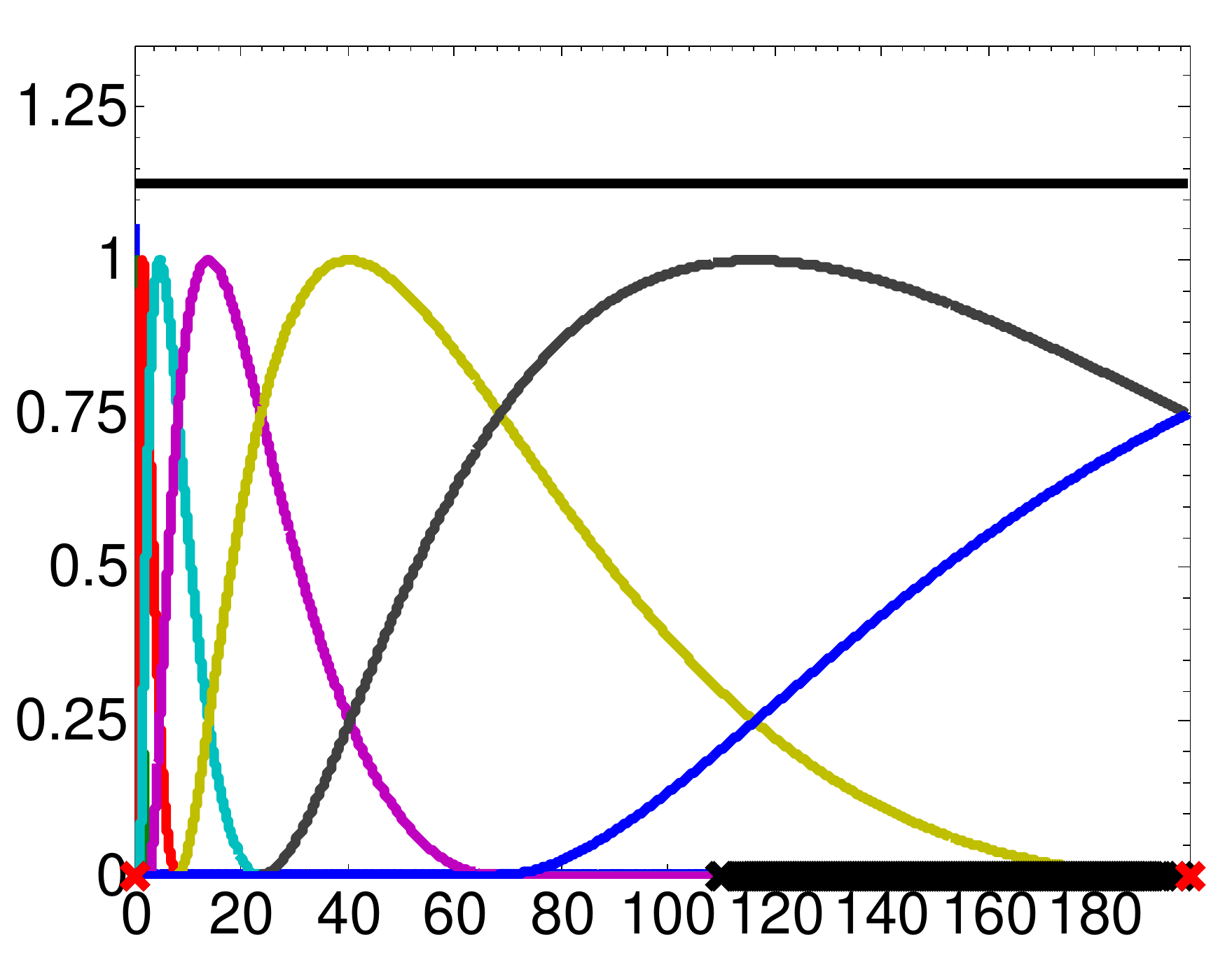}} 
\centerline{\small{~~~~$\lambda$}}
\centerline{\small{~~~~(e)}} 
\end{minipage}
\hfill
\begin{minipage}[b]{.32\linewidth}
\centerline{~~~Spectrum-Adapted Tight}
\centerline{~~~Graph Wavelet Frame}
\centerline{\includegraphics[width=\linewidth]{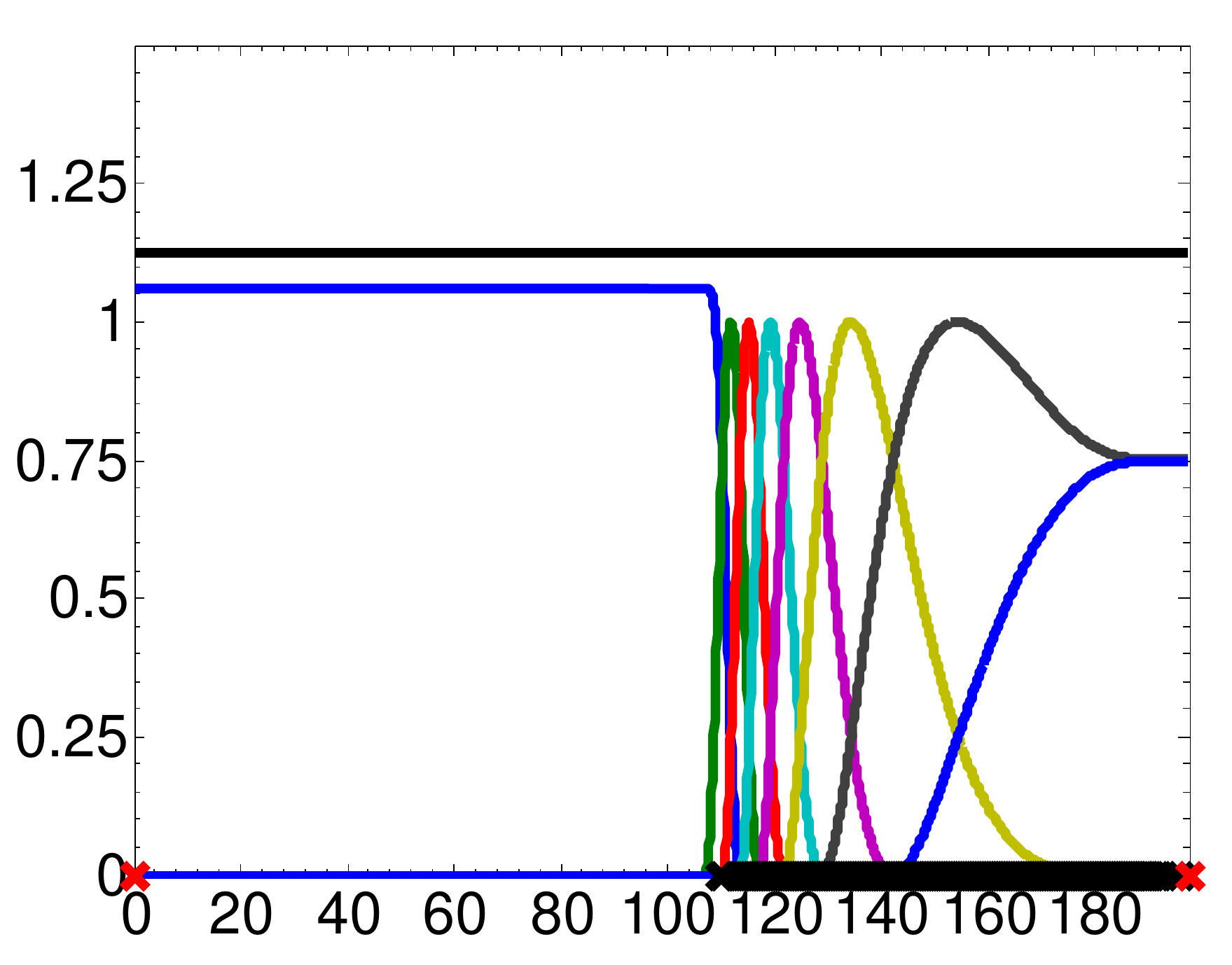}} 
\centerline{\small{~~~~$\lambda$}}
\centerline{\small{~~~~(f)}} 
\end{minipage} 
\caption {Five different sets of wavelet and scaling kernels on the graph Laplacian spectrum for Erd\H{o}s-R{\'e}nyi random graphs with $N=3000$ and edge probability $p=.05$. The spectral graph wavelet transform, Meyer-like tight wavelet frame, and log-warped tight wavelet frame in (a), (b), and (e) are only adapted to an approximation of the length of the spectrum, $\lambda_{upper}$. The Meyer-like tight wavelet frame of (c) is also adapted to the maximum degree via the warping function $C \arccos \left(1-\frac{\lambda}{d_{\max}}\right)$, where the constant $C=\lambda_{upper}/\arccos \left(1-\frac{\lambda_{upper}}{d_{\max}}\right)$ ensures that the range of the warping function is $[0,\lambda_{upper}]$. The tight frame kernels in (f) are adapted to 
an approximation of the empirical spectral cumulative distribution via the composite warping function \eqref{Eq:composite_warp}, which is shown in (d). Although not used in the construction of any of the above filters, the eigenvalues of a single realization from this class of graphs are shown on the horizontal axis of each system of filters. We see that the system of filters in (f) is the only one of the five concentrated on the area of the spectrum where the eigenvalues are concentrated.
} 
\label{Fig:wavelet_comparison}
\end{figure}

\section{Illustrative Example: Scalable Vertex-Frequency Analysis}
In order to extend classical time-frequency analysis to the graph setting, \cite{shuman_SSP_2012,shuman_ACHA_2013} introduce windowed graph Fourier frames, which consist of atoms of the form $g_{i,k}:=M_k T_i g$, where $T_i$ is the generalized translation operator of 
\eqref{Eq:atom_form} and $M_k$ is a generalized modulation operator. The inner products of these atoms with a signal comprise the windowed graph Fourier transform (WGFT), and the squared magnitudes of the WGFT coefficients yield a ``graph spectrogram.'' The graph spectrogram of a given graph signal can be viewed as a frequency-lapse video that shows which frequency components are present in which areas of the graph. In Example \ref{Ex:minn_spectrogram} below, we show how the spectrum-adapted tight frames proposed in Sections \ref{Se:adapted} and \ref{Se:random} can also be used to perform vertex-frequency analysis.

\begin{figure}[t]
\centering
\hfill
\hfill
\begin{minipage}[b]{.4\linewidth}
\centerline{Clusters~}
\centerline{\includegraphics[width=\linewidth]{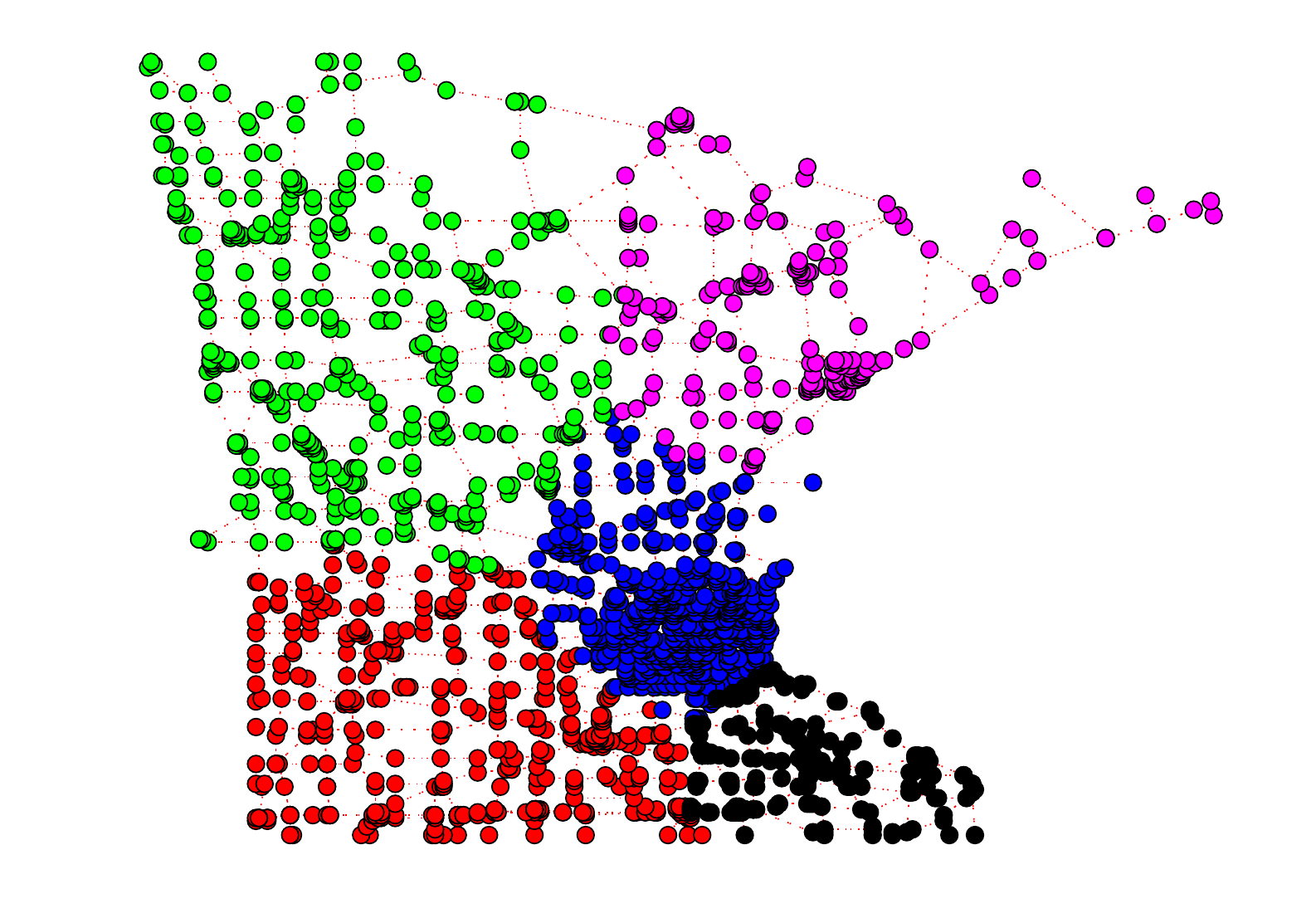}} 
\centerline{\small{(a)}}
\end{minipage}
\hfill
\begin{minipage}[b]{.4\linewidth}
\centerline{$f$~}
\centerline{\includegraphics[width=\linewidth]{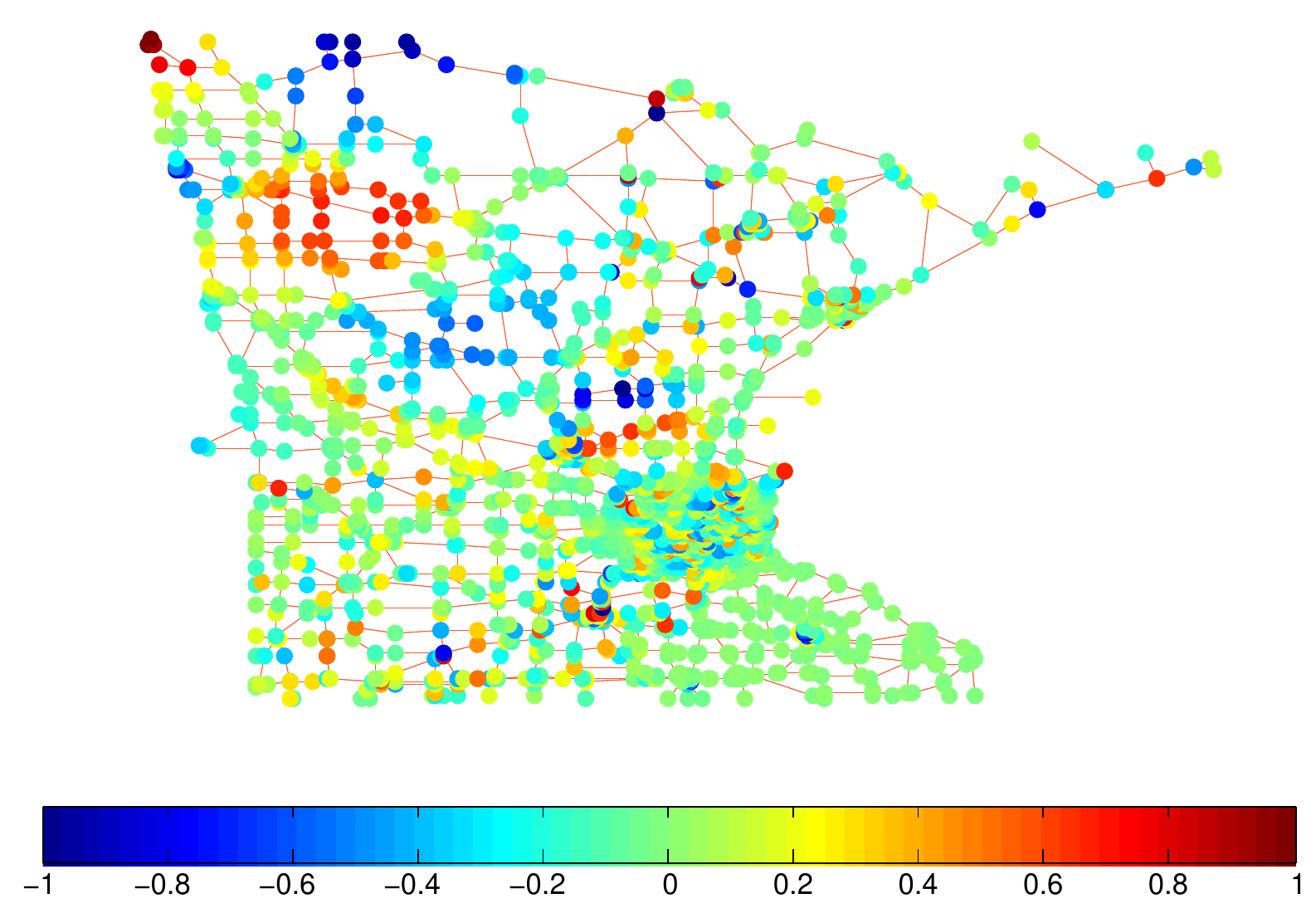}} 
\centerline{\small{~(b)}}
\end{minipage}
\hfill 
\hfill
\hfill
\\
\vspace{.15in}
\centering
\begin{minipage}[b]{.3\linewidth}
\centerline{~~~$\hat{f}\left(\lambda_{\l}\right)$}
\centerline{\includegraphics[width=\linewidth]{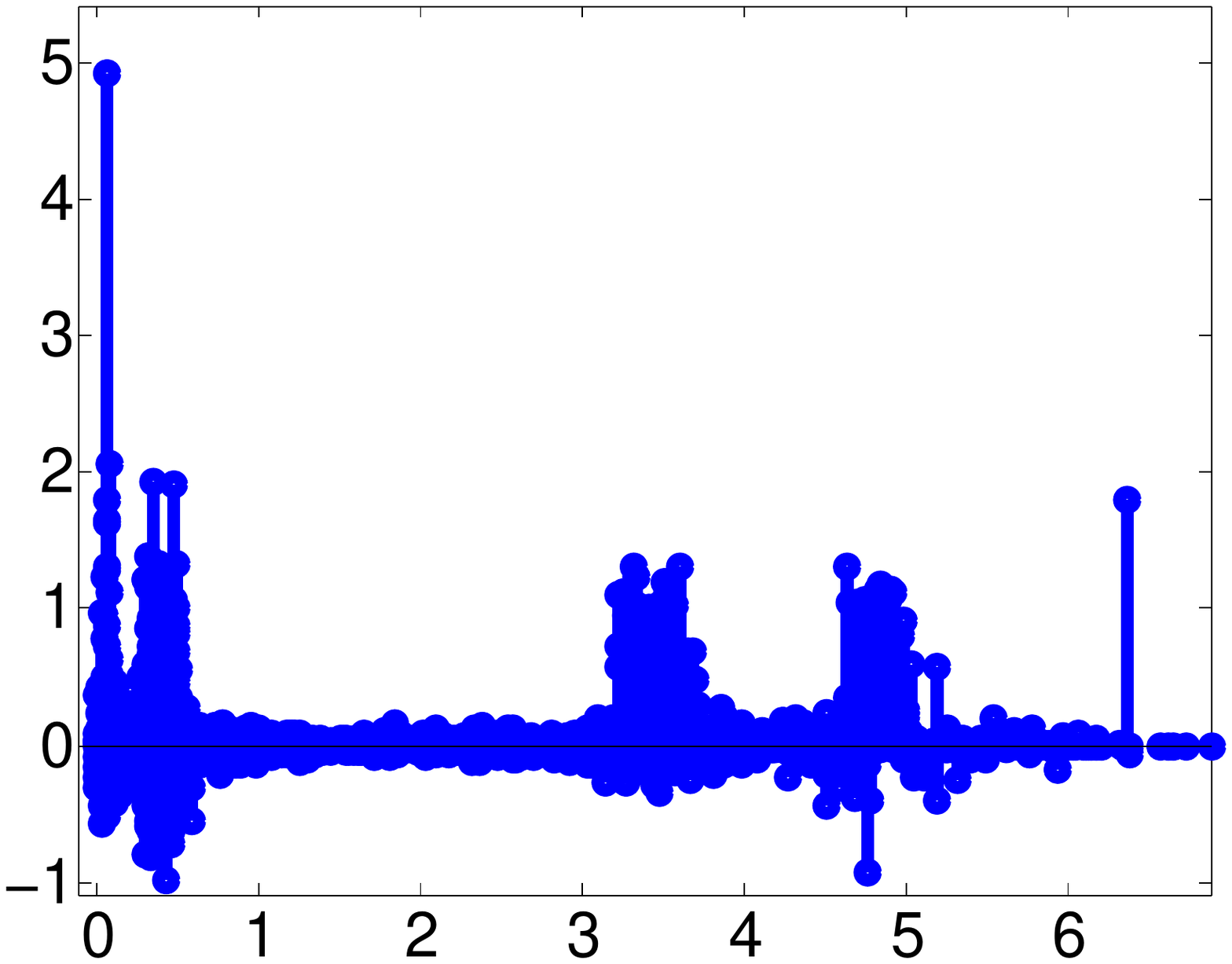}} 
\centerline{\small{~~~~$\lambda$}}
\centerline{\small{~~~~(c)}}
\end{minipage}
\hfill
\begin{minipage}[b]{.3\linewidth}
\centerline{~~~Warping Function}
\centerline{\includegraphics[width=\linewidth]{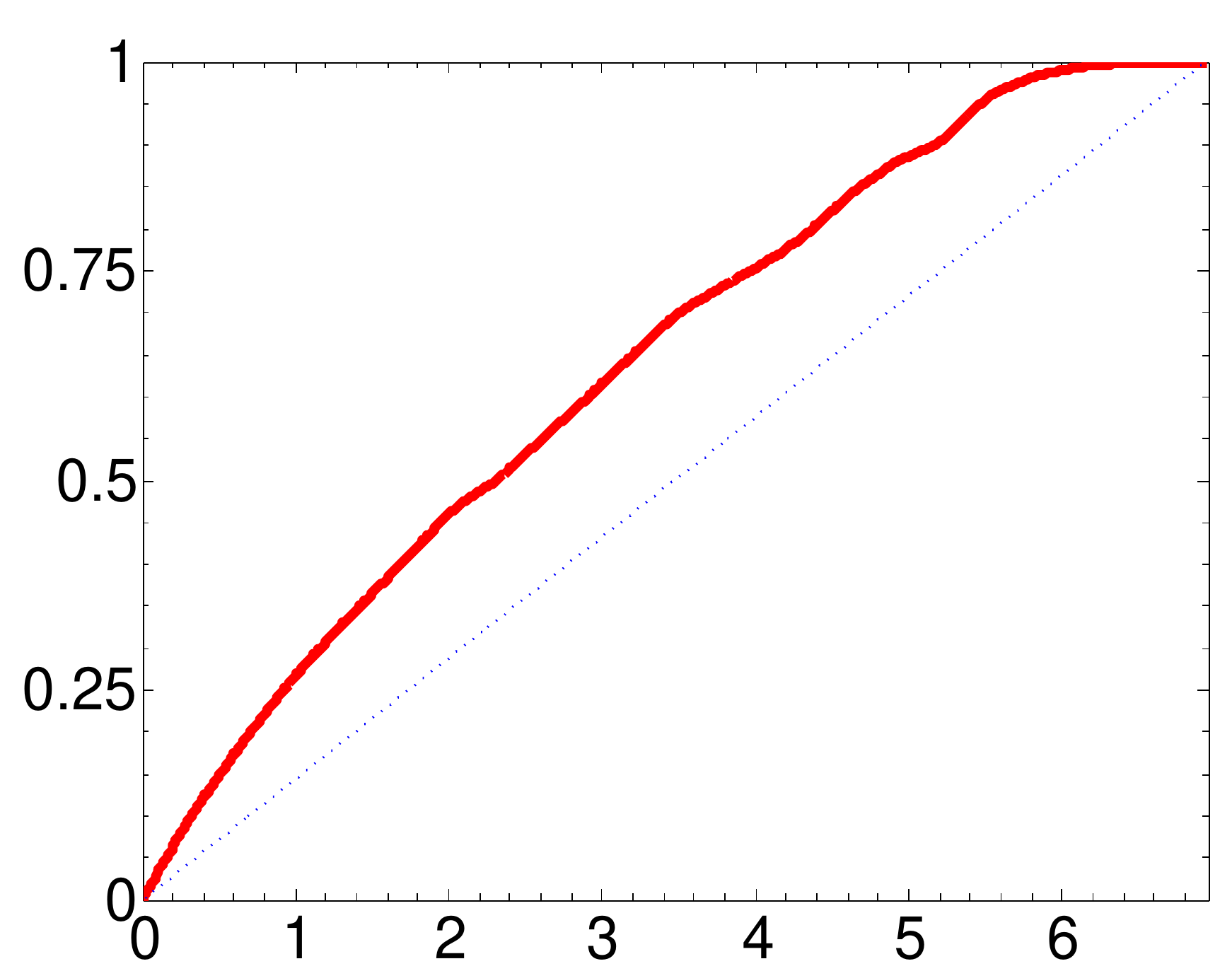}} 
\centerline{\small{~~~~$\lambda$}}
\centerline{\small{~~~~(d)}}
\end{minipage}
\hfill
\begin{minipage}[b]{.3\linewidth}
\centerline{~~~Spectrum-Adapted Filters}
\centerline{\includegraphics[width=\linewidth]{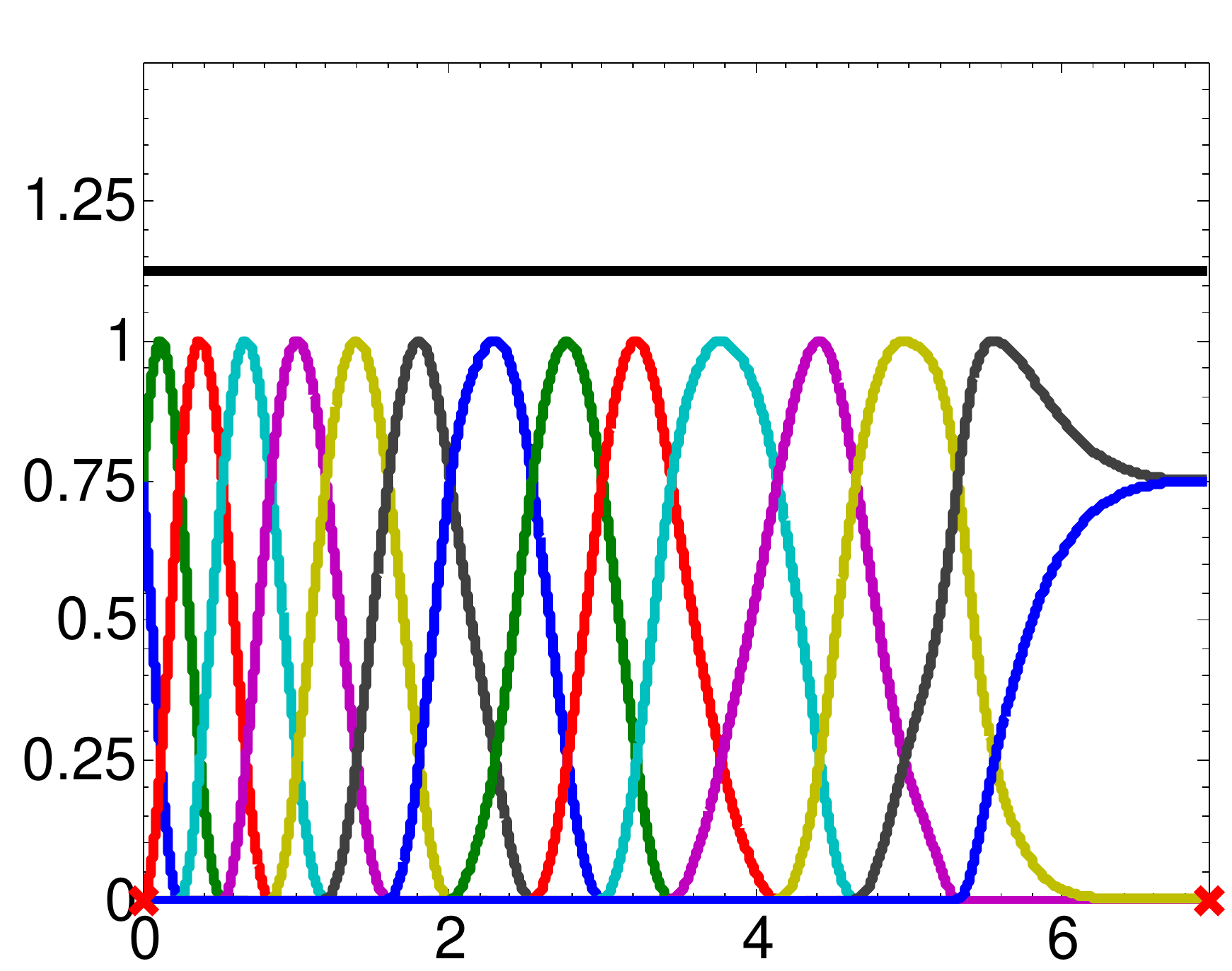}} 
\centerline{\small{~~~~$\lambda$}}
\centerline{\small{~~~~(e)}}
\end{minipage} 
\caption {(a) The Minnesota road graph segmented into five different clusters. (b) The graph signal defined in \eqref{Eq:minn_signal}. (c) The graph Fourier transform of the signal in (b). (d) A warping function (in red) generated from an approximation of the cumulative spectral density function of the graph Laplacian eigenvalues of the Minnesota graph. The dashed blue line spans the diagonal for comparison. (e) A system of warped filters, $\{\widehat{g_m}\}_{m=1,2,\ldots,15}$, where $\widehat{g_m}(\lambda)=\widehat{g_m^U}(\omega(\lambda))$ and $\omega(\cdot)$ is the warping function shown in (d).} 
 \label{Fig:minn1}
\end{figure}

\begin{example}\label{Ex:minn_spectrogram}
We form a signal $f$ on the (unweighted) Minnesota road graph \cite{gleich}
by first using spectral clustering (see, e.g., \cite{spectral_clustering}) to partition the graph into the five clusters shown in Figure \ref{Fig:minn1}(a), and then by summing up eigenvectors in different frequency bands and restricting them to different clusters of the graph. More specifically, $f:=\sum_{j=1}^5 f_j/\norm{f_j}_{\infty}$, where 
\begin{align} \label{Eq:minn_signal}
f_j(i) := \Identity_{\{\hbox{vertex i is in cluster j}\}}\sum_{\l=0}^{N-1} u_{\l}(i)\Identity_{\{ \underline{\tau}_j \leq \lambda_{\l}\leq \bar{\tau}_j \}}.
\end{align}
We take the sequence $\{[\underline{\tau}_j,\bar{\tau}_j]\}_{j=1,2,\ldots,5}$ to be $ [0.06,0.08], [0.3,0.5], [3.2,3.7], [4.6,5.0], [6.0,6.6]$ for the green, blue, red, magenta, and black clusters, respectively.
In Figure \ref{Fig:minn1}(c), we plot the graph Fourier transform of $f$, and we can see the different frequency components of the signal, but we can not tell that the frequency components are localized to different sections of the graph. We then use the method of Section \ref{Se:approx_spectrum} to approximate the cumulative spectral density function of the graph Laplacian eigenvalues, and use it as a warping function (shown in Figure \ref{Fig:minn1}(d)) to generate the system of 15 spectral graph filters in Figure \ref{Fig:minn1}(e). 
We then generate a tight frame of vertex-frequency atoms of the form of ${\cal D}$ in Lemma \ref{Le:frame}, and plot the magnitudes of the inner products of the signal $f$ with some of these atoms in Figure \ref{Fig:minn2}. 
While the structure of $f$ is not apparent from its plot 
in Figure \ref{Fig:minn1}(b), the coefficients in 
%
Figure \ref{Fig:minn2} show the varying degree of local smoothness of the signal in different regions of the graph.

\begin{figure}[h]
\centering
\begin{minipage}[b]{.23\linewidth}
\centerline{$|\ip{f}{T_i g_1}|~~~~$}
\centerline{\includegraphics[width=\linewidth]{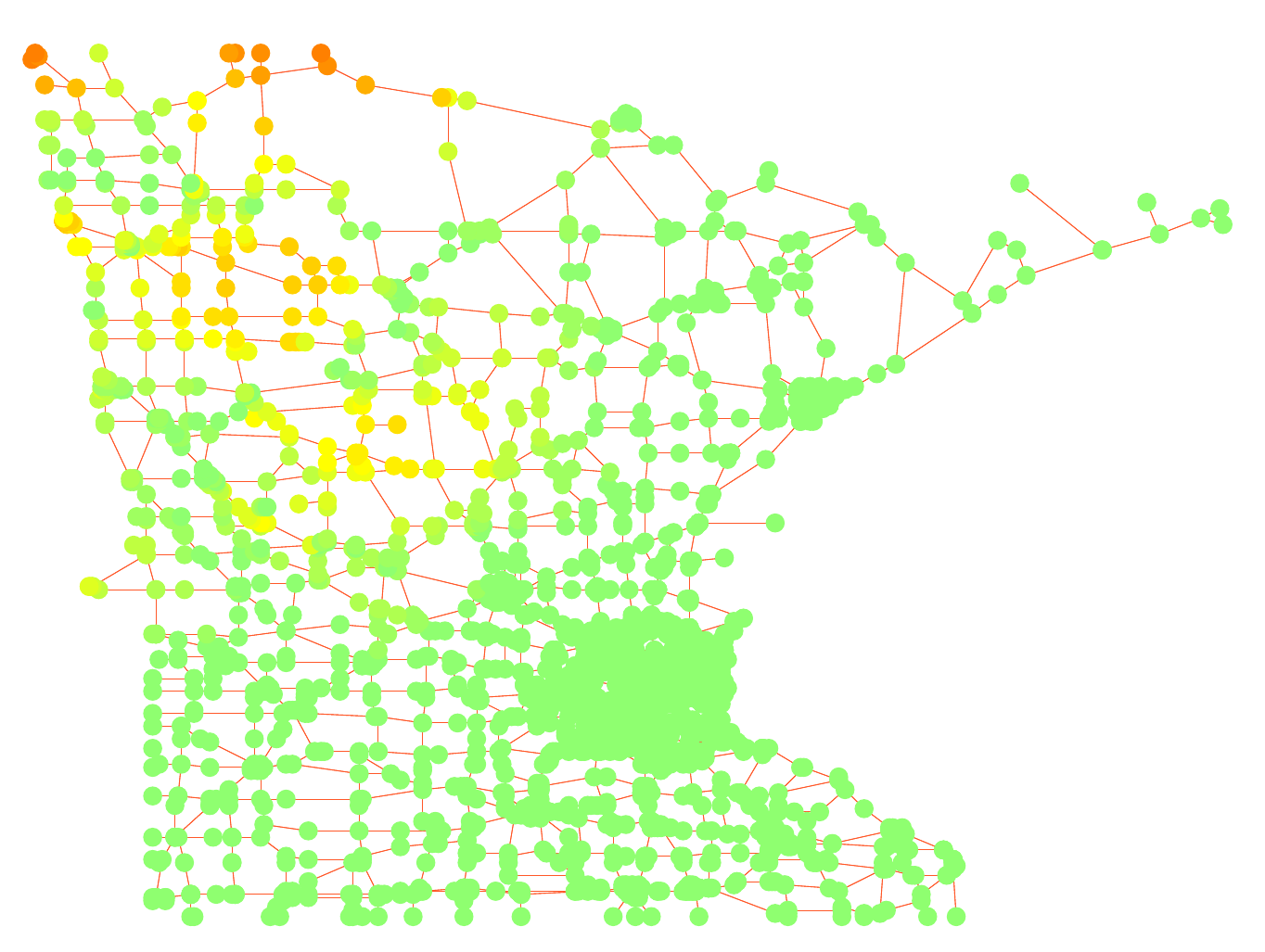}} 
\end{minipage}
\hfill
\begin{minipage}[b]{.23\linewidth}
\centerline{$|\ip{f}{T_i g_2}|~~~~$}
\centerline{\includegraphics[width=\linewidth]{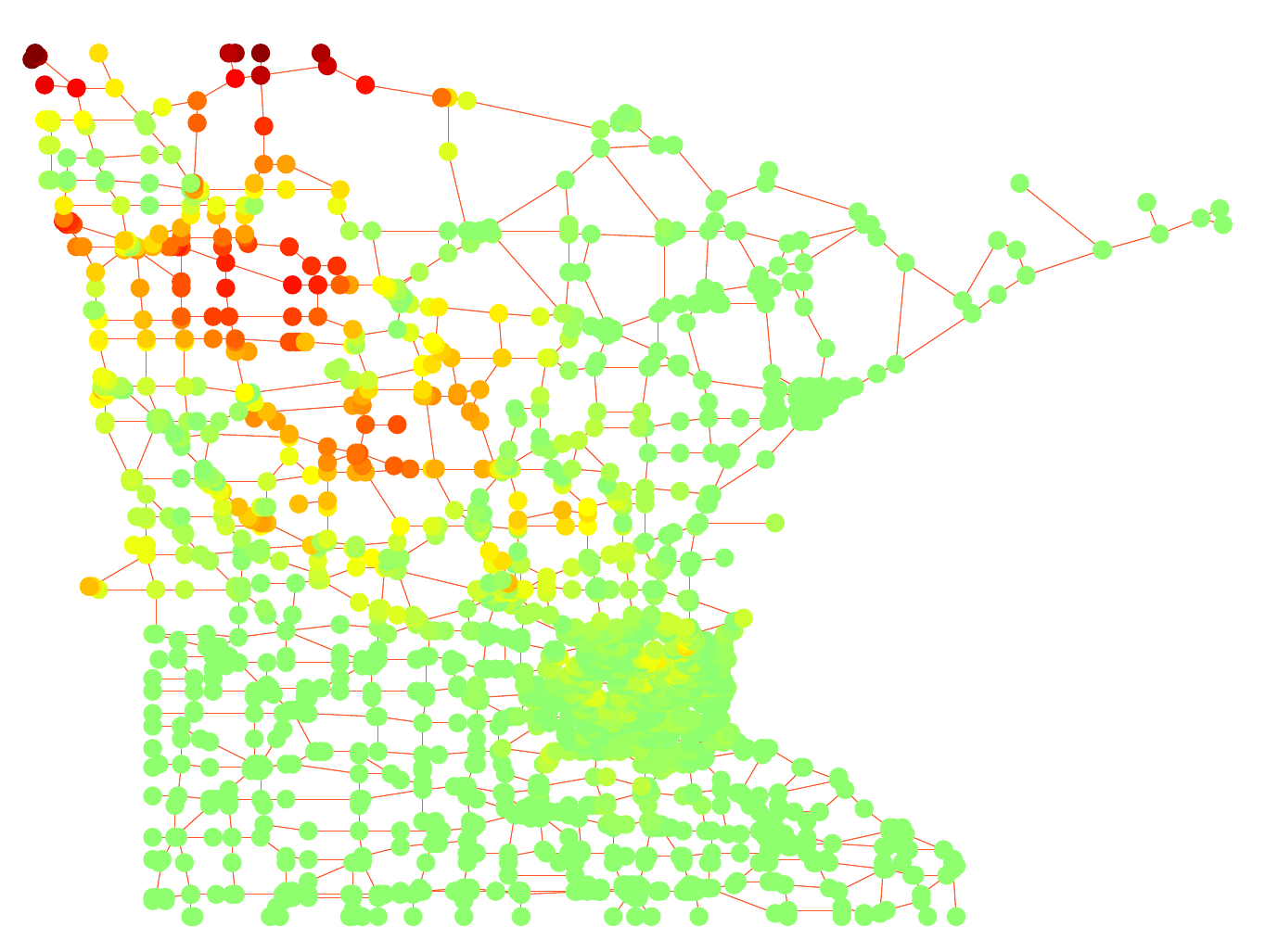}} 
\end{minipage}
\hfill
\begin{minipage}[b]{.23\linewidth}
\centerline{$|\ip{f}{T_i g_3}|~~~~$}
\centerline{\includegraphics[width=\linewidth]{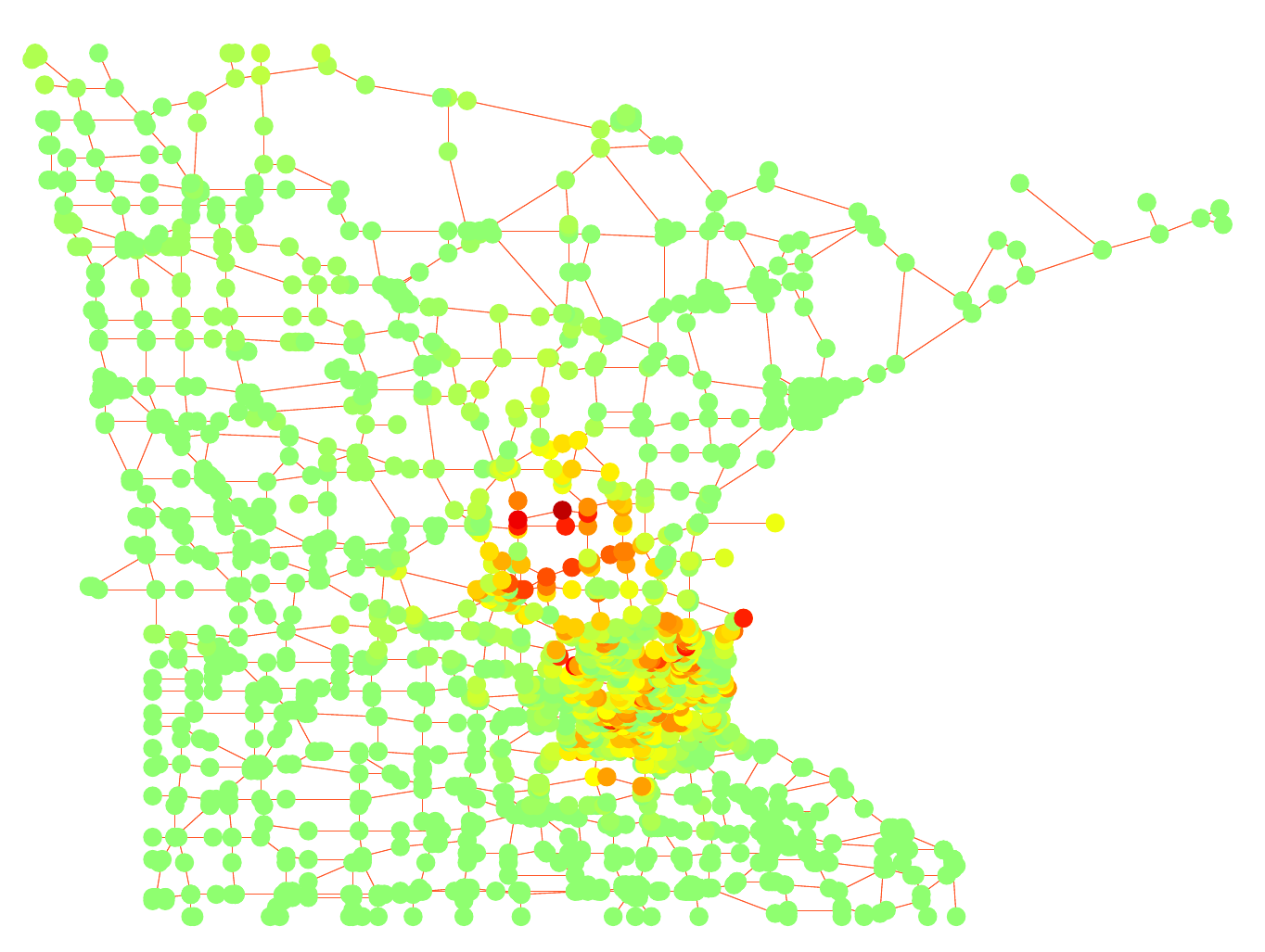}} 
\end{minipage}
\hfill
\begin{minipage}[b]{.23\linewidth}
\centerline{$|\ip{f}{T_i g_4}|~~~~$}
\centerline{\includegraphics[width=\linewidth]{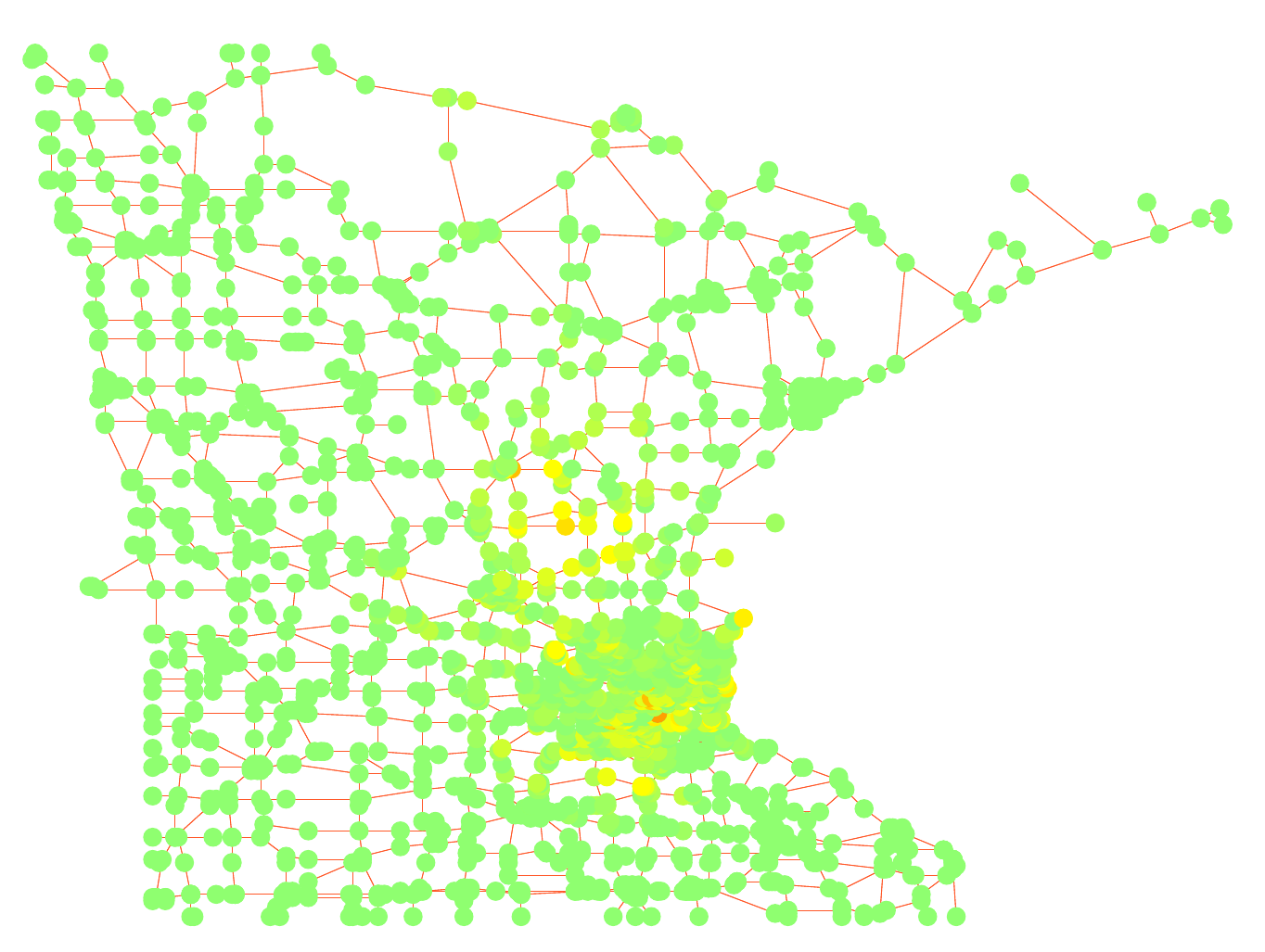}} 
\end{minipage} \\
\vspace{.2in}
\centering
\begin{minipage}[b]{.23\linewidth}
\centerline{$|\ip{f}{T_i g_{8}}|~~~~$}
\centerline{\includegraphics[width=\linewidth]{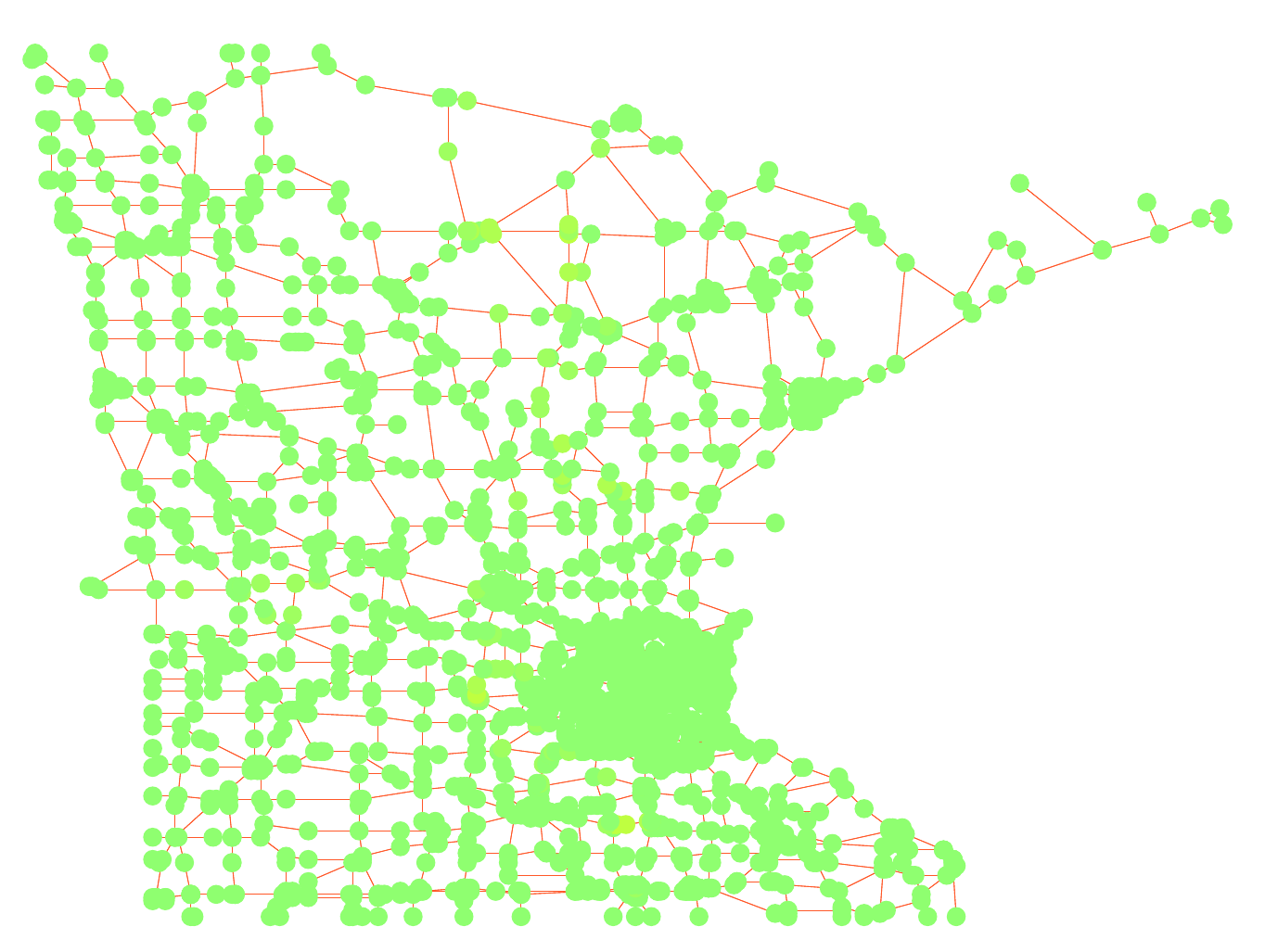}} 
\end{minipage}
\hfill
\begin{minipage}[b]{.23\linewidth}
\centerline{$|\ip{f}{T_i g_{9}}|~~~~$}
\centerline{\includegraphics[width=\linewidth]{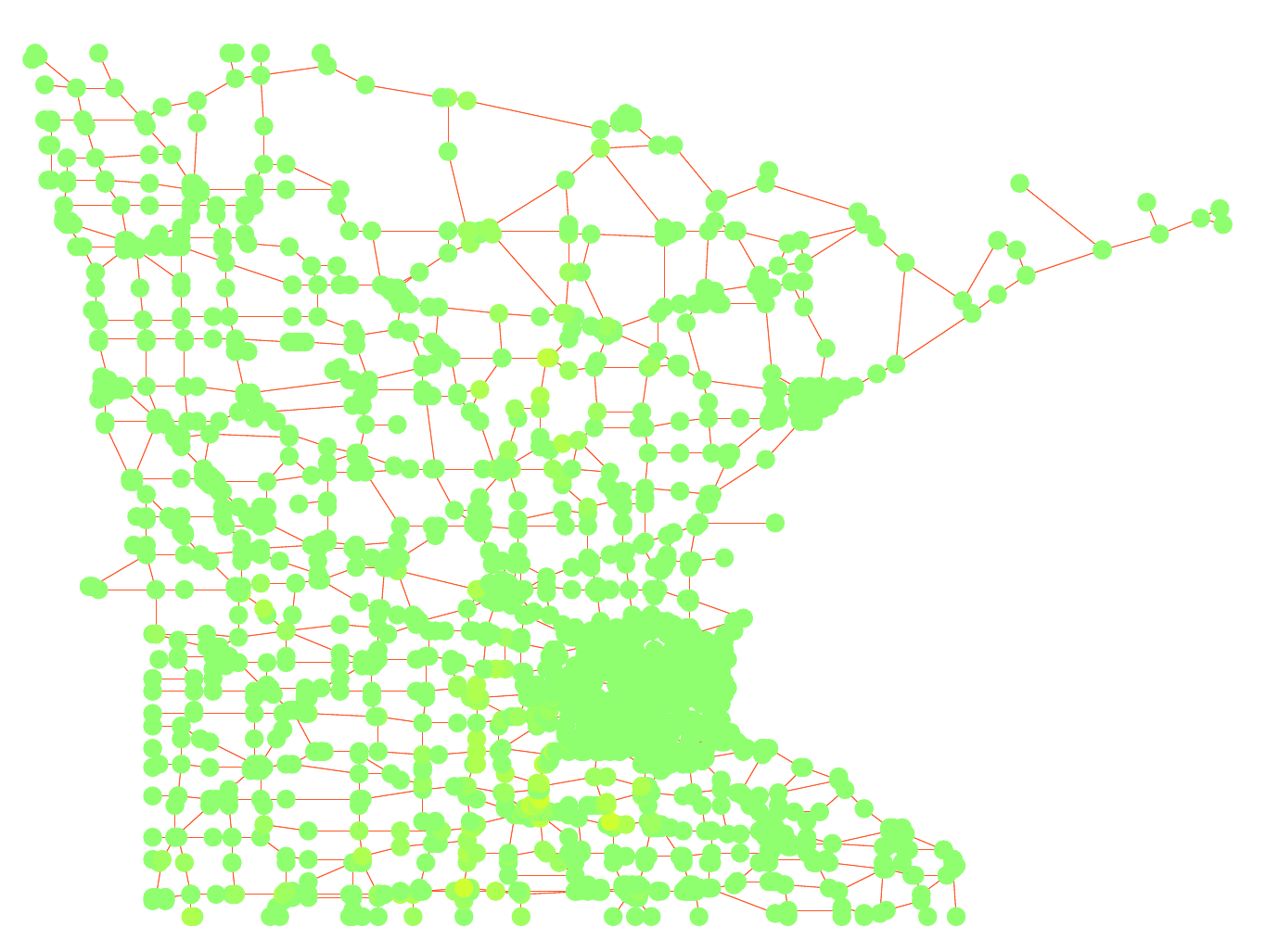}} 
\end{minipage}
\hfill
\begin{minipage}[b]{.23\linewidth}
\centerline{$|\ip{f}{T_i g_{10}}|~~~~$}
\centerline{\includegraphics[width=\linewidth]{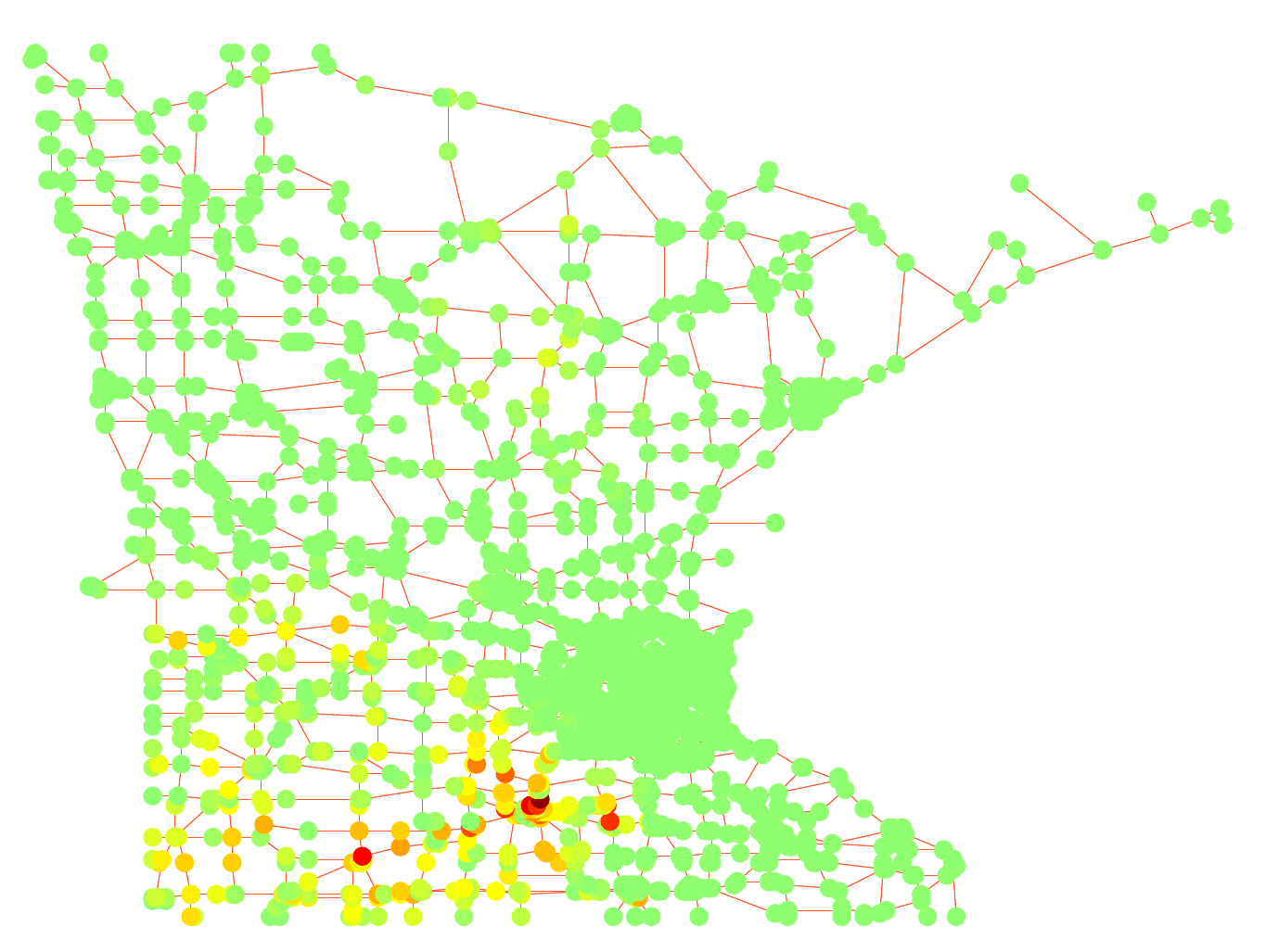}} 
\end{minipage}
\hfill
\begin{minipage}[b]{.23\linewidth}
\centerline{$|\ip{f}{T_i g_{11}}|~~~~$}
\centerline{\includegraphics[width=\linewidth]{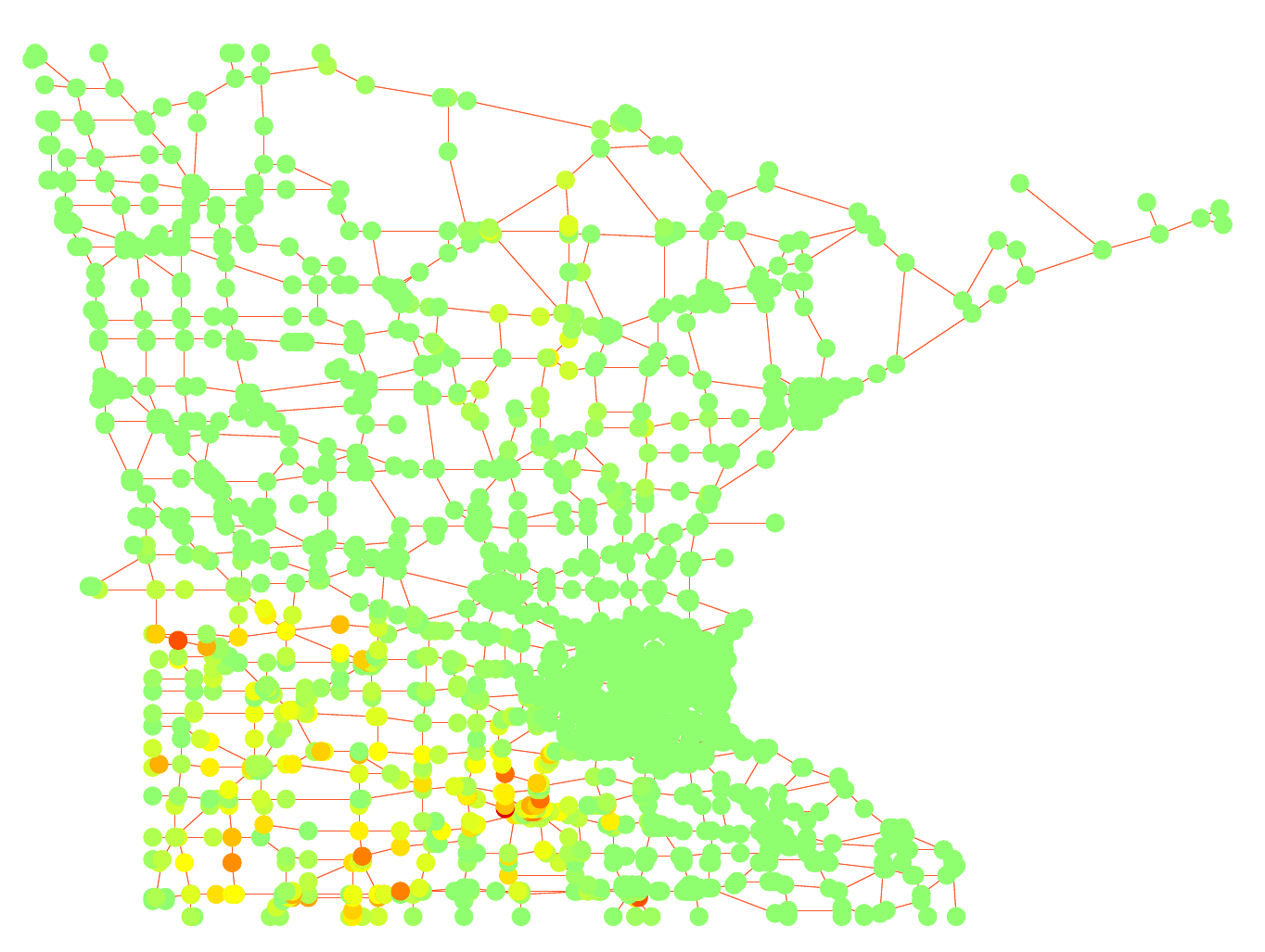}} 
\end{minipage} \\
\vspace{.2in}
\centering
\begin{minipage}[b]{.23\linewidth}
\centerline{$|\ip{f}{T_i g_{12}}|~~~~$}
\centerline{\includegraphics[width=\linewidth]{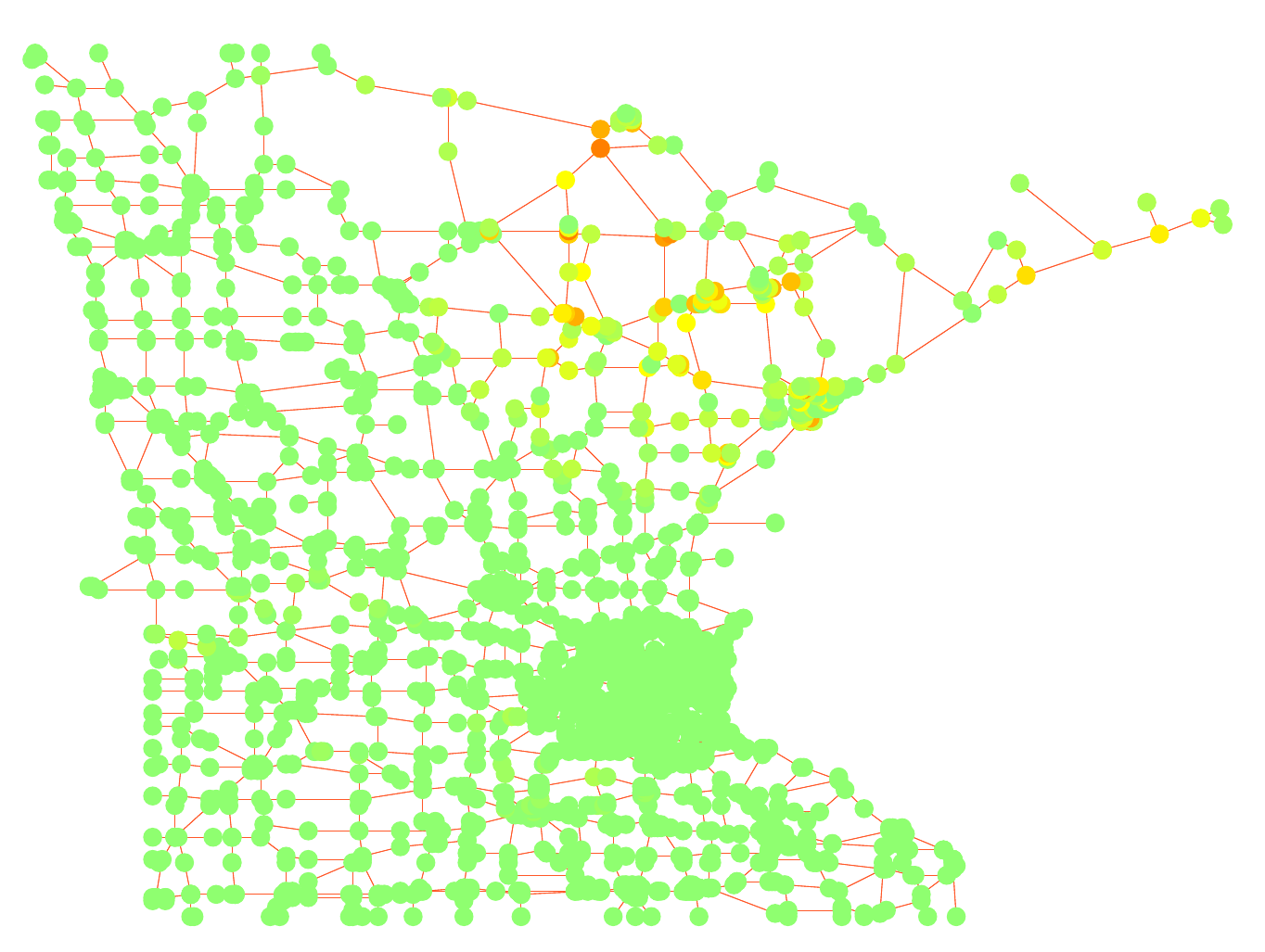}} 
\end{minipage}
\hfill
\begin{minipage}[b]{.23\linewidth}
\centerline{$|\ip{f}{T_i g_{13}}|~~~~$}
\centerline{\includegraphics[width=\linewidth]{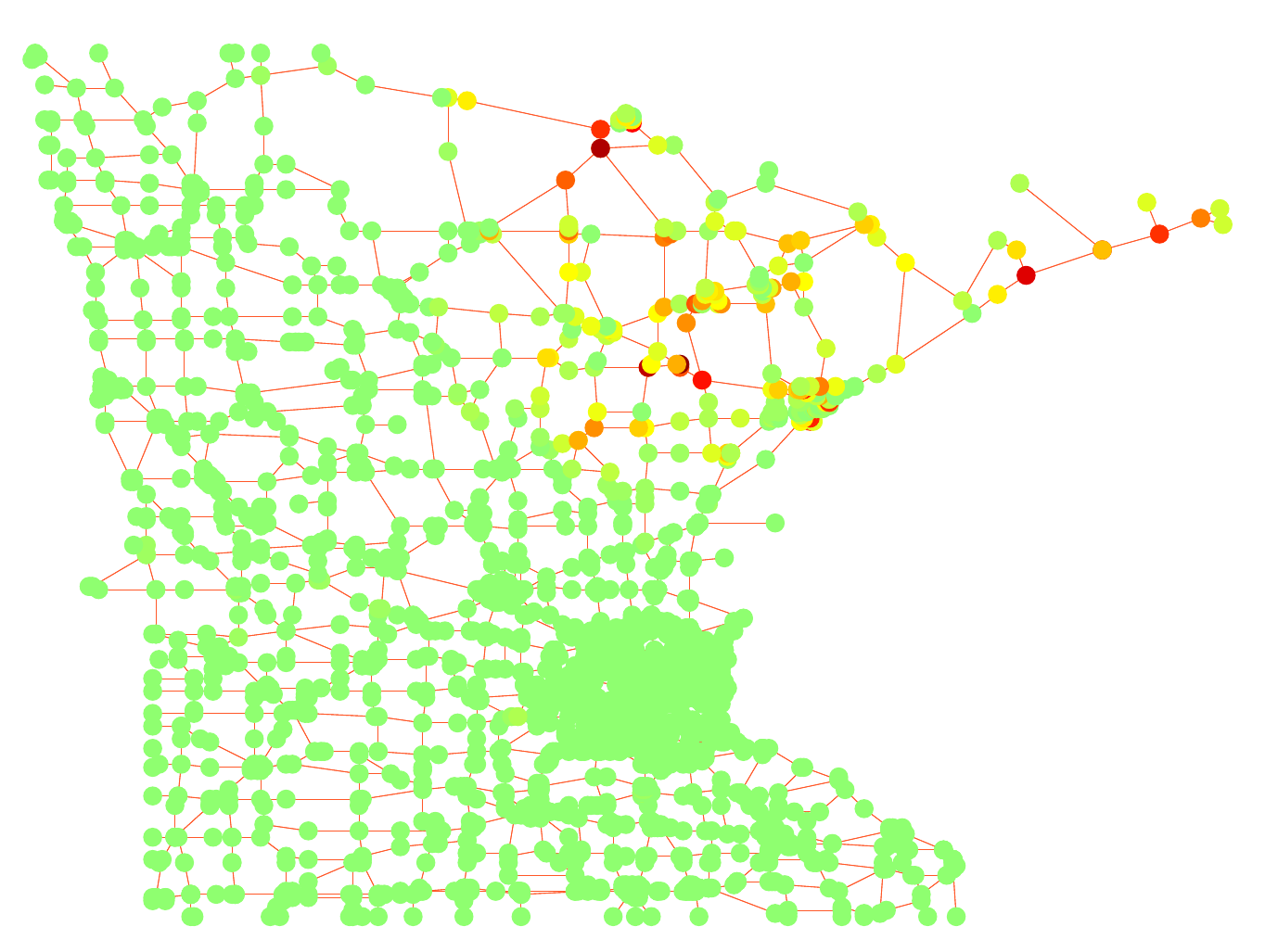}} 
\end{minipage}
\hfill
\begin{minipage}[b]{.23\linewidth}
\centerline{$|\ip{f}{T_i g_{14}}|~~~~$}
\centerline{\includegraphics[width=\linewidth]{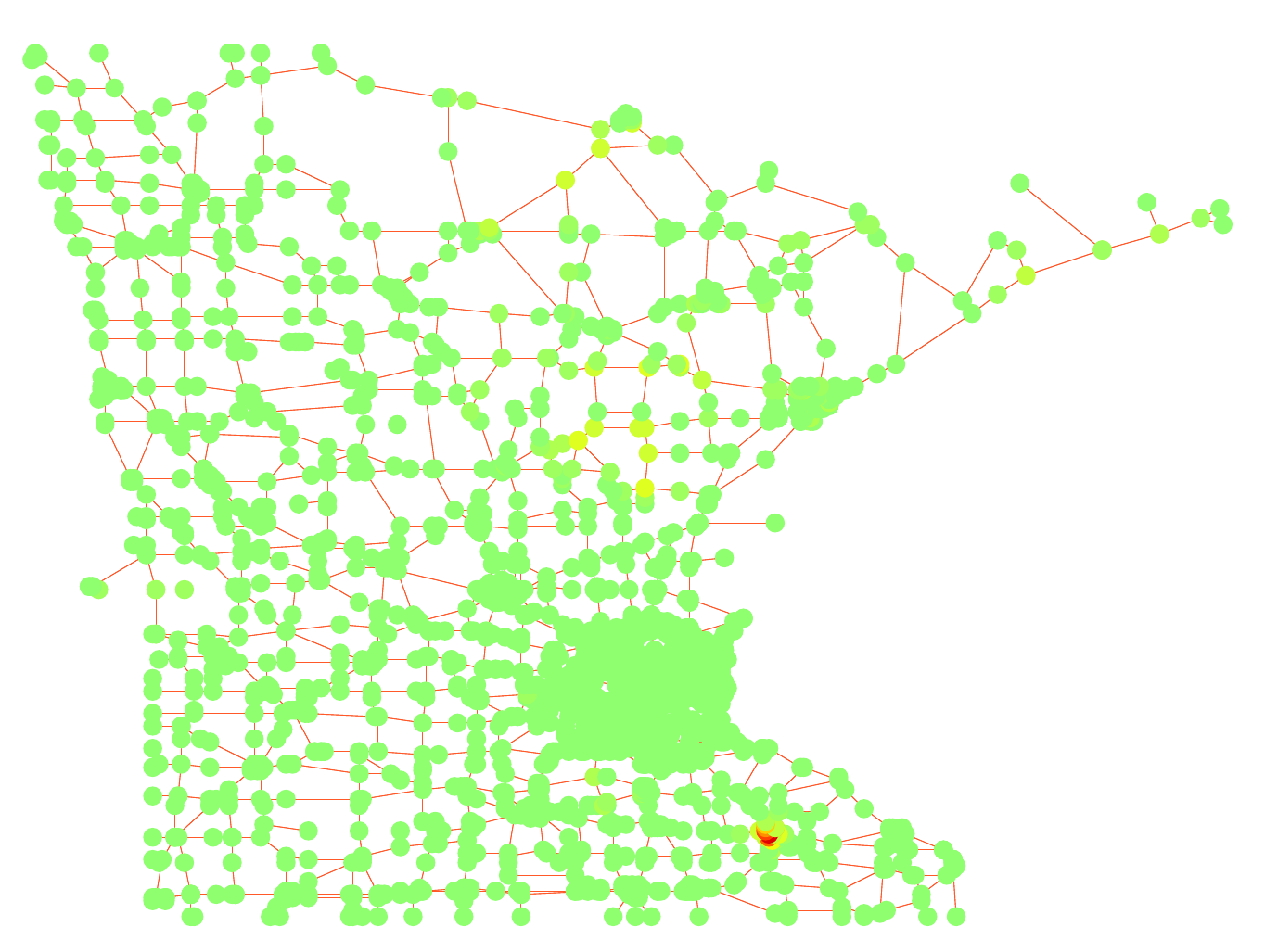}} 
\end{minipage}
\hfill
\begin{minipage}[b]{.23\linewidth}
\centerline{$|\ip{f}{T_i g_{15}}|~~~~$}
\centerline{\includegraphics[width=\linewidth]{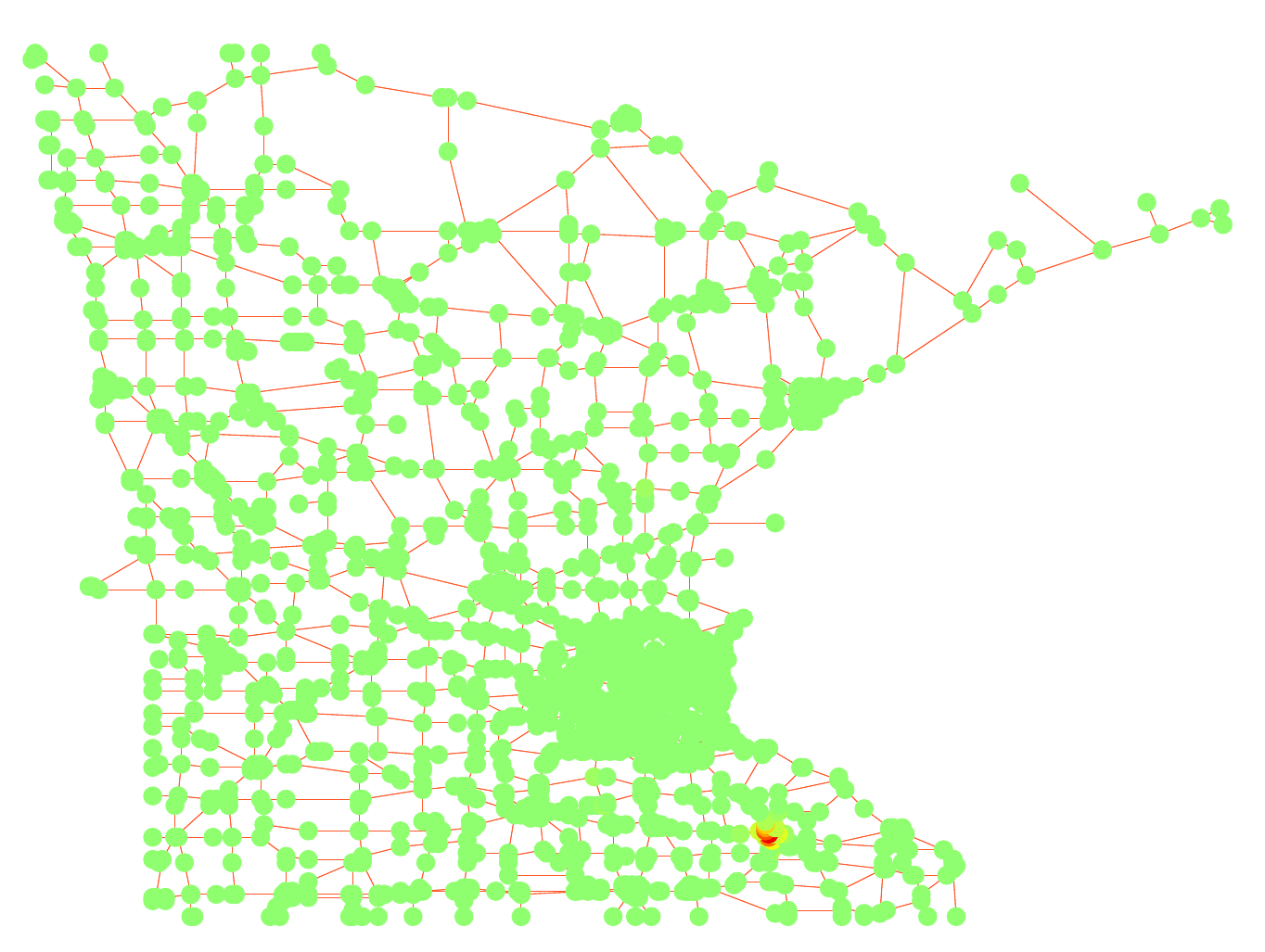}} 
\end{minipage} 
\caption {Vertex-frequency analysis of the signal $f$ from Figure \ref{Fig:minn1}(b) using atoms generated from the system of warped filters in Figure \ref{Fig:minn1}(e). Each plot contains the $N$ coefficients resulting from a single graph spectral filter. We have omitted the coefficients arising from the filters indexed by $m=5,6,7,8$ since they are nearly zero. This frequency-lapse sequence shows us which frequency components are present in the signal in which parts of the graph. For example, the lower left portion of the graph has larger coefficients for those atoms generated by filters indexed by $m=10,11$, corresponding to eigenvectors associated with eigenvalues roughly between 3 and 4. The black cluster is the most difficult to make out from the coefficients, because some of the eigenvectors associated with larger eigenvalues are more localized.} 
 \label{Fig:minn2}
\end{figure}
\end{example}

The WGFT may be a more natural generalization of classical time-frequency analysis, 
but
there are a number of practical advantages of performing vertex-frequency analysis with the method proposed in this paper rather than with the WGFT. First, the proposed method is considerably more efficient from a computational standpoint as it does not require the full eigendecomposition of the graph Laplacian. Second, because the filters lead to a tight frame of vertex-frequency atoms, the squared magnitudes of the transform coefficients can indeed be interpreted as an energy density function, which is true with the classical spectrogram, but not with the spectrogram arising from the WGFT since windowed graph Fourier frames are not generally tight. Third, by choosing the number of filters to be significantly smaller than the number of vertices, we reduce the redundancy of the transform.

\section{Conclusion}
We have presented new methods to generate tight frames of atoms to represent signals residing on weighted graphs. Our primary approach is to construct spectral graph filters by warping systems of uniform translates, and then 
generate the dictionary atoms by translating these filters to each vertex in the graph. The main benefits of this construction are (i) by choosing the uniform translates from a certain family (c.f. Theorem \ref{Th:uniform_translates}), we are able to guarantee that the resulting frames are tight; (ii) the resulting frames are computationally efficient to implement, as they do not require a full eigendecomposition of the graph Laplacian; and (iii) the warping function enables us to adapt the spectral graph filters to the specific distribution of Laplacian eigenvalues, rather than just the length of the spectrum, which leads to dictionary atoms with better ability to discriminate between different graph signals. As examples of spectrum-adapted graph frames, we used an approximation of the cumulative spectral density function as the warping function to generate tight vertex-frequency frames, and a composition of that warping function with a logarithmic warping function to generate tight spectrum-adapted graph wavelet frames. One line of ongoing work is the investigation of different methods to approximate the cumulative spectral density function for extremely large graphs.

\section{Appendix}

\begin{IEEEproof}[Proof of Lemma \ref{Le:frame}]
\begin{align}
\sum_{i=1}^N \sum_{m=1}^{M} \left|\ip{f}{g_{i,m}}\right|^2
&=\sum_{i=1}^N \sum_{m=1}^{M} \left|\ip{\hat{f}}{\widehat{T_i g_m}}\right|^2 \label{Eq:frame1} \\
&=\sum_{i=1}^N \sum_{m=1}^{M} \left(\sum_{\l=0}^{N-1}\hat{f}(\lambda_{\l})\sqrt{N}\widehat{g_m}^*(\lambda_{\l})u_{\l}(i)\right)\left(\sum_{\l^{\prime}=0}^{N-1}\hat{f}(\lambda_{\l^{\prime}})\sqrt{N}\widehat{g_m}^*(\lambda_{\l^{\prime}})u_{\l^{\prime}}(i)\right)^* \nonumber \\
&=N\sum_{\l=0}^{N-1}|\hat{f}(\lambda_{\l})|^2 \sum_{m=1}^{M}|\widehat{g_m}(\lambda_{\l})|^2 \label{Eq:frame2} \\
&=N\sum_{\l=0}^{N-1}|\hat{f}(\lambda_{\l})|^2 G(\lambda_{\l}), \nonumber
\end{align}
where \eqref{Eq:frame1} follows from Parseval's Relation, and \eqref{Eq:frame2} follows from the fact that $\sum_{i=1}^N u_{\l}(i)u_{\l^{\prime}}^*(i)=\delta_{\l,\l^{\prime}}$, by the orthonormal nature of the eigenvectors. Applying Parseval's Relation a second time yields the desired result.
\end{IEEEproof}

\begin{IEEEproof}[Proof of Theorem \ref{Th:uniform_translates}]
Let $t\in \Rbb$ be arbitrary. Then
  \begin{align}\label{Eq:th_pr_1}
    \sum_{m\in {\mathbb Z}} \left|q\left(t-\frac{m}{R}\right)\right|^2 &= 
    \sum_{m\in {\mathbb Z}} \left|\sum_{k=0}^K a_k \cos\Biggl(2\pi k \Bigl(t-\frac{m}{R}-\frac{1}{2}\Bigr)\Biggr) \Identity_{\{\frac{m}{R} \leq t < 1+\frac{m}{R}\}}\right|^2 \nonumber \\
    &=   \sum_{m=\lfloor Rt-(R-1) \rfloor}^{\lfloor Rt \rfloor} \left|\sum_{k=0}^K a_k \cos\Biggl(2\pi k \Bigl(t-\frac{m}{R}-\frac{1}{2}\Bigr)\Biggr) \right|^2 \nonumber \\
    &=\sum_{k=0}^K \sum_{j=0}^K a_k a_j \sum_{m=\lfloor Rt-(R-1) \rfloor}^{\lfloor Rt \rfloor} \cos\Biggl(2\pi k \Bigl(t-\frac{m}{R}-\frac{1}{2}\Bigr)\Biggr) \cos\Biggl(2\pi j \Bigl(t-\frac{m}{R}-\frac{1}{2}\Bigr)\Biggr).
        \end{align}
        \vspace{-.75cm}
        
    \noindent   Defining $z:=2\pi k\Bigl(t-\frac{1}{2}\Bigr)$ and $y:=2\pi j\Bigl(t-\frac{1}{2}\Bigr)$, and the inner terms of  \eqref{Eq:th_pr_1} as 
       \begin{align*}
       A_{k,j}:=\sum_{m=\lfloor Rt-(R-1) \rfloor}^{\lfloor Rt \rfloor} \cos\Biggl(z-\frac{2\pi k m}{R}\Biggr) \cos\Biggl(y-\frac{2\pi j m}{R}\Biggr),~~ 0\leq k,j \leq K,
       \end{align*} 
       and expanding the cosines into complex exponentials, 
       we have
\begin{align}
A_{k,j}&=\frac{1}{4}\sum_{m=\lfloor Rt-(R-1) \rfloor}^{\lfloor Rt \rfloor}
\left\{
\begin{array}{l}
\biggl[ \exp\left(iz + 2\pi i \frac{km}{R}\right) + \exp\left(-iz - 2\pi i \frac{km}{R}\right)\biggr] \\
\cdot  \biggl[ \exp\left(iy + 2\pi i \frac{jm}{R}\right) + \exp\left(-iy - 2\pi i \frac{jm}{R}\right) \biggr]
 \end{array}
\right\} \nonumber \\
&=\frac{1}{4}~\sum_{m=0}^{R-1}
\left\{
\begin{array}{l}
\biggl[ \exp\left(iz + 2\pi i \frac{km}{R}\right) + \exp\left(-iz - 2\pi i \frac{km}{R}\right)\biggr] \\
\cdot  \biggl[ \exp\left(iy + 2\pi i \frac{jm}{R}\right) + \exp\left(-iy - 2\pi i \frac{jm}{R}\right) \biggr]
 \end{array}
\right\} \nonumber \\
&= \frac{1}{4}\sum_{m=0}^{R-1} \exp\left(iz + 2\pi i \frac{km}{R}\right)\exp\left(iy + 2\pi i \frac{jm}{R}\right) \label{11} \\
	&\quad+\frac{1}{4}\sum_{m=0}^{R-1} \exp\left(-iz - 2\pi i \frac{km}{R}\right)\exp\left(iy + 2\pi i \frac{jm}{R}\right)\label{22} \\
	&\quad+\frac{1}{4}\sum_{m=0}^{R-1} \exp\left(iz + 2\pi i \frac{km}{R}\right)\exp\left(-iy - 2\pi i \frac{jm}{R}\right) \label{33}\\
	&\quad+\frac{1}{4}\sum_{m=0}^{R-1} \exp\left(-iz - 2\pi i \frac{km}{R}\right)\exp\left(-iy + 2\pi i \frac{jm}{R}\right). \label{44} 
\end{align}
We now use the fact that if $\xi\neq 1$ is any $R^{th}$ root of unity, then 
  \begin{align*}
    \sum_{m=0}^{R-1} \xi^m =0.
  \end{align*}
  Since $K< R/2$, $k+j<R$ for all $0 \leq k,j \leq K$, and therefore 
  \begin{align*}
    \begin{split}
      \eqref{11} &= 
      \begin{cases}
	\frac{R}{4}\exp(iz+iy), \quad \text{if $k=j=0$}\\
	0, \quad \text{otherwise}
      \end{cases}, \\
      \eqref{22}=\eqref{33} &=
      \begin{cases}
	\frac{R}{4}, \quad\text{if $k=j$} \\
	0, \quad \text{otherwise}
      \end{cases},\\
      \hbox{and } \eqref{44} &= 
      \begin{cases}
	\frac{R}{4}\exp(-iz-iy), \quad \text{if $k=j=0$}\\
	0, \quad \text{otherwise}
      \end{cases}. 
  \end{split}
  \end{align*}
  Since for $k=j=0$ we have $y=z=0$, we see that
  \begin{align}\label{Eq:Akj}
    A_{k,j}=
    \begin{cases}
      R, \quad \text{if $k=j=0$}\\
      \frac{R}{2}, \quad \text{if $k=j\neq0$}\\
      0, \quad\text{otherwise}
    \end{cases}.
  \end{align}
Finally, substituting \eqref{Eq:Akj} back into \eqref{Eq:th_pr_1} yields
\begin{align*}
\sum_{m\in {\mathbb Z}} \left|q\left(t-\frac{m}{R}\right)\right|^2 = \sum_{k=0}^K \sum_{j=0}^K a_k a_j A_{k,j} = R a_0^2 + \frac{R}{2} \sum_{k=1}^K a_k^2.
\end{align*}
\end{IEEEproof}

\begin{IEEEproof}[Proof of Corollary \ref{Co:undilated_translates}]
Letting $\hat{h}(y)=q(y)$,
 it follows immediately from Theorem \ref{Th:uniform_translates} that 
\begin{align*}
\sum_{m \in {\mathbb Z}}\left[\hat{h}\left(y-\frac{m}{R}\right)\right]^2=
Ra_0^2+\frac{R}{2}\sum_{k=1}^K a_k^2,
~~\forall y\in \Rbb.
\end{align*}
Moreover, for  $y \in \left[0,\frac{M+1-R}{R}\right]$, $\hat{h}\left(y-\frac{m}{R}\right)=0$ if 
$m < 1-R$ 
or $m > M+1-R$.   
For $m=M+1-R$,
$\hat{h}\left(y-\frac{m}{R}\right)=0$ for all $y \in \left[0,\frac{M+1-R}{R}\right)$, and, due to \eqref{Eq:a_constraint}, for $y=\frac{M+1-R}{R}$, 
\begin{align*}
\hat{h}\left(y-\frac{m}{R}\right)=\hat{h}(0)=\sum_{k=0}^K (-1)^k a_k=0. 
\end{align*}
\end{IEEEproof}

\bibliography{bib_adapted}
\bibliographystyle{IEEETran}
\end{document}